\documentclass[12pt,english]{article}
\usepackage{mathptmx}

\usepackage[T1]{fontenc}
\usepackage[latin9]{inputenc}
\usepackage[a4paper]{geometry}
\usepackage{babel}
\usepackage{textcomp}
\usepackage{amsmath}
\usepackage{amsthm}
\usepackage{amssymb}
\usepackage{setspace}
\usepackage[numbers,square,sort,comma,numbers]{natbib}
\usepackage[all]{xy}
\usepackage[unicode=true,
 bookmarks=true,bookmarksnumbered=false,bookmarksopen=false,
 breaklinks=false,pdfborder={0 0 1},backref=false,colorlinks=false]
 {hyperref}
\hypersetup{pdftitle={Use Lyx for typesetting: A template for empirical research projects},
 pdfauthor={First Author and Second Author}}

\makeatletter
\hypersetup{
    colorlinks=true,       
    linkcolor=blue,          
    citecolor=blue,        
    filecolor=blue,      
    urlcolor=blue,           
    breaklinks=true
}

\usepackage{booktabs}
\usepackage{longtable}
\usepackage{pdflscape}
\usepackage[square,sort,comma,numbers]{natbib}
\usepackage{amsthm}
\usepackage[title]{appendix}

\makeatother

\begin{document}
\title{\onehalfspacing{}On the Burau representation of $B_{3}$ modulo $p$}
\author{Donsung Lee}
\date{June 20, 2024}

\maketitle
\medskip{}

\begin{abstract}
\begin{spacing}{0.9}
\noindent We present an algorithm that, given a prime $p$ as input,
determines whether or not the Burau representation of the 3-strand
braid group modulo $p$ is faithful. We also prove that the representation
is indeed faithful when $p\le13$. Additionally, we re-pose Salter's
question on the Burau representation of $B_{3}$ over finite fields
$\mathbb{F}_{p}$, and solve it for every $p$.\vspace{1cm}

\medskip{}

\end{spacing}
\end{abstract}
\noindent $ $\theoremstyle{definition}
\theoremstyle{remark}
\newtheorem{theorem}{Theorem}[section]
\newtheorem{lemma}[theorem]{Lemma}
\newtheorem{corollary}[theorem]{Corollary}
\newtheorem{question}[theorem]{Question}
\newtheorem{remark}[theorem]{Remark}
\newtheorem{example}[theorem]{Example}
\newtheorem{definition}[theorem]{Definition}
\newtheorem{algorithm}[theorem]{Algorithm}
\newtheorem{conjecture}[theorem]{Conjecture}

\newtheorem*{ack}{Acknowledgements}
\newtheorem*{theoremA}{Theorem \textnormal{A}}
\newtheorem*{question1}{Question \textnormal{1}}
\newtheorem*{question2}{Question \textnormal{2}}
\newtheorem*{claim1}{Claim \textnormal{1}}
\newtheorem*{proofclaim1}{Proof of Claim \textnormal{1}}
\newtheorem*{claim2}{Claim \textnormal{2}}
\newtheorem*{proofclaim2}{Proof of Claim \textnormal{2}}
\newtheorem*{proofofa}{Proof of \textnormal{(a)}}
\newtheorem*{proofofb}{Proof of \textnormal{(b)}}
\newtheorem*{cases1}{Cases \textnormal{(1)}}
\newtheorem*{cases2}{Cases \textnormal{(2)}}
\newtheorem*{prooflemma218}{Proof of Lemma \textnormal{2.18} for p=\textnormal{3}}
\newtheorem*{prooftheorem11}{Proof of Theorem \textnormal{1.1}}
\newtheorem*{prooftheorem12}{Proof of Theorem \textnormal{1.2}}
\newtheorem*{prooftheorem41}{Proof of Theorem \textnormal{4.1}}

\section{Introduction}

The \emph{Burau representation} is a fundamental linear representation
for braid groups $B_{n}$, which has the generators $\sigma_{1},\,\sigma_{2},\cdots,\,\sigma_{n-1}$.
For each integer $i$ such that $1\le i\le n-1$, the representation
$\beta_{n}$ is given explicitly by
\begin{center}
$\sigma_{i}\mapsto\left(\begin{array}{cccc}
I_{i-1} & 0 & 0 & 0\\
0 & 1-t & t & 0\\
0 & 1 & 0 & 0\\
0 & 0 & 0 & I_{n-i-1}
\end{array}\right)\in\mathrm{GL}\left(n,\,\mathbb{Z}\left[t,\,t^{-1}\right]\right)$.
\par\end{center}

The faithfulness of the Burau representation has been a subject of
interest for a long time. Although the faithfulness was proven for
the case $n=3$ at least in the late 1960s \citep{MR0222880}, it
was in the 1990s that Moody \citep{MR1158006}, Long\textendash Paton
\citep{MR1217079}, and Bigelow \citep{MR1725480} proved the unfaithfulness
for all $n\ge5$. Meanwhile, we may define another representation
for $B_{n}$ to $\mathrm{GL}\left(n,\,\mathbb{F}_{p}\left[t,\,t^{-1}\right]\right)$
by taking modulo $p$ on the image $\beta_{n}\left(B_{n}\right)$.
Call this representation the \emph{Burau representation modulo} $p$,
and denote it by $\beta_{n,\,p}$. There have been research on $\beta_{4,\,p}$
as well, related to the study of the kernel of the original Burau
representation $\beta_{4}$, the triviality of which remains unknown.
For example, Cooper\textendash Long \citep{MR1431138} showed that
$\beta_{4,\,2}$ is unfaithful, Cooper\textendash Long \citep{MR1668343}
proved that $\beta_{4,\,3}$ is unfaithful, and Gibson\textendash Williamson\textendash Yacobi
\citep{gibson20234strand} proved that $\beta_{4,\,5}$ is unfaithful.
Witzel\textendash Zaremsky \citep{MR3415244} implicitly dealt with
$\beta_{4,\,p}$, proving that two matrices in the Burau image generate
a free group. Lee\textendash Song \citep{MR2175118} showed that every
nontrivial element in $\ker\beta_{4,\,p}$ is pseudo-Anosov.

However, there has been relatively little attention given to $\beta_{3,\,p}$.
As Cooper\textendash Long \citep{MR1431138} briefly stated, the methods
used by Moody \citep{MR1158006} and Long\textendash Paton \citep{MR1217079}
were applicable to prove that $\beta_{3,\,2}$ is faithful. In this
paper, we prove that the faithfulness problem of $\beta_{3,\,p}$
is decidable for every $p$. In other words, we provide an algorithm
to determine whether $\beta_{3,\,p}$ is faithful.

\noindent \begin{theorem}

For each prime $p$, the faithfulness problem of $\beta_{3,\,p}$
is decidable.

\noindent \end{theorem}

It is well-known that the quotient group of $B_{3}$ by its center
is isomorphic to the modular group $\mathrm{PSL}\left(2,\,\mathbb{Z}\right)$
\citep{MR2435235}. The first step to approach the faithfulness problem
is to reinterpret the Burau representation $\beta_{3}$ through a
map
\begin{align*}
\mu & :\mathrm{PSL}\left(2,\,\mathbb{Z}\right)\to\mathrm{PGL}\left(2,\,\mathbb{Z}\left[t,\,t^{-1},\,\left(1+t\right)^{-1}\right]\right),
\end{align*}
defined by
\begin{align*}
\mu\left[\begin{array}{cc}
1 & 0\\
1 & 1
\end{array}\right] & :=\left[\begin{array}{cc}
1 & 0\\
0 & -t
\end{array}\right],\;\mu\left[\begin{array}{cc}
1 & -1\\
0 & 1
\end{array}\right]:=\left[\begin{array}{cc}
-\frac{t^{2}}{1+t} & \frac{t}{1+t}\\
\frac{1+t+t^{2}}{1+t} & \frac{1}{1+t}
\end{array}\right],
\end{align*}
where we used the convention that a matrix in the square brakets $\left[\right]$
is included in a projective linear group, while a usual matrix is
in the round brakets $\left(\right)$. This map is injective by the
faithfulness of $\beta_{3}$. Denote by $m_{p}$ the modulo $p$ map
\begin{align*}
m_{p}:\mathrm{PGL}\left(2,\,\mathbb{Z}\left[t,\,t^{-1},\,\left(1+t\right)^{-1}\right]\right) & \to\mathrm{PGL}\left(2,\,\mathbb{F}_{p}\left[t,\,t^{-1},\,\left(1+t\right)^{-1}\right]\right).
\end{align*}

Define $\mu_{p}$ to be $m_{p}\circ\,\mu$. Then, one of our key lemmas
states that $\beta_{3,\,p}$ is faithful if and only if $\mu_{p}$
is injective. Moreover, the image of $\mu_{p}$ satisfies a kind of
\emph{unitarity}, as found by Squier \citep{MR0727232}. If we put
\begin{align*}
J_{n} & :=\left(\begin{array}{ccccc}
1 & -t^{-1} & -t^{-1} & \cdots & -t^{-1}\\
-t & 1 & -t^{-1} & \cdots & -t^{-1}\\
-t & -t & 1 & \cdots & -t^{-1}\\
\vdots & \vdots & \vdots & \ddots & \vdots\\
-t & -t & -t & \cdots & 1
\end{array}\right),
\end{align*}
then any element $A\in\beta_{n}\left(B_{n}\right)\subset\mathrm{GL}\left(n,\,\mathbb{Z}\left[t,\,t^{-1}\right]\right)$
satisfies the \emph{unitarity relation}
\begin{align}
\overline{A}\,J_{n}\,A^{T} & =J_{n},
\end{align}
where for a matrix $A$ with Laurent polynomial entries, $A\mapsto\overline{A}$
maps $t\mapsto t^{-1}$ for every entry. Define further for any Laurent
polynomial $f\left(t\right)\in\mathbb{Z}\left[t,\,t^{-1}\right]$,
$\overline{f\left(t\right)}:=f\left(t^{-1}\right)$. The image of
$\mu_{p}$ satisfies a {*}-diagonalized version of unitarity inherited
from (1).

The image of an element $M\in\mathrm{PSL}\left(2,\,\mathbb{Z}\right)$
under $\mu_{p}$ is included in $\mathrm{PGL}\left(2,\,\mathbb{F}_{p}\left[t,\,t^{-1}\right]\right)$
if and only if $M\in\Gamma^{0}\left(p\right)$, or $M$ is included
in the transpose of the Hecke congruence subgroup $\Gamma_{0}\left(p\right)$
(Lemma 2.13). Our main preliminary task is to establish a structure
theorem (${\mathrm{Theorem}\;3.11}$) of the \emph{quaternionic group}
$\mathcal{Q}_{p}$ (Definition 2.14), closely related to the unitary
group over the Laurent polynomial ring. To prove Theorem 3.11, we
will use Fenchel\textendash Nielsen's theorem, which is a variant
of the ping-pong lemma for HNN extensions \citep{MR1958350,MR0959135,MR4735233}.
On the other hand, Rademacher \citep{MR3069525} described the structure
and a generating set of $\Gamma_{0}\left(p\right)$. Our proof of
decidability relies on Rademacher's classical results, Bass\textendash Serre
theory \citep{MR1954121,MR0263928,MR1239551} and Stallings' folding
process \citep{MR0695906,MR2130178,MR1882114}. At the end of Section
5, we will prove that $\beta_{3,\,p}$ is faithful for $p\le13$ (Theorem
5.18).

In 2021, Salter in \citep{MR4228497} identified several algebraic
properties satisfied by the elements in the image of $\beta_{n}$
and defined the \emph{target group} $\Gamma_{n}$ as a subgroup of
the unitary group satisfying the properties. Denote by $\overline{\Gamma_{n}}$
the quotient of $\Gamma_{n}$ by its center via the quotient map $q_{n}:\Gamma_{n}\to\overline{\Gamma_{n}}$.
He then posed the question of whether the composition $q_{n}\circ\beta_{n}:B_{n}\to\overline{\Gamma_{n}}$
is surjective. We negatively solved Salter's problem for the case
$n=3$ in \citep{lee2024salters}. For every prime $p$, we can pose
a problem analogous to Salter's, asking whether the Burau representation
modulo $p$ surjects onto the central quotient $\overline{\Gamma_{n,\,p}}=q_{n,\,p}\left(\Gamma_{n,\,p}\right)$
of the target group $\Gamma_{n,\,p}$ defined modulo $p$. Let's term
this a \emph{Salter-type problem modulo} $p$. We will solve this
for every prime $p$ in the case $n=3$. This result also follows
from our structure theorem (Theorem 3.11).

\noindent \begin{theorem}

The map $q_{3,\,2}\circ\beta_{3,\,2}:B_{3}\to\overline{\Gamma_{3,\,2}}$
is surjective. However, for any $p>2$, the map $q_{3,\,p}\circ\beta_{3,\,p}$
is not surjective. Moreover, we have $\left[\overline{\Gamma_{3,\,p}}\,:\,q_{3,\,p}\circ\beta_{3,\,p}\left(B_{3}\right)\right]=\infty$.

\noindent \end{theorem}

The rest of this paper is organized as follows. In Section 2, we prove
basic properties of the Burau representation modulo $p$ and the quaternionic
group. In Section 3, we list the generating set of the quaternionic
group, and state the structure theorem (Theorem 3.11). We prove Theorem
3.11 when $p\not\equiv1$ (mod 6). In Section 4, we prove Theorem
3.11 when $p\equiv1$ (mod 6). In Section 5, we prove Theorem 1.1.
In Section 6, we prove Theorem 1.2.

\noindent \begin{ack}

\noindent This paper is supported by \emph{National Research Foundation
of Korea (grant number 2020R1C1C1A01006819).} The author thanks Dohyeong
Kim (SNU), Sang-hyun Kim (KIAS), and Carl-Fredrik Nyberg Brodda (KIAS)
for valuable and encouraging comments.

\noindent \end{ack}

\section{Preliminary Results}

Here we review the unitarity of the Burau representation, define quaternionic
groups over finite fields, and explore their basic properties. Many
aspects covered here overlap with our previous work \citep{lee2024salters}.
Therefore, we simplify or omit several proofs to avoid redundancy.
Towards the end of this section, our focus will shift to phenomena
that occur because we are working with the base ring $\mathbb{F}_{p}$,
rather than $\mathbb{Q}$ or $\mathbb{Z}$.

Consider a matrix $A\in\beta_{3}\left(B_{3}\right)$. By definition,
$A$ preserves two vectors as
\begin{align*}
A\left(\begin{array}{c}
1\\
1\\
1
\end{array}\right) & =\left(\begin{array}{c}
1\\
1\\
1
\end{array}\right),\;\left(\begin{array}{c}
t\\
t^{2}\\
t^{3}
\end{array}\right)^{T}A=\left(\begin{array}{c}
t\\
t^{2}\\
t^{3}
\end{array}\right)^{T}.
\end{align*}

These properties and the unitarity (1) still hold modulo $p$, and
we use the same symbols abusing notation. Unless we specially mention
it, we suppose that the prime $p$ is generic.

\noindent \begin{definition}

The \emph{formal Burau group} $\mathcal{B}_{p}$ is defined as
\begin{align*}
\mathcal{B}_{p} & :=\left\{ A\in\mathrm{GL}\left(3,\,\mathbb{F}_{p}\left[t,\,t^{-1}\right]\right)\::\:vA=v,\:A\overrightarrow{1}=\overrightarrow{1},\:\overline{A}\,J_{3}\,A^{T}=J_{3}\right\} .
\end{align*}
\end{definition}

We express the entries of $A\in\mathcal{B}_{p}$ as
\begin{align}
A & =\left(\begin{array}{ccc}
A_{11} & A_{12} & 1-A_{11}-A_{12}\\
A_{21} & A_{22} & 1-A_{21}-A_{22}\\
t^{-2}\left(1-A_{11}-tA_{21}\right) & t^{-2}\left(t-A_{12}-tA_{22}\right) & *
\end{array}\right),
\end{align}
where $*=t^{-2}\left(-1-t+t^{2}+A_{11}+A_{12}+tA_{21}+tA_{22}\right).$

Define matrices
\begin{align*}
 & \xi:=\left(\begin{array}{ccc}
t+t^{2} & 0 & -1\\
-1 & t & -1\\
-1 & -1 & -1
\end{array}\right),\\
 & D:=\left(\overline{\xi}\right)^{-1}J_{3}\left(\xi^{T}\right)^{-1}.
\end{align*}

By direct computation, we have a diagonal matrix
\begin{align*}
D & =\left(\begin{array}{ccc}
\frac{t}{1+t+t^{2}} & 0 & 0\\
0 & 1 & 0\\
0 & 0 & -\frac{t}{1+t+t^{2}}
\end{array}\right).
\end{align*}

To apply the unitarity (1), for integers $1\le k,\,l\le2$, define
a Laurent polynomial
\begin{align*}
f_{kl} & :=\left(A\xi\right)_{kl}.
\end{align*}

Then, from the definition of $\xi$, we have
\begin{align}
f_{k1} & =A_{k1}\left(1+t+t^{2}\right)-1,\;f_{k2}=A_{k1}+A_{k2}\left(1+t\right)-1,\;k=1,2.
\end{align}

\noindent \begin{lemma}

For each matrix $A\in\mathcal{B}_{p}$, we have the following properties.
\begin{description}
\item [{(a)}] the determinant is expressed by $\left\{ f_{kl}\right\} _{1\le k,\,l\le2}$
as
\begin{align}
\det\left(A\right) & =\frac{f_{11}f_{22}-f_{12}f_{21}}{t^{2}\left(1+t\right)}.
\end{align}
\item [{(b)}] the Laurent polynomials $\left\{ f_{kl}\right\} _{1\le k,\,l\le2}$
satisfies the three functional equations
\begin{align}
 & f_{k1}\overline{f_{k1}}+\left(1+t+t^{-1}\right)f_{k2}\overline{f_{k2}}=\frac{\left(1+t\right)^{2}}{t},\;k=1,2,\\
 & 1+t+t^{2}\left(\overline{f_{11}}f_{21}+\left(1+t+t^{-1}\right)\overline{f_{12}}f_{22}\right)=0.
\end{align}
\end{description}
\noindent \end{lemma}
\begin{proof}

\noindent Rewrite the unitarity satisfied by $A\in\mathcal{B}_{p}$
as
\begin{align*}
\overline{A\xi}D\left(A\xi\right)^{T} & =J_{3}.
\end{align*}

The properties are direct from computations using (2) and (3). $\qedhere$

\noindent \end{proof}

By substitution using $f_{ij}$, we write the entries of a matrix
$A\in\mathcal{B}_{p}$ from (2) as
\begin{align}
A & =\left(\begin{array}{ccc}
\frac{1+f_{11}}{1+t+t^{2}} & \frac{-f_{11}+t\left(1+t\right)+f_{12}\left(1+t+t^{2}\right)}{\left(1+t\right)\left(1+t+t^{2}\right)} & *\\
\frac{1+f_{21}}{1+t+t^{2}} & \frac{-f_{21}+t\left(1+t\right)+f_{22}\left(1+t+t^{2}\right)}{\left(1+t\right)\left(1+t+t^{2}\right)} & *\\
* & * & *
\end{array}\right).
\end{align}

For integers $1\le i,j\le2$, we introduce a new substitution
\begin{align}
g_{ij} & :=\frac{f_{ij}}{t\left(1+t\right)},
\end{align}

Then, the entries of $A\in\beta_{3}\left(B_{3}\right)$ are also written
as
\begin{align}
A & =\left(\begin{array}{ccc}
\frac{1+g_{11}t\left(1+t\right)}{1+t+t^{2}} & \frac{t\left(1-g_{11}\right)}{\left(1+t+t^{2}\right)}+tg_{12} & *\\
\frac{1+g_{21}t\left(1+t\right)}{1+t+t^{2}} & \frac{t\left(1-g_{21}\right)}{\left(1+t+t^{2}\right)}+tg_{22} & *\\
* & * & *
\end{array}\right).
\end{align}

\noindent \begin{definition}

Define the \emph{first homomorphism} $\phi_{p}:\mathcal{B}_{p}\to\mathrm{GL}\left(2,\,\mathbb{F}_{p}\left[t,\,t^{-1},\,\left(1+t\right)^{-1}\right]\right)$
by
\begin{align*}
A & \mapsto\left(\begin{array}{cc}
g_{11} & g_{12}\\
t^{-1}g_{11}+\left(1+t\right)g_{21} & t^{-1}g_{12}+\left(1+t\right)g_{22}
\end{array}\right).
\end{align*}
\end{definition}
\begin{lemma}

The map $\phi$ is an injective group homomorphism preserving the
determinant.

\noindent \end{lemma}
\begin{proof}See the proof of \citep[Lemma 3.4]{lee2024salters}. $\qedhere$

\noindent \end{proof}

The functional equations (5) and (6) imply that $f_{21}$ and $f_{22}$
are expressed in terms of $f_{11}$, $f_{12}$, and the determinant.

\noindent \begin{lemma}

For a matrix $A\in\mathcal{B}_{p}$, when we write the entries of
$A$ in terms of $\left\{ f_{kl}\right\} _{1\le k,\,l\le2}$ as in
(3), we have
\begin{align}
 & f_{21}=-\left(\frac{f_{11}+\overline{f_{12}}t^{3}\left(1+t+t^{2}\right)\det\left(A\right)}{t\left(1+t\right)}\right),\\
 & f_{22}=\frac{-f_{12}+\overline{f_{11}}t^{4}\det\left(A\right)}{t\left(1+t\right)}.
\end{align}
\end{lemma}
\begin{proof}See the proof of \citep[Lemma 3.5]{lee2024salters}. $\qedhere$

\noindent \end{proof}
\begin{corollary}

The first homomorphism $\phi$ on $\mathcal{B}$ maps $A\in\mathcal{B}$
into
\begin{align*}
\left(\begin{array}{cc}
g_{11} & g_{12}\\
-\det\left(A\right)\left(1+t^{-1}+t\right)\overline{g_{12}} & \det\left(A\right)\overline{g_{11}}
\end{array}\right).
\end{align*}
\end{corollary}
\begin{proof}

\noindent In terms of $\left\{ g_{kl}\right\} _{1\le k,\,l\le2}$,
by (8), the equations (10) and (11) in Lemma 2.5 are written as
\begin{align*}
 & g_{21}=-\left(\frac{g_{11}+\overline{g_{12}}\left(1+t+t^{2}\right)\det\left(A\right)}{t\left(1+t\right)}\right),\\
 & g_{22}=\frac{-g_{12}+\overline{g_{11}}t\det\left(A\right)}{t\left(1+t\right)}.
\end{align*}
which simplifies the image of the map $\phi$ as desired. $\qedhere$

\noindent \end{proof}

For a matrix $A\in\mathcal{B}_{p}$, we can evaluate the entries at
$t=-1$. Abusing notation slightly, we define the evaluation map:
\begin{align*}
\mathrm{eval}_{-1} & :\mathcal{B}_{p}\to\mathrm{GL}\left(3,\,\mathbb{F}_{p}\right).
\end{align*}

For a matrix $B\in\mathrm{eval}_{-1}\left(\mathcal{B}\right)$, we
write the entries in terms of the four corner entries using (2) as
\begin{align*}
B & =\left(\begin{array}{ccc}
B_{11} & 1-B_{11}-B_{13} & B_{13}\\
-1+B_{11}+B_{31} & 3-B_{11}-B_{13}-B_{31}-B_{33} & -1+B_{13}+B_{33}\\
B_{31} & 1-B_{31}-B_{33} & B_{33}
\end{array}\right).
\end{align*}

\noindent \begin{definition}

Define the \emph{second homomorphism} $\rho_{p}:\mathrm{eval}_{-1}\left(\mathcal{B}_{p}\right)\to\mathrm{GL}\left(2,\,\mathbb{F}_{p}\right)$
by
\begin{align*}
B & \mapsto\left(\begin{array}{cc}
1-B_{13} & 1-B_{11}\\
1-B_{33} & 1-B_{31}
\end{array}\right).
\end{align*}
\end{definition}
\begin{lemma}

The map $\rho$ is an injective group homomorphism preserving the
determinant.

\noindent \end{lemma}
\begin{proof}

\noindent See the proof of \citep[Lemma 2.8]{lee2024salters}. $\qedhere$

\noindent \end{proof}

We compute the images of $\beta_{3}\left(\sigma_{1}\right)=\left(\begin{array}{ccc}
1-t & t & 0\\
1 & 0 & 0\\
0 & 0 & 1
\end{array}\right)$ and $\beta_{3}\left(\sigma_{2}\right)=\left(\begin{array}{ccc}
1 & 0 & 0\\
0 & 1-t & t\\
0 & 1 & 0
\end{array}\right)$ under the two homomorphisms $\phi_{p}$ and $\rho_{p}$. They are
given by
\begin{align}
 & \phi_{p}\left(\beta_{3,\,p}\left(\sigma_{1}\right)\right)=\left(\begin{array}{cc}
-\frac{t^{2}}{1+t} & \frac{t}{1+t}\\
\frac{1+t+t^{2}}{1+t} & \frac{1}{1+t}
\end{array}\right),\;\phi_{p}\left(\beta_{3,\,p}\left(\sigma_{2}\right)\right)=\left(\begin{array}{cc}
1 & 0\\
0 & -t
\end{array}\right),\\
 & \rho_{p}\left(\beta_{3,\,p}\left(\sigma_{1}\right)|_{t=-1}\right)=\left(\begin{array}{cc}
1 & -1\\
0 & 1
\end{array}\right),\;\rho_{p}\left(\beta_{3,\,p}\left(\sigma_{2}\right)|_{t=-1}\right)=\left(\begin{array}{cc}
1 & 0\\
1 & 1
\end{array}\right).
\end{align}

Put $\pi_{p}$ as the natural projection
\begin{align*}
\pi_{p} & :\mathrm{GL}\left(2,\,\mathbb{F}_{p}\left(t\right)\right)\to\mathrm{PGL}\left(2,\,\mathbb{F}_{p}\left(t\right)\right),
\end{align*}
where $\ker\pi_{p}$ includes the scalar matrices $uI$, for units
$u$ in $\mathbb{F}_{p}\left(t\right)$. Denote by $\overline{\phi_{p}}$
(resp. $\overline{\rho_{p}}$) the composition $\pi_{p}\circ\phi_{p}$
(resp. $\pi_{p}\circ\rho_{p}\circ\mathrm{eval}_{-1}$). 

Because both $\beta_{3,\,p}\left(\sigma_{1}\right)$ and $\beta_{3,\,p}\left(\sigma_{2}\right)$
have determinant $-t$, the restriction $\overline{\rho_{p}}$ on
$\beta_{3,\,p}\left(B_{3}\right)$ maps into $\mathrm{PSL}\left(2,\,\mathbb{F}_{p}\right)$.
In fact, the image is the entire group $\mathrm{PSL}\left(2,\,\mathbb{F}_{p}\right)$
by the strong approximation theorem for $\mathrm{PSL}\left(2,\,\mathbb{Z}\right)$
\citep{MR1278263}. The following lemma guarantees that the image
of $\mathcal{B}_{p}$ under $\overline{\rho_{p}}$ is also $\mathrm{PSL}\left(2,\,\mathbb{F}_{p}\right)$.

\noindent \begin{lemma}

For a matrix $A\in\mathcal{B}_{p}$, we have $\det\left(A\right)=\left(-t\right)^{k}$
for some integer $k$. In particular, we have $\overline{\rho_{p}}\left(\mathcal{B}_{p}\right)=\mathrm{PSL}\left(2,\,\mathbb{F}_{p}\right)$.

\noindent \end{lemma}
\begin{proof}See the proof of \citep[Lemma 3.10]{lee2024salters}. The strong approximation
on $\mathrm{PSL}\left(2,\,\mathbb{Z}\right)$ ensures that $\overline{\rho_{p}}$
is surjective onto $\mathrm{PSL}\left(2,\,\mathbb{F}_{p}\right)$.
$\qedhere$

\noindent \end{proof}

The following lemma characterizes the image of $\phi_{p}$.

\noindent \begin{lemma}

Suppose a matrix $M=\left(\begin{array}{cc}
g_{11} & g_{12}\\
-\det\left(M\right)\left(1+t^{-1}+t\right)\overline{g_{12}} & \det\left(M\right)\overline{g_{11}}
\end{array}\right)$ which included in $\mathrm{GL}\left(2,\,\mathbb{F}_{p}\left[t,\,t^{-1}\right]\right)$
satisfies that $\det\left(M\right)=\left(-t\right)^{k}$ for some
integer $k$. Then, the image group $\phi_{p}\left(\mathcal{B}_{p}\right)$
includes $M$ if and only if $g_{11}=1$ (mod $1+t+t^{2}$).

\noindent \end{lemma}
\begin{proof}

\noindent See the proof of \citep[Lemma 3.11]{lee2024salters}. $\qedhere$

\noindent \end{proof}

Define a braid $\Delta\in B_{3}$ to be $\left(\sigma_{1}\sigma_{2}\sigma_{1}\right)^{2}$.
It is a generator of the center of $B_{3}$ \citep{MR2435235}.

\noindent \begin{corollary}

A matrix $A\in\mathcal{B}_{p}$ is contained in $\ker\overline{\phi_{p}}$
if and only if $A=\beta_{3,\,p}\left(\Delta\right)^{k}$ for some
$k\in\mathbb{Z}$.

\noindent \end{corollary}
\begin{proof}

\noindent This proof is a reduplication of the proof of \citep[Corollary 3.12]{lee2024salters}.
The Burau image of $\Delta$ is given by
\begin{align*}
\beta_{3,\,p}\left(\Delta\right) & =\left(\begin{array}{ccc}
1-t+t^{3} & t-t^{2} & t^{2}-t^{3}\\
1-t & t-t^{2}+t^{3} & t^{2}-t^{3}\\
1-t & t-t^{2} & t^{2}
\end{array}\right),
\end{align*}
by which we compute
\begin{align*}
\phi_{p}\left(\beta_{3,\,p}\left(\Delta\right)\right) & =\left(\begin{array}{cc}
t^{3} & 0\\
0 & t^{3}
\end{array}\right),
\end{align*}
which is included in $\ker\overline{\phi_{p}}$.

On the other hand, choose a matrix $A\in\ker\overline{\phi_{p}}$,
which is equivalent to the condition that $\phi\left(A\right)=uI$
for some unit $u=\pm t^{l_{1}}\left(1+t\right)^{l_{2}}$ in $\mathbb{F}_{p}\left[t,\,t^{-1},\,\left(1+t\right)^{-1}\right]$,
where $l_{1}$ and $l_{2}$ are integers. By Corollary 2.6, we have
\begin{align}
\phi\left(A\right) & =\left(\begin{array}{cc}
g_{11} & 0\\
0 & \det\left(A\right)\overline{g_{11}}
\end{array}\right).
\end{align}

Since the determinant is preserved under $\phi$ by Lemma 2.4, by
equating $\det\phi\left(A\right)$ with $\det\left(A\right)$ in (14),
we have a functional equation
\begin{align*}
g_{11}\overline{g_{11}} & =1,
\end{align*}
which has only solutions $g_{11}=\pm t^{l}$ for some integer $l$
in $g_{11}\in\mathbb{F}_{p}\left[t,\,t^{-1},\,\left(1+t\right)^{-1}\right]$.
From Lemma 2.10, we need to require that
\begin{align*}
t^{l}=1 & \;(\mathrm{mod}\:1+t+t^{2}),
\end{align*}
which implies $l=3k$ for some integer $k$. $\qedhere$

\noindent \end{proof}
\begin{corollary}

The center of the formal Burau group $\mathcal{B}_{p}$ is generated
by $\beta_{3,\,p}\left(\Delta\right)$.

\noindent \end{corollary}
\begin{proof}See the proof of \citep[Corollary 3.13]{lee2024salters}. $\qedhere$

\noindent \end{proof}

The following lemma provides a characterization of when the image
$\phi_{p}\left(A\right)$ of an element $A\in\mathcal{B}_{p}$ has
Laurent polynomial entries.

\noindent \begin{lemma}

For a matrix $A\in\mathcal{B}_{p}$, the matrix $\phi_{p}\left(A\right)$
is contained in $\mathrm{GL}\left(2,\,\mathbb{F}_{p}\left[t,\,t^{-1}\right]\right)$
if and only if $\rho_{p}\left(A|_{t=-1}\right)_{12}=0$.

\noindent \end{lemma}
\begin{proof}

\noindent For a matrix $A\in\mathcal{B}_{p}$, Corollary 2.6 ensures
that $\phi_{p}\left(A\right)\in\mathrm{GL}\left(2,\,\mathbb{F}_{p}\left[t,\,t^{-1}\right]\right)$
if and only if $g_{11}$ and $g_{12}$ are Laurent polynomials. By
definition, $g_{11}$ and $g_{12}$ are Laurent polynomials if and
only if $f_{11}|_{t=-1}=0=f_{12}|_{t=-1}$. From (3), we see the last
is equivalent to $A_{11}|_{t=-1}=1$, which is in turn equivalent
to $\rho_{p}\left(A|_{t=-1}\right)_{12}=0$ by Definition 2.7. $\qedhere$

\noindent \end{proof}

By Lemma 2.9 and Lemma 2.13, we see that the formal Burau group $\mathcal{B}_{p}$
\emph{virtually} has its image in $\mathrm{PGL}\left(2,\,\mathbb{Z}\left[t,\,t^{-1}\right]\right)$
under $\overline{\phi_{p}}$. This observation will be further elaborated
in Section 5. For now, we analyze the elements of $\overline{\phi_{p}}\left(\mathcal{B}\right)\cap\mathrm{PGL}\left(2,\,\mathbb{F}_{p}\left[t,\,t^{-1}\right]\right)$.
For the sake of abbreviation, from now on, denote by $\Phi$ the Laurent
polynomial $1+t^{-1}+t$.

\noindent \begin{definition}

Define the \emph{quaternionic group} as
\begin{align*}
\mathcal{Q}_{p}\, & :=\,\left\{ M\in\mathrm{PGL}\left(2,\,\mathbb{F}_{p}\left[t,\,t^{-1}\right]\right)\::\:M=\left[\begin{array}{cc}
g_{1} & g_{2}\\
-\Phi\overline{g_{2}} & \overline{g_{1}}
\end{array}\right],\,\mathrm{where}\;g_{1},\,g_{2}\in\mathbb{F}_{p}\left[t,\,t^{-1}\right]\right\} .
\end{align*}

In addition, define the $\mathcal{Q}_{p}^{1}$ to be the subgroup
of $\mathcal{Q}_{p}$ of elements having a representative $B$ such
that $\det\left(B\right)=1$. We also define two groups by
\begin{align*}
 & I_{p}:=\overline{\phi_{p}}\left(\mathcal{B}_{p}\right)\cap\mathrm{PGL}\left(2,\,\mathbb{F}_{p}\left[t,\,t^{-1}\right]\right),\\
 & I_{p}^{1}:=I_{p}\cap\mathcal{Q}_{p}^{1},
\end{align*}
where we call $I_{p}$ the \emph{Laurent image subgroup}.

\noindent \end{definition}

For each element $M\in\mathcal{Q}_{p}$, by definition, there exists
a representative $B=\left(\begin{array}{cc}
g_{1} & g_{2}\\
-\Phi\overline{g_{2}} & \overline{g_{1}}
\end{array}\right)$ in $\mathrm{GL}\left(2,\,\mathbb{F}_{p}\left[t,\,t^{-1}\right]\right)$.
Since $\overline{\Phi}=\Phi$, we observe that
\begin{align*}
\det\left(B\right) & =\overline{\det\left(B\right)}.
\end{align*}

Since the only palindromic units in $\mathbb{F}_{p}\left[t,\,t^{-1}\right]$
are units in $\mathbb{F}_{p}$, we have a functional equation for
elements of $\mathcal{Q}_{p}$
\begin{align}
g_{1}\overline{g_{1}}+\left(1+t^{-1}+t\right)g_{2}\overline{g_{2}} & =k,
\end{align}
where $k$ is included in $\mathbb{F}_{p}^{\times}$. The number $k$
depends on the element $M$ up to multiplication by an element in
$\left(\mathbb{F}_{p}^{\times}\right)^{2}$.

For a rational number $r\in\mathbb{F}_{p}$, note that
\begin{align*}
 & \,\det\left(\begin{array}{cc}
t-r^{2} & r\\
-r\left(1+t^{-1}+t\right) & t^{-1}-r^{2}
\end{array}\right)\\
= & \,1-r^{2}\left(t^{-1}+t\right)+r^{4}+r^{2}\left(1+t^{-1}+t\right)\\
= & \,1+r^{2}+r^{4},
\end{align*}
which is always included in $\mathbb{F}_{p}$. Unlike the case that
the base field is $\mathbb{Q}$, the value $1+r^{2}+r^{4}$ may be
zero in $\mathbb{F}_{p}$.

\noindent \begin{definition}

For a number $r\in\mathbb{F}_{p}$, define a matrix $g_{p}\left[r\right]$
to be
\begin{equation}
g_{p}\left[r\right]\,:=\,\left(\begin{array}{cc}
t-r^{2} & r\\
-r\left(1+t^{-1}+t\right) & t^{-1}-r
\end{array}\right).
\end{equation}

When $r$ satisfies that $1+r^{2}+r^{4}\ne0$, we define $\overline{g_{p}}\left[r\right]$
to be its projectivization $\pi_{p}\left(g_{p}\left[r\right]\right)$,
and call this the \emph{elementary generator}.

\noindent \end{definition}

Since the determinant of the matrix $g_{p}\left[r\right]$ is $1+r^{2}+r^{4}$,
the group $\mathcal{Q}_{p}$ includes $\overline{g_{p}}\left[r\right]$
for any $r\in\mathbb{F}_{p}$ such that $1+r^{2}+r^{4}\ne0$. When
we examine the determinant 1 subgroup $\mathcal{Q}_{p}^{1}$ of $\mathcal{Q}_{p}$
or $I_{p}^{1}$ of $I_{p}$, we consider whether an element has a
representative whose determinant is a square. In the case of $\overline{g_{p}}\left[r\right]$,
we need to find a solution of the equation
\begin{align}
y^{2} & =x^{4}+x^{2}+1.
\end{align}

Whenever $p>3$, the equation (17) defines an elliptic curve as a
Jacobi quartic, since the discriminant is 144. We call this curve
the \emph{determinant curve} $E_{d}$. Hasse's theorem on elliptic
curves \citep[p. 138]{MR2514094} gives bounds for the number of rational
points over a finite field.

\noindent \begin{theorem}

Let $E$ be an elliptic curve over a finite field $\mathbb{F}_{q}$.
The number of points $\#E\left(\mathbb{F}_{q}\right)$ satisfies the
following inequality:
\begin{align*}
q+1-2\sqrt{q} & \le\#E\left(\mathbb{F}_{q}\right)\le q+1+2\sqrt{q}.
\end{align*}
\end{theorem}

Unlike the case where the base ring is $\mathbb{Z}$ \citep[Lemma 4.5]{lee2024salters},
the group $I_{p}^{1}$ is generally not the same as $\mathcal{Q}_{p}^{1}$.
Nonetheless, the index is finite.

\noindent \begin{theorem}

For the primes $p\le7$, we have $\mathcal{Q}_{p}^{1}=I_{p}^{1}$.
For each prime $p>7$, we have an inequality:
\begin{align}
2\le\left[\mathcal{Q}_{p}^{1}\::\:I_{p}^{1}\right] & \le\frac{p-\left(\frac{-3}{p}\right)}{6},
\end{align}
where $\left(\frac{\cdot}{p}\right)$ is the usual Legendre symbol.

\noindent \end{theorem}
\begin{proof}

\noindent We introduce a series of temporary definitions. For a Laurent
polynomial $f$ in $\mathbb{F}_{p}\left[t,\,t^{-1}\right]$, define
$f^{e}$ to be its modulo $1+t+t^{2}$ quotient image in the ring
$\frac{\mathbb{F}_{p}\left[t,\,t^{-1}\right]}{\left(1+t+t^{2}\right)}$.
Define
\begin{align*}
Q_{p}^{1} & :=\pi_{p}^{-1}\left(\mathcal{Q}_{p}^{1}\right)\cap\mathrm{SL}\left(2,\,\mathbb{F}_{p}\left[t,\,t^{-1}\right]\right),
\end{align*}
and define a group homomorphism $\eta_{p}:Q_{p}^{1}\to\left(\frac{\mathbb{F}_{p}\left[t,\,t^{-1}\right]}{\left(1+t+t^{2}\right)}\right)^{\times}$
by $B\mapsto\left(B_{11}\right)^{e}$ as the composition of two maps:
\begin{align*}
\left(\begin{array}{cc}
g_{1} & g_{2}\\
-\Phi\overline{g_{2}} & \overline{g_{1}}
\end{array}\right) & \mapsto\left(\begin{array}{cc}
g_{1}^{e} & g_{2}^{e}\\
0 & \overline{g_{1}}^{e}
\end{array}\right),\;\left(\begin{array}{cc}
g_{1}^{e} & g_{2}^{e}\\
0 & \overline{g_{1}}^{e}
\end{array}\right)\mapsto g_{1}^{e},
\end{align*}
where one should note that the quotient $Q_{p}^{1}\to\mathcal{Q}_{p}^{1}$
is a 2-to-1 map, by which $B$ and $-B$ has the same image. For an
element $f^{e}\in\left(\frac{\mathbb{F}_{p}\left[t,\,t^{-1}\right]}{\left(1+t+t^{2}\right)}\right)^{\times}$,
define an equivalence relation $f^{e}\sim\pm f^{e}$. Denote by $R_{p}$
the quotient $\left(\frac{\mathbb{F}_{p}\left[t,\,t^{-1}\right]}{\left(1+t+t^{2}\right)}\right)^{\times}/\sim$,
and denote by $q_{\sim}$ the quotient map. It is direct to see that
$R_{p}$ is also a group. By projectivizing $\eta_{p}$ to $\overline{\eta_{p}}:\mathcal{Q}_{p}^{1}\to R_{p}$,
we have a commutative diagram:
\begin{align*}
\xymatrix{1\ar[r] & \ker\eta_{p}\ar[r]\ar[d]^{\pi_{p}} & Q_{p}^{1}\ar[d]^{\pi_{p}}\ar[r]^{\eta_{p}} & \left(\frac{\mathbb{F}_{p}\left[t,\,t^{-1}\right]}{\left(1+t+t^{2}\right)}\right)^{\times}\ar[d]^{q_{\sim}}\\
1\ar[r] & \ker\overline{\eta_{p}}\ar[r] & \mathcal{Q}_{p}^{1}\ar[r]^{\overline{\eta_{p}}} & R_{p}
}
\end{align*}

By Lemma 2.10, we have $\ker\eta_{p}\subset\phi_{p}\left(\mathcal{B}_{p}\right)$,
which implies $\ker\overline{\eta_{p}}\subset I_{p}^{1}$. Therefore,
we have an identity:
\begin{align}
\left[\mathcal{Q}_{p}^{1}\::\:I_{p}^{1}\right] & =\left[\overline{\eta_{p}}\left(\mathcal{Q}_{p}^{1}\right)\::\:\overline{\eta_{p}}\left(I_{p}^{1}\right)\right].
\end{align}

Suppose $p>3$. For the upper bound of (18), it suffices to show that
the right-hand side of (19) is not larger than $\frac{p-\left(\frac{-3}{p}\right)}{6}$.
Take an element $M\in\mathcal{Q}_{p}^{1}$, and choose its representative
$B=\left(\begin{array}{cc}
g_{1} & g_{2}\\
-\Phi\overline{g_{2}} & \overline{g_{1}}
\end{array}\right)$ such that $g_{1},\,g_{2}\in\mathbb{F}_{p}\left[t,\,t^{-1}\right]$
and $\det\left(B\right)=1$. The condition about the determinant yields
\begin{align*}
g_{1}\overline{g_{1}}+\left(1+t^{-1}+t\right)g_{2}\overline{g_{2}} & =1.
\end{align*}

Therefore, we have $g_{1}^{e}\overline{g_{1}}^{e}=1$. Put $g_{1}^{e}=at+b$
for some $a,\,b\in\mathbb{F}_{p}$. Then, we have
\begin{align}
g_{1}^{e}\overline{g_{1}}^{e} & =a^{2}-ab+b^{2}.
\end{align}

We rewrite the right-hand side of (20) as
\begin{align*}
\left(a-\frac{1}{2}b\right)^{2}+3\left(\frac{b}{2}\right)^{2} & =1,
\end{align*}
which implies there are at most $p-\left(\frac{-3}{p}\right)$ possible
combinations of pairs $\left(a,\,b\right)\in\mathbb{F}_{p}^{2}$ through
the usual parametrization of generalized Pythagorean triples. Because
the equivalence between $\left(a,\,b\right)$ and $\left(-a,\,-b\right)$
under $\sim$ halves the number of elements in $R_{p}$ mapped from
$\left(\frac{\mathbb{F}_{p}\left[t,\,t^{-1}\right]}{\left(1+t+t^{2}\right)}\right)^{\times}$,
the order of the group $\overline{\eta_{p}}\left(\mathcal{Q}_{p}^{1}\right)$
is at most $\frac{p-\left(\frac{-3}{p}\right)}{2}$.

On the other hand, take an element $M\in I_{p}^{1}$, and choose its
representative $B=\left(\begin{array}{cc}
g_{1} & g_{2}\\
-\Phi\overline{g_{2}} & \overline{g_{1}}
\end{array}\right)$ such that $g_{1},\,g_{2}\in\mathbb{F}_{p}\left[t,\,t^{-1}\right]$
and $\det\left(B\right)=1$. Lemma 2.9 and Lemma 2.10 requires that
$\left(\pm t^{k}g_{1}\right)^{e}=1$ for some integer $k$. Since
we have $\left(t^{3}\right)^{e}=1$, the only possible options for
$g_{1}^{e}$ are
\begin{align}
\pm g_{1}^{e} & =1,\,t,\,1+t,
\end{align}
which makes the order of $\overline{\eta_{p}}\left(I_{p}^{1}\right)$
at most 3. The order is indeed 3, because the group $I_{p}^{1}$ includes
$\overline{g_{p}}\left[0\right]$ and $\overline{g_{p}}\left[0\right]^{2}$.
Therefore, we have
\begin{align*}
\left[\overline{\eta_{p}}\left(\mathcal{Q}_{p}^{1}\right)\::\:\overline{\eta_{p}}\left(I_{p}^{1}\right)\right] & \le\frac{p-\left(\frac{-3}{p}\right)}{6},
\end{align*}
which gives the upper bound for (18). This bound implies $\mathcal{Q}_{p}^{1}=I_{p}^{1}$
for the cases $p=5,\,7$. This is also true for the cases $p=2,\,3$
by counting possible combinations $\left(a,\,b\right)$ satisfying
(20).

We now prove the lower bound for (18) by using Hasse's theorem.

\noindent \begin{claim1}

Suppose $p>3$ and there is a matrix $g_{p}\left[r\right]$ such that
$r\ne0$, $r^{2}\ne-1$, and the determinant $1+r^{2}+r^{4}$ is a
nonzero quadratic residue modulo $p$. Then, we have
\begin{align*}
2 & \le\left[\mathcal{Q}_{p}^{1}\::\:I_{p}^{1}\right].
\end{align*}
\end{claim1}
\begin{proofclaim1}

\noindent Suppose $\mathcal{Q}_{p}^{1}=I_{p}^{1}$. Choose an element
$\alpha\in\mathbb{F}_{p}^{\times}$ such that $\alpha^{2}\left(1+r^{2}+r^{4}\right)=1$.
Then, we have $\det\left(\alpha g\left[r\right]\right)=1$ and $\left(\alpha g\left[r\right]\right)_{11}^{e}=\alpha t-\alpha r^{2}$.
We expand the cube of this entry as
\begin{align*}
 & \,\left(\left(\alpha g\left[r\right]\right)_{11}^{e}\right)^{3}\\
= & \,\left(\alpha t-\alpha r^{2}\right)^{3}\\
= & \,\alpha^{3}\left(t^{3}-3t^{2}r^{2}+3tr^{4}-r^{6}\right)\\
= & \,\alpha^{3}\left(1+3r^{2}-r^{6}+3t\left(r^{2}+r^{4}\right)\right),
\end{align*}
which is in particular not included in $\mathbb{F}_{p}$, because
the assumption implies $r^{2}+r^{4}\ne0$. When we suppose $\mathcal{Q}_{p}^{1}=I_{p}^{1}$,
we must have $\left(\left(\alpha g\left[r\right]\right)_{11}^{e}\right)^{3}=\pm1$
by (21), which implies $\mathcal{Q}_{p}^{1}\ne I_{p}^{1}$. $\qed$

\noindent \end{proofclaim1}

We return to the proof of Theorem 2.17. Suppose $p>7$. By Claim 1,
it suffices to show that there is a matrix $g_{p}\left[r\right]$
satisfying the conditions given in Claim 1. We consider how many rational
points $\left(x,\,y\right)$ the determinant curve $E_{d}$ given
by (17) has, depending on the conditions in Claim 1. The condition
$r=0$ gives two points $\left(0,\,\pm1\right)$. For an element $r_{0}\in\mathbb{F}_{p}^{\times}$
satisfying $r_{0}^{2}=-1$, the condition $r^{2}=-1$ gives four points
$\left(\pm r_{0},\,\pm1\right)$. If the equation $x^{4}+x^{2}+1=0$
has a solution, then it has exactly four distinct solutions. In this
case, for each solution $x$, we have a point $\left(x,\,0\right)$.
Therefore, the condition $1+r^{2}+r^{4}=0$ gives four points.

In any case, the number of points excluded by Claim 1 is at most 10.
The lower bound $p+1-2\sqrt{p}$ by Theorem 2.16 implies $10<\#E_{d}\left(\mathbb{F}_{p}\right)$
when $p>17$. Since there is no solution to the equation ${x^{4}+x^{2}+1=0}$
over $\mathbb{F}_{17}$, we have $6<\#E_{d}\left(\mathbb{F}_{17}\right)$
by the bound. Since both equations $x^{2}=-1$ and $x^{4}+x^{2}+1=0$
have no solutions over $\mathbb{F}_{11}$, we have $2<\#E_{d}\left(\mathbb{F}_{11}\right)$.
Now, the only remaining prime is 13. In the field $\mathbb{F}_{13}$,
$3$ is a quadratic residue as $4^{2}=3$. Thus, the matrix $g_{13}\left[1\right]$
satisfies the three conditions in Claim 1. $\qedhere$

\noindent \end{proof}
\begin{lemma}

The subgroup index of $\mathcal{Q}_{p}^{1}$ in $\mathcal{Q}_{p}$
is 1 for $p\le3$ and 2 for $p>3$.

\noindent \end{lemma}
\begin{proof}

\noindent The case $p=2$ is trivial. The case $p=3$ requires a detailed
analysis of the structure of $\mathcal{Q}_{p}$, which we postpone
to the end of Section 3. Suppose $p>3$. Since the determinant map
is defined only up to square multiplication for projective linear
groups, we have an exact sequence induced by the determinant map
\begin{align*}
1 & \to\mathcal{Q}_{p}^{1}\to\mathcal{Q}_{p}\to\mathbb{F}_{p}^{\times}/\left(\mathbb{F}_{p}^{\times}\right)^{2},
\end{align*}
where $\mathbb{F}_{p}^{\times}/\left(\mathbb{F}_{p}^{\times}\right)^{2}$
is cyclic of order 2 for the cases $p>3$. Thus, the index $\left[\mathcal{Q}_{p}\::\:\mathcal{Q}_{p}^{1}\right]$
must be 1 or 2. Suppose this value is 1. Then, each elementary generator
$\overline{g_{p}}\left[r\right]$ has a representative having a square
determinant. In other words, $r^{4}+r^{2}+1$ is always square for
each $r\in\mathbb{F}_{p}$, and the determinant curve $E_{d}$ given
by (17) has at least $2p-4$ rational points, since if $\left(x,\,y\right)$
is a rational point such that $x^{4}+x^{2}+1\ne0$, then $\left(x,\,-y\right)$
is also a rational point. By the upper bound given by Theorem 2.16,
we have an inequality:
\[
2p-4\le\#E_{d}\left(\mathbb{F}_{p}\right)\le p+1+2\sqrt{p},
\]
which holds only when $p\le11$, and we get a contradiction for the
cases $p>11$. If the equation $x^{4}+x^{2}+1=0$ has no solution
over $\mathbb{F}_{p}$, the bound gives an inequality:
\[
2p\le\#E_{d}\left(\mathbb{F}_{p}\right)\le p+1+2\sqrt{p},
\]
which holds only when $p\le5$, making the case $p=11$ contradictory.
Now, the only remaining primes are 5 and 7. For these cases, 3 is
a quadratic nonresidue, which implies that $\overline{g_{5}}\left[1\right]\notin\mathcal{Q}_{5}^{1}$
and $\overline{g_{7}}\left[1\right]\notin\mathcal{Q}_{7}^{1}$. $\qedhere$

\noindent \end{proof}

\section{Structure Theorem of $\mathcal{Q}_{p}$}

In our previous paper \citep{lee2024salters}, we proved that the
quaternionic group $\mathcal{Q}$ over the field $\mathbb{Q}$ is
freely generated by the set of elementary generators $\overline{g}\left[r\right]$
for $r\in\mathbb{Q}$. Over $\mathbb{F}_{p}$, the equation
\begin{align}
x^{4}+x^{2}+1 & =0
\end{align}
may have a solution. Concretely, the equation is equivalent to $\left(x^{2}+x+1\right)\left(x^{2}-x+1\right)=0$,
which has no solution if and only if $\left(\frac{-3}{p}\right)=-1$
when $p>3$, where $\left(\frac{\cdot}{p}\right)$ is the Legendre
symbol. By elementary number theory, this condition is equivalent
to $p$ being 5 modulo 6. For $p=2$, it is direct that (22) has no
solution. In these cases, we need only the elementary generators to
generate the whole quaternionic group $\mathcal{Q}_{p}$. On the other
hand, when $p\equiv1,\,3$ (mod 6), we are not able to define the
elementary generator $\overline{g_{p}}\left[r\right]$ when $r$ is
a solution of (22), and need other sorts of generators. The goal of
this section is to identify the generators and describe the structure
of $\mathcal{Q}_{p}$ for every prime $p$.

\noindent \begin{definition}

For a Laurent polynomial
\begin{align*}
g & =a_{0}t^{m}+a_{1}t^{m+1}+\cdots+a_{n-m}t^{n}
\end{align*}
such that $n\ge m$ and $a_{0}\ne0\ne a_{n-m}$, define the \emph{relative
degree} $\mathrm{rd}\left(g\right):=n-m$.

\noindent \end{definition}

When a matrix $\left(\begin{array}{cc}
g_{1} & g_{2}\\
-\Phi\overline{g_{2}} & \overline{g_{1}}
\end{array}\right)$ is included in $\mathrm{GL}\left(2,\,\mathbb{F}_{p}\left[t,t^{-1}\right]\right)$,
from the functional equation (15), it is straightforward that
\begin{align*}
\mathrm{rd}\left(g_{1}\right) & =\mathrm{rd}\left(g_{2}\right)+1.
\end{align*}

Thus, in this case, we also define
\begin{align*}
\mathrm{rd}\left(\begin{array}{cc}
g_{1} & g_{2}\\
-\Phi\overline{g_{2}} & \overline{g_{1}}
\end{array}\right) & :=\mathrm{rd}\left(g_{1}\right).
\end{align*}

For a matrix $B=\left(\begin{array}{cc}
g_{1} & g_{2}\\
-\Phi\overline{g_{2}} & \overline{g_{1}}
\end{array}\right)\in\mathrm{GL}\left(2,\,\mathbb{F}_{p}\left[t,t^{-1}\right]\right)$ such that $g_{2}\ne0$, write the first row entries as
\begin{align}
 & g_{1}=a_{0}t^{m_{1}}+a_{1}t^{m_{1}+1}+\cdots+a_{n_{1}-m_{1}}t^{n_{1}},\\
 & g_{2}=b_{0}t^{m_{2}}+b_{1}t^{m_{2}+1}+\cdots+b_{n_{2}-m_{2}}t^{n_{2}}.
\end{align}

Then, we also have $n_{1}-m_{1}=n_{2}-m_{2}+1$ from (15).

\noindent \begin{definition}

For a matrix $B=\left(\begin{array}{cc}
g_{1} & g_{2}\\
-\Phi\overline{g_{2}} & \overline{g_{1}}
\end{array}\right)\in\mathrm{GL}\left(2,\,\mathbb{F}_{p}\left[t,t^{-1}\right]\right)$ such that $g_{2}\ne0$, write the first row entries as in (23) and
(24). Define $B$ to be \emph{upper-balanced} if $m_{1}=m_{2}$ and
${n_{1}=n_{2}+1}$, and \emph{lower-balanced} if $m_{1}=m_{2}-1$
and $n_{1}=n_{2}$. Also define $B$ to be \emph{balanced} if $B$
is upper-balanced or lower-balanced. If $g_{2}=0$ or $B$ is not
balanced, we call $B$ \emph{unbalanced}. For a balanced matrix $B$,
define $B$ to be \emph{evenly balanced} (resp. \emph{oddly balanced})
if $\mathrm{rd}\left(B\right)$ is an even number (resp. an odd number).

\noindent \end{definition}
\begin{lemma}

For a matrix $B=\left(\begin{array}{cc}
g_{1} & g_{2}\\
-\Phi\overline{g_{2}} & \overline{g_{1}}
\end{array}\right)\in\mathrm{GL}\left(2,\,\mathbb{F}_{p}\left[t,t^{-1}\right]\right)$, suppose $g_{2}\ne0$. Then, there exists a unique balanced matrix
$\mathrm{bl}\left(B\right)$ among the set of matrices
\begin{align*}
\left\{ Bg_{p}\left[0\right]^{k}\;|\;k\in\mathbb{Z}\right\} .
\end{align*}

Let us call $\mathrm{bl}\left(B\right)$ the \emph{balanced companion
of} $B$.

\noindent \end{lemma}
\begin{proof}

From (16), we have $g_{p}\left[0\right]=\left(\begin{array}{cc}
t & 0\\
0 & t^{-1}
\end{array}\right)$. Thus, the number $n_{1}-n_{2}$ increases by 2 multiplying $g_{p}\left[0\right]$
on the right of $B$, and decreases by 2 multiplying $g_{p}\left[0\right]^{-1}$.
Thus, the matrix $Bg_{p}\left[0\right]^{k}$ is balanced if and only
if $k=\left\lfloor \frac{n_{2}-n_{1}+1}{2}\right\rfloor $. $\qedhere$

\noindent \end{proof}
\begin{definition}

For a matrix $B=\left(\begin{array}{cc}
g_{1} & g_{2}\\
-\Phi\overline{g_{2}} & \overline{g_{1}}
\end{array}\right)\in\mathrm{GL}\left(2,\,\mathbb{F}_{p}\left[t,t^{-1}\right]\right)$ such that $g_{2}\ne0$, write the first row entries as in (23) and
(24). We define the \emph{type} $\tau\left(B\right)$ to be $-a_{0}b_{0}^{-1}$
when $\mathrm{bl}\left(A\right)$ is upper-balanced, and $-a_{0}^{-1}b_{0}$
when $\mathrm{bl}\left(B\right)$ is lower-balanced. For an unbalanced
matrix $B$, define $\tau\left(B\right)$ to be 0. Abusing notation,
for an element $M\in\mathcal{Q}_{p}$, we also define $\tau\left(M\right)$
to be $\tau\left(B\right)$ for a representative $B$ of $M$, and
we will call $M$ (resp. upper-, lower-, evenly, oddly) balanced if
$B$ is (resp. upper-, lower-, evenly, oddly) balanced.

\noindent \end{definition}

From now on, denote by $I$ the identity matrix $\left(\begin{array}{cc}
1 & 0\\
0 & 1
\end{array}\right)$ over any base ring. For simplicity, for any pairs of Laurent polynomials
$g_{1},\,g_{2}\in\mathbb{F}_{p}\left[t,\,t^{-1}\right]$, we define
the \emph{dagger} notation as
\begin{align*}
\left(\begin{array}{cc}
g_{1} & g_{2}\\
-\Phi\overline{g_{2}} & \overline{g_{1}}
\end{array}\right)^{\dagger} & :=\left(\begin{array}{cc}
\overline{g_{1}} & -g_{2}\\
\Phi\overline{g_{2}} & g_{1}
\end{array}\right).
\end{align*}

For a matrix $B=\left(\begin{array}{cc}
g_{1} & g_{2}\\
-\Phi\overline{g_{2}} & \overline{g_{1}}
\end{array}\right)$, it is direct that if $B\in\mathrm{GL}\left(2,\,\mathbb{F}_{p}\left[t,\,t^{-1}\right]\right)$,
then we have $B^{\dagger}=\det\left(B\right)B^{-1}$. On the other
hand, when $\det\left(B\right)=0$, we have $BB^{\dagger}=0=B^{\dagger}B$.
We are ready to list the generators of the quaternionic group $\mathcal{Q}_{p}$,
aside from the elementary generators.

\noindent \begin{definition}

For a unit $\tau\in\mathbb{F}_{3}^{\times}$ and a polynomial $f\left(x\right)\in\mathbb{F}_{3}\left[x\right]$,
we define a matrix
\begin{align*}
a_{3,\,\tau,\,u}\left[f\right] & :=I+f\left(t+t^{-1}\right)\left(\begin{array}{cc}
1+t^{-1} & 0\\
0 & 1+t
\end{array}\right)g_{3}\left[\tau\right],
\end{align*}
where we simplified $f\left(x\right)$ to $f$ for simplicity if there
is no room for confusion.

Define the \emph{upper additive generator }for $f$ to be its projectivization
$\overline{a_{3,\,\tau,\,u}}\left[f\right]:=\pi_{3}\left(a_{3,\,\tau,\,u}\left[f\right]\right)$.
Similarly, define a matrix
\begin{align*}
a_{3,\,\tau,\,l}\left[f\right] & :=I+f\left(t+t^{-1}\right)\left(\begin{array}{cc}
1+t & 0\\
0 & 1+t^{-1}
\end{array}\right)g_{3}\left[\tau\right]^{\dagger}.
\end{align*}

Define the \emph{lower additive generator }for $f$ to be its projectivization
$\overline{a_{3,\,\tau,\,l}}\left[f\right]:=\pi_{3}\left(a_{3,\,\tau,\,l}\left[f\right]\right)$.

\noindent \end{definition}
\begin{definition}

For a prime $p$ such that $p\equiv1$ (mod 6), suppose an element
$\tau\in\mathbb{F}_{p}$ satisfies that $\tau^{4}+\tau^{2}+1=0$.
We define matrices
\begin{align*}
 & a_{p,\,\tau,\,u}\left[f\right]=I+f\left(t+t^{-1}\right)g_{p}\left[\tau\right]^{\dagger}\left(\begin{array}{cc}
t-t^{-1} & 0\\
0 & 0
\end{array}\right)g_{p}\left[\tau\right],\\
 & a_{p,\,\tau,\,l}\left[f\right]=I+f\left(t+t^{-1}\right)g_{p}\left[\tau\right]\left(\begin{array}{cc}
t-t^{-1} & 0\\
0 & 0
\end{array}\right)g_{p}\left[\tau\right]^{\dagger}.
\end{align*}

Define the \emph{upper additive generator} for $f$ to be $\overline{a_{p,\,\tau,\,u}}\left[f\right]:=\pi_{p}\left(a_{p,\,\tau,\,u}\left[f\right]\right)$,
and the \emph{lower additive generator} for $f$ to be $\overline{a_{p,\,\tau,\,l}}\left[f\right]:=\pi_{p}\left(a_{p,\,\tau,\,l}\left[f\right]\right)$.

\noindent \end{definition}
\begin{definition}

For a prime $p$ such that $p\equiv1$ (mod 6), suppose an element
$\tau\in\mathbb{F}_{p}$ satisfies that $\tau^{4}+\tau^{2}+1=0$.
For each $a\in\mathbb{F}_{p}$, we define matrices
\begin{align*}
 & m_{p,\,\tau,\,u}\left[a\right]:=\left(\begin{array}{cc}
\frac{-\tau+at+\tau t^{2}}{t} & \frac{1+\tau^{2}t}{t}\\
-\Phi\left(\tau^{2}+t\right) & \frac{\tau+at-\tau t^{2}}{t}
\end{array}\right),\\
 & m_{p,\,\tau,\,l}\left[a\right]:=\left(\begin{array}{cc}
\frac{-\tau+at+\tau t^{2}}{t} & \tau^{2}+t\\
-\Phi\left(\frac{1+\tau^{2}t}{t}\right) & \frac{\tau+at-\tau t^{2}}{t}
\end{array}\right).
\end{align*}

Suppose that $\det\left(m_{p,\,\tau,\,u}\left[a\right]\right)\ne0$.
Define the \emph{upper multiplicative generator }for $a$ to be $\overline{m_{p,\,\tau,\,u}}\left[a\right]:=\pi_{p}\left(m_{p,\,\tau,\,u}\left[a\right]\right)$.
Similarly, when $\det\left(m_{p,\,\tau,\,l}\left[a\right]\right)\ne0$,
define the \emph{lower multiplicative generator} for $a$ to be $\overline{m_{p,\,\tau,\,l}}\left[a\right]:=\pi_{p}\left(m_{p,\,\tau,\,l}\left[a\right]\right)$.

\noindent \end{definition}
\begin{lemma}

For a prime $p$ such that $p\equiv1$ (mod 6), suppose an element
$\tau\in\mathbb{F}_{p}$ satisfies that $\tau^{4}+\tau^{2}+1=0$.
We have
\begin{align*}
\det\left(m_{p,\,\tau,\,u}\left[a\right]\right) & =a^{2}+3\tau^{2}=\det\left(m_{p,\,\tau,\,l}\left[a\right]\right).
\end{align*}

In particular, the multiplicative generators $\overline{m_{p,\,\tau,\,u}}\left[a\right]$
and $\overline{m_{p,\,\tau,\,l}}\left[a\right]$ are defined if and
only if $a\ne\pm\left(\tau+2\tau^{-1}\right)$.

\noindent \end{lemma}
\begin{proof}

\noindent We have the determinant by direct computation using the
assumption $\tau^{4}+\tau^{2}+1=0$. Therefore, the determinant is
nonzero if and only if $a^{2}\ne-3\tau^{2}$. By the assumption, the
only solutions for the equation $x^{2}=-3$ are $\pm\left(2\tau^{2}+1\right)$.
We conclude the proof by the equality $\pm\tau\left(2\tau^{2}+1\right)=\pm\left(\tau+2\tau^{-1}\right)$.
$\qedhere$

\noindent \end{proof}

For the cases $p\equiv1$ (mod 6), we will collectively refer to additive
generators and multiplicative generators as \emph{affine generators}.

\noindent \begin{definition}

For a prime $p$ such that $p\equiv1$ (mod 6), suppose an element
$\tau\in\mathbb{F}_{p}$ satisfies that $\tau^{4}+\tau^{2}+1=0$.
A \emph{stable generator} is an element in $\mathcal{Q}_{p}$ of the
following form:
\begin{align*}
\left[\begin{array}{cc}
\frac{-\tau+a_{1}t+a_{2}t^{2}+\tau^{-1}t^{3}}{t} & \frac{1+a_{3}t+t^{2}}{t}\\
-\Phi\left(\frac{1+a_{3}t+t^{2}}{t}\right) & \frac{\tau^{-1}+a_{2}t+a_{1}t^{2}-\tau t^{3}}{t^{2}}
\end{array}\right],
\end{align*}
where we used the upper-balanced representative as a convention.

\noindent \end{definition}
\begin{lemma}

For a prime $p$ such that $p\equiv1$ (mod 6), suppose an element
$\tau\in\mathbb{F}_{p}$ satisfies that $\tau^{4}+\tau^{2}+1=0$.
Then, there are exactly $p-1$ stable generators for type $\tau$.

\noindent \end{lemma}
\begin{proof}

\noindent By Definition 3.9, we compute the determinant:
\begin{equation}
\det\left(\begin{array}{cc}
\frac{-\tau+a_{1}t+a_{2}t^{2}+\tau^{-1}t^{3}}{t} & \frac{1+a_{3}t+t^{2}}{t}\\
-\Phi\left(\frac{1+a_{3}t+t^{2}}{t}\right) & \frac{\tau^{-1}+a_{2}t+a_{1}t^{2}-\tau t^{3}}{t^{2}}
\end{array}\right).
\end{equation}

From the $t^{2}$ coefficient of (25), we derive the following equation:
\begin{align*}
\tau^{-1}\left(-\tau+a_{1}-\tau^{2}a_{2}+2\tau a_{3}\right) & =0.
\end{align*}

We choose $a_{1}=-\tau+\tau^{2}a_{2}-2\tau a_{3}$. Substituting this
into (25), we obtain an equation from the $t$ coefficient
\begin{align*}
a_{3}^{2}+2a_{3}+2\tau^{2}a_{3}-2\tau a_{2}a_{3}+3+\tau^{2}+2\tau^{-1}a_{2}+\tau^{2}a_{2}^{2} & =0.
\end{align*}

Choosing as $a_{3}$ the variable, we solve this qudratic equation.
We have two solutions:
\begin{align*}
a_{3} & =-1-\tau^{2}+\tau a_{2}\pm\left(2\tau^{2}+1\right),
\end{align*}
where we used the assumption $\left(2\tau^{2}+1\right)^{2}=-3$. By
taking the $+$ sign, suppose $a_{3}=\tau^{2}+\tau a_{2}$. Then,
the determinant (25) becomes
\begin{align*}
\tau^{-2}\left(1+\tau^{2}+\tau^{4}\right)\left(1+4\tau^{4}+4\tau^{3}a_{2}+\tau^{2}\left(1+a_{2}^{2}\right)\right),
\end{align*}
which must be zero by the assumption $\tau^{4}+\tau^{2}+1=0$. Therefore,
we should take the $-$ sign and suppose $a_{3}=-3\tau^{2}+\tau a_{2}-2$.
Then, the determinant (25) becomes
\begin{align*}
-12\left(1+3\tau^{2}-\tau a_{2}\right),
\end{align*}
which is zero if and only if $a_{2}=3\tau+\tau^{-1}$. In conclusion,
there are exactly $p-1$ options for the choices of $a_{2}$. $\qedhere$

\noindent \end{proof}

From now on, by the proof of Lemma 3.10, for $a\in\mathbb{F}_{p}\backslash\left\{ 3\tau+\tau^{-1}\right\} $
we define the stable generator
\begin{align*}
\overline{s_{p,\,\tau}}\left[a\right] & :=\left[\begin{array}{cc}
\frac{-\tau+a_{1}t+at^{2}+\tau^{-1}t^{3}}{t} & \frac{1+a_{3}t+t^{2}}{t}\\
-\Phi\left(\frac{1+a_{3}t+t^{2}}{t}\right) & \frac{\tau^{-1}+at+a_{1}t^{2}-\tau t^{3}}{t^{2}}
\end{array}\right],
\end{align*}
where $a_{1}=\tau\left(3-a\tau+6\tau^{2}\right)$ and $a_{3}=-3\tau^{2}+\tau a-2$.
For future use, we also define the matrix representative
\begin{align*}
s_{p,\,\tau}\left[a\right] & :=\left(\begin{array}{cc}
\frac{-\tau+a_{1}t+at^{2}+\tau^{-1}t^{3}}{t} & \frac{1+a_{3}t+t^{2}}{t}\\
-\Phi\left(\frac{1+a_{3}t+t^{2}}{t}\right) & \frac{\tau^{-1}+at+a_{1}t^{2}-\tau t^{3}}{t^{2}}
\end{array}\right).
\end{align*}

We state the main theorem of this section. In the following statement,
we denote by $\mathbb{F}_{p}\left[x\right]$ the additive group isomorphic
to the infinite sum of finite cyclic groups $C_{p}^{\infty}$, and
denote by $\mathrm{Aff}\left(\mathbb{F}_{p}\left[x\right]\right)$
the affine group over the ring $\mathbb{F}_{p}\left[x\right]$, which
is isomorphic to $\mathbb{F}_{p}\left[x\right]\rtimes\mathbb{F}_{p}^{\times}$.

\noindent \begin{theorem}

For each prime $p$, there exist $p$ subgroups $\mathcal{G}_{p,\,0},\,\mathcal{G}_{p,\,1},\cdots,\,\mathcal{G}_{p,\,p-1}$,
called \emph{type groups}, of the quaternionic group $\mathcal{Q}_{p}$
satisfying the following properties.
\begin{enumerate}
\item For each integer $r$ such that $0\le r\le p-1$, each nontrivial
element of $\mathcal{G}_{p,\,r}$ has type $r$.
\item $\mathcal{Q}_{p}$ is the free product of $p$ subgroups $\left\{ \mathcal{G}_{p,\,r}\right\} _{0\le r\le p-1}$.
\item Suppose $r^{4}+r^{2}+1\ne0$ (mod $p$). Then, the group $\mathcal{G}_{p,\,r}$
is an infinite cyclic group generated by the elementary generator
$\overline{g}\left[r\right]$. In particular, for the cases $p\equiv2,\,5$
(mod 6), $\mathcal{Q}_{p}$ is freely generated by $p$ elementary
generators.
\item Suppose $p=3$. For a unit $\tau\in\mathbb{F}_{3}^{\times}$, the
set of upper additive generators $\overline{a_{3,\,\tau,\,u}}\left[f\right]$
(resp. the set of lower additive generators $\overline{a_{3,\,\tau,\,l}}\left[f\right]$)
generates an abelian group isomorphic to $\mathbb{F}_{3}\left[x\right]$.
The group $\mathcal{G}_{3,\,\tau}$ is the free product of the subgroup
generated by upper additive generators $\overline{a_{3,\,\tau,\,u}}\left[f\right]$
and the subgroup generated by lower additive generators $\overline{a_{3,\,\tau,\,l}}\left[f\right]$.
In particular, $\mathcal{G}_{3,\,\tau}$ is isomorphic to $\mathbb{F}_{p}\left[x\right]*\mathbb{F}_{p}\left[x\right]$.
\item For $p\equiv1$ (mod 6), suppose an element $\tau\in\mathbb{F}_{p}$
satisfies that $\tau^{4}+\tau^{2}+1=0$. Then, we have the following:
\begin{enumerate}
\item the set of upper additive generators $\overline{a_{p,\,\tau,\,u}}\left[f\right]$
(resp. the set of lower additive generators $\overline{a_{p,\,\tau,\,l}}\left[f\right]$)
generates an abelian group isomorphic to $\mathbb{F}_{p}\left[x\right]$, 
\item the set of upper multiplicative generators $\overline{m_{p,\,\tau,\,u}}\left[a\right]$
(resp. the set of lower multiplicative generators $\overline{m_{p,\,\tau,\,l}}\left[a\right]$)
generates an abelian group isomorphic to the unit group $\mathbb{F}_{p}^{\times}$,
\item the set of upper (resp. lower) affine generators generates a semidirect
product isomorphic to $\mathbb{F}_{p}\left[x\right]\rtimes\mathbb{F}_{p}^{\times}$,
\item for an element $a\in\mathbb{F}_{p}\backslash\left\{ 3\tau+\tau^{-1}\right\} $,
a stable generator $\overline{s_{p,\,\tau}}\left[a\right]$ generates
an infinite cyclic group, and
\item the group $\mathcal{G}_{p,\,\tau}$ is generated by a stable generator
and the set of affine generators, and is an (internal) HNN extension
in which the upper affine generators and the lower affine generators
generate the free product of the subgroups, and the stable generator
plays the role of the stable letter. Concretely, we have an isomorphism
\begin{align*}
\mathcal{G}_{p,\,\tau} & \cong\left(\mathrm{Aff}\left(\mathbb{F}_{p}\left[x\right]\right)*\mathrm{Aff}\left(\mathbb{F}_{p}\left[x\right]\right)\right)*_{\alpha_{\tau}},
\end{align*}
where the affine group $\mathrm{Aff}\left(\mathbb{F}_{p}\left[x\right]\right)$,
isomorphic to $\mathbb{F}_{p}\left[x\right]\rtimes\mathbb{F}_{p}^{\times}$,
is corresponding to the subgroup generated by upper (or lower) affine
generators, $\alpha_{\tau}:\left(\mathbb{F}_{p}^{\times}\right)_{u}\to\left(\mathbb{F}_{p}^{\times}\right)_{l}$
is corresponding to the isomorphism from the unit group generated
by the upper generators to that by the lower generators, and $*_{\alpha_{\tau}}$
means the group is the HNN extension by $\alpha_{\tau}$.
\end{enumerate}
\end{enumerate}
For the case $p=3$, there are two isomorphic factors $\mathcal{G}_{3,\,\tau}$
such that $\tau\ne0$. For the cases $p\equiv1$ (mod 6), there are
exactly four isomorphic factors $\mathcal{G}_{p,\,\tau}$ such that
$\tau^{4}+\tau^{2}+1=0$.

\noindent \end{theorem}
\begin{proof}

\noindent The proof is divided into several cases. For primes $p$
such that $p\equiv2,\,5$ (mod 6), we prove Theorem 3.14. For the
prime $3$, we prove Theorem 3.20. For primes $p$ such that $p\equiv1$
(mod 6), we prove Theorem 4.1. $\qedhere$

\noindent \end{proof}

This theorem is a full-fledged version of \citep[Theorem 4.9]{lee2024salters},
where there is no solution to ${x^{4}+x^{2}+1=0}$ over the base field
$\mathbb{Q}$. The rest of this section and the entire next section
will be devoted to the proof of Theorem 3.11.

\noindent \begin{definition}

For a balanced element $M\in\mathcal{Q}_{p}$ such that $\tau\left(M\right)\ne0$,
choose a a representative $B=\left(\begin{array}{cc}
g_{1} & g_{2}\\
-\Phi\overline{g_{2}} & \overline{g_{1}}
\end{array}\right)$ and write the first row entries:
\begin{align*}
 & g_{1}=a_{0}t^{m_{1}}+a_{1}t^{m_{1}+1}+\cdots+a_{n_{1}-m_{1}}t^{n_{1}},\\
 & g_{2}=b_{0}t^{m_{2}}+b_{1}t^{m_{2}+1}+\cdots+b_{n_{2}-m_{2}}t^{n_{2}}.
\end{align*}

Define the \emph{signature} $\sigma\left(M\right)$ to be the ratio
$\frac{a_{n_{1}-m_{1}}}{a_{0}}$. It is invariant with respect to
the choice of representative $B$.

\noindent \end{definition}

We first introduce the following lemma for the proof of (2) in Theorem
3.11.

\noindent \begin{lemma}

For two balanced elements $M_{1},\,M_{2}\in\mathcal{Q}_{p}$ with
nonzero types, suppose one of the following:
\begin{enumerate}
\item $M_{1}$ is upper-balanced and $M_{2}$ is oddly upper-balanced.
\item $M_{1}$ is upper-balanced, $M_{2}$ is evenly upper-balanced, and
$\sigma\left(M_{2}\right)\ne-\frac{\tau\left(M_{1}\right)}{\tau\left(M_{2}\right)}$.
\item $M_{1}$ is upper-balanced and $M_{2}$ is evenly lower-balanced.
\item $M_{1}$ is upper-balanced, $M_{2}$ is oddly lower-balanced, and
$\sigma\left(M_{2}\right)\ne-\left(\tau\left(M_{1}\right)\tau\left(M_{2}\right)\right)$.
\item $M_{1}$ is lower-balanced and $M_{2}$ is evenly upper-balanced.
\item $M_{1}$ is lower-balanced, $M_{2}$ is oddly upper-balanced, and
$\sigma\left(M_{2}\right)\ne-\left(\tau\left(M_{1}\right)\tau\left(M_{2}\right)\right)^{-1}$.
\item $M_{1}$ is lower-balanced and $M_{2}$ is oddly lower-balanced.
\item $M_{1}$ is lower-balanced, $M_{2}$ is evenly lower-balanced, and
$\sigma\left(M_{2}\right)\ne-\frac{\tau\left(M_{2}\right)}{\tau\left(M_{1}\right)}$.
\end{enumerate}
Then, we have $\tau\left(M_{1}M_{2}\right)=\tau\left(M_{2}\right)$.

\noindent \end{lemma}
\begin{proof}

\noindent Suppose that both $M_{1}$ and $M_{2}$ are upper-balanced.
Temporarily denote by $\tau_{i}$ the type $\tau\left(M_{i}\right)$
for $i=1,\,2$.

We choose representatives $B_{1}=\left(\begin{array}{cc}
g_{1,\,1} & g_{1,\,2}\\
-\Phi\overline{g_{1,\,2}} & \overline{g_{1,\,1}}
\end{array}\right)$ for $M_{1}$ and $B_{2}=\left(\begin{array}{cc}
g_{2,\,1} & g_{2,\,2}\\
-\Phi\overline{g_{2,\,2}} & \overline{g_{2,\,1}}
\end{array}\right)$ for $M_{2}$ such that
\begin{align*}
 & g_{1,\,1}=-\tau_{1}t^{m_{1}}+a_{1,\,1}t^{m_{1}+1}+\cdots+a_{1,\,n_{1}-m_{1}+1}t^{n_{1}+1},\\
 & g_{1,\,2}=t^{m_{1}}+b_{1,\,1}t^{m_{1}+1}+\cdots+\tau_{1}a_{1,\,n_{1}-m_{1}+1}t^{n_{1}},\\
 & g_{2,\,2}=-\tau_{2}t^{m_{2}}+a_{2,\,1}t^{m_{2}+1}+\cdots+a_{2,\,n_{2}-m_{2}+1}t^{n_{2}+1},\\
 & g_{2,\,2}=t^{m_{2}}+b_{2,\,1}t^{m_{2}+1}+\cdots+\tau_{2}a_{2,\,n_{2}-m_{2}+1}t^{n_{2}},
\end{align*}
where the highest degree coefficients and the lowest degree coefficients
are normalized according to the types and the functional equation
(15). Then, the entry $\left(B_{1}B_{2}\right)_{11}$ is computed
as
\begin{align*}
 & \left(\tau_{1}\tau_{2}t^{m_{1}+m_{2}}-\tau_{2}a_{2,\,n_{2}-m_{2}+1}t^{m_{1}-n_{2}-1}+\cdots\right.\\
 & \left.+a_{1,\,n_{1}-m_{1}+1}a_{2,\,n_{2}-m_{2}+1}t^{n_{1}+n_{2}+2}-\tau_{1}a_{1,\,n_{1}-m_{1}+1}t^{n_{1}-m_{2}+1}\right),
\end{align*}
where we expressed the terms with the lowest and the highest degrees.
The entry $\left(B_{1}B_{2}\right)_{11}$ is also computed as
\begin{align*}
 & \left(-\tau_{1}t^{m_{1}+m_{2}}+a_{2,\,n_{2}-m_{2}+1}t^{m_{1}-n_{2}-1}+\cdots\right.\\
 & \left.\tau_{2}a_{1,\,n_{1}-m_{1}+1}a_{2,\,n_{2}-m_{2}+1}t^{n_{1}+n_{2}+1}-\tau_{1}\tau_{2}a_{1,\,n_{1}-m_{1}+1}t^{n_{1}-m_{2}}\right).
\end{align*}

By comparing these terms, when $m_{2}\ne-n_{2}-1$, the matrix $B_{1}B_{2}$
is upper-balanced and $\tau\left(B_{1}B_{2}\right)=\tau_{2}$, which
proves the case (1). On the other hand, suppose $m_{2}=-n_{2}-1$.
In this case, $M_{2}$ must be evenly upper-balanced. By comparing
the coefficients, if $\tau_{1}\ne a_{2,\,n_{2}-m_{2}+1}$, the matrix
$B_{1}B_{2}$ is again upper-balanced and $\tau\left(B_{1}B_{2}\right)=\tau_{2}$,
which proves the case (2). The other cases are proven in the same
way. $\qedhere$

\noindent \end{proof}

The following are the simplest cases of Theorem 3.11.

\noindent \begin{theorem}

For a prime $p$ such that $p\equiv2,\,5$ (mod 6), the group $\mathcal{Q}_{p}$
is freely generated by the set of elementary generators $\left\{ \overline{g_{p}}\left[r\right]\right\} _{r\in\mathbb{F}_{p}}$.

\noindent \end{theorem}
\begin{proof}

\noindent The proof structure mirrors that of \citep[Theorem 4.9]{lee2024salters},
with the exception of the base fields involved. Here is a sketch of
the proof. To prove generation, for any element $g\in\mathcal{Q}_{p}$
such that $\mathrm{rd}\left(g\right)>1$, suppose $g$ is upper-balanced.
Then, we have
\begin{align*}
\mathrm{rd}\left(g\right) & >\mathrm{rd}\left(g\overline{g_{p}}\left[\tau\left(g\right)\right]^{-1}\right).
\end{align*}

On the other hand, suppose $g$ is lower-balanced. Then, we have
\begin{align*}
\mathrm{rd}\left(g\right) & >\mathrm{rd}\left(g\overline{g_{p}}\left[\tau\left(g\right)\right]\right).
\end{align*}

To prove freeness, observe that for a nonzero element $r\in\mathbb{F}_{p}$,
a nontrivial power of an elementary generator $\overline{g_{p}}\left[r\right]$
is always oddly balanced. We establish (2) of Theorem 3.11 by applying
the ping-pong lemma, considering the cases (1), (4), (6), and (7)
in Lemma 3.13. $\qedhere$

\noindent \end{proof}

In the rest of this section, we focus on the case $p=3$ of Theorem
3.11. In Definition 3.6, we listed the additive generators of $\mathcal{Q}_{3}$.
Note that for an element $\tau\in\mathbb{F}_{3}^{\times}$, the matrices
on the right-hand sides in Definition 3.6 are
\begin{align}
 & \left(\begin{array}{cc}
1+t^{-1} & 0\\
0 & 1+t
\end{array}\right)g_{3}\left[\tau\right]=\left(\begin{array}{cc}
t-t^{-1} & \tau\left(1+t^{-1}\right)\\
-\Phi\tau\left(1+t\right) & t^{-1}-t
\end{array}\right),\\
 & \left(\begin{array}{cc}
1+t & 0\\
0 & 1+t^{-1}
\end{array}\right)g_{3}\left[\tau\right]^{\dagger}=\left(\begin{array}{cc}
t^{-1}-t & -\tau\left(1+t\right)\\
\Phi\tau\left(1+t^{-1}\right) & t-t^{-1}
\end{array}\right).
\end{align}
\begin{lemma}

In the monoid of matrices $\mathrm{M}\left(2,\,\mathbb{F}_{3}\left[t,\,t^{-1}\right]\right),$
for an element $\tau\in\mathbb{F}_{3}^{\times}$, we have the following
identities:
\begin{align}
 & \left(\begin{array}{cc}
1+t^{-1} & 0\\
0 & 1+t
\end{array}\right)g_{3}\left[\tau\right]=g_{3}\left[\tau\right]^{\dagger}\left(\begin{array}{cc}
-1-t & 0\\
0 & -1-t^{-1}
\end{array}\right),\\
 & \left(\begin{array}{cc}
1+t & 0\\
0 & 1+t^{-1}
\end{array}\right)g_{3}\left[\tau\right]^{\dagger}=g_{3}\left[\tau\right]\left(\begin{array}{cc}
-1-t^{-1} & 0\\
0 & -1-t
\end{array}\right).
\end{align}
\end{lemma}
\begin{proof}

\noindent Direct by computation. $\qedhere$\end{proof}
\begin{corollary}

For an element $\tau\in\mathbb{F}_{3}^{\times}$ and a pair of polynomials
$f_{1}\left(x\right),\,f_{2}\left(x\right)\in\mathbb{F}_{3}\left[x\right]$,
we have the following identities:
\begin{align}
 & a_{3,\,\tau,\,u}\left[f_{1}\right]a_{3,\,\tau,\,u}\left[f_{2}\right]=a_{3,\,\tau,\,u}\left[f_{1}+f_{2}\right],\\
 & a_{3,\,\tau,\,l}\left[f_{1}\right]a_{3,\,\tau,\,l}\left[f_{2}\right]=a_{3,\,\tau,\,l}\left[f_{1}+f_{2}\right].
\end{align}
\end{corollary}
\begin{proof}

\noindent By definition, we have $\det\left(g_{3}\left[\tau\right]\right)=0$,
which implies $g_{3}\left[\tau\right]g_{3}\left[\tau\right]^{\dagger}=0=g_{3}\left[\tau\right]^{\dagger}g_{3}\left[\tau\right]$.
We prove the second identity (30) as follows.
\begin{align*}
 & \,a_{3,\,\tau,\,u}\left[f_{1}\right]a_{3,\,\tau,\,u}\left[f_{2}\right]\\
= & \,\left(I+f_{1}\left(t+t^{-1}\right)\left(\begin{array}{cc}
1+t^{-1} & 0\\
0 & 1+t
\end{array}\right)g_{3}\left[\tau\right]\right)\left(I+f_{2}\left(t+t^{-1}\right)\left(\begin{array}{cc}
1+t^{-1} & 0\\
0 & 1+t
\end{array}\right)g_{3}\left[\tau\right]\right)\\
= & \,a_{3,\,\tau,\,u}\left[f_{1}+f_{2}\right]+\left(f_{1}f_{2}\right)\left(t+t^{-1}\right)\left(\begin{array}{cc}
1+t^{-1} & 0\\
0 & 1+t
\end{array}\right)g_{3}\left[\tau\right]g_{3}\left[\tau\right]^{\dagger}\left(\begin{array}{cc}
-1-t & 0\\
0 & -1-t^{-1}
\end{array}\right)\\
= & \,a_{3,\,\tau,\,u}\left[f_{1}+f_{2}\right],
\end{align*}
where we used the identity (28) in Lemma 3.15 for the second eqality
and $g_{3}\left[\sigma\right]g_{3}\left[\sigma\right]^{\dagger}=0$
for the last. The case for $a_{3,\,\tau,\,l}\left[f\right]$ is proven
in the same way by using (29). $\qedhere$

\noindent \end{proof}
\begin{lemma}

For an element $\tau\in\mathbb{F}_{3}^{\times}$ and a polynomial
$f\left(x\right)\in\mathbb{F}_{3}\left[x\right]$, we have the following:
\begin{align*}
\det\left(a_{3,\,\tau,\,u}\left[f\right]\right) & =1=\det\left(a_{3,\,\tau,\,l}\left[f\right]\right).
\end{align*}
\end{lemma}
\begin{proof}

\noindent Direct by computation. $\qedhere$

\noindent \end{proof}

Corollary 3.16 and Lemma 3.17 ensure that the set of upper (or lower)
additive generators generates an abelian linear group, where every
element has determinant 1. The structure is described as follows.

\noindent \begin{corollary}

For an element $\tau\in\mathbb{F}_{3}^{\times}$, we define a map
$a_{3,\,\tau,\,u}:\mathbb{F}_{3}\left[x\right]\to\mathrm{SL}\left(2,\,\mathbb{F}_{3}\left[t,\,t^{-1}\right]\right)$
by $f\mapsto a_{3,\,\tau,\,u}\left[f\right]$, and another map $a_{3,\,\tau,\,l}:\mathbb{F}_{3}\left[x\right]\to\mathrm{SL}\left(2,\,\mathbb{F}_{3}\left[t,\,t^{-1}\right]\right)$
by $f\mapsto a_{3,\,\tau,\,l}\left[f\right]$, where $\mathbb{F}_{3}\left[x\right]$
is regarded as the group equipped with the addition. Then, the maps
$a_{3,\,\tau,\,u}$ and $a_{3,\,\tau,\,l}$ are injective group homomorphisms.
Moreover, the projectivizations $\overline{a_{3,\,\tau,\,u}}:=\pi_{3}\circ a_{3,\,\tau,\,u}$
and $\overline{a_{3,\,\tau,\,l}}:=\pi_{3}\circ a_{3,\,\tau,\,l}$
are still injective.

\noindent \end{corollary}
\begin{proof}

\noindent By the identity (30) and (31) in Corollary 3.16, the maps
$a_{3,\,\tau,\,u}$ and $a_{3,\,\tau,\,l}$ are group homomorphisms.
These are injective by definition. Lemma 3.17 implies that the speical
linear group contains the images of $a_{3,\,\tau,\,u}$ and $a_{3,\,\tau,\,l}$.

We show that $\overline{a_{3,\,\tau,\,u}}$ and $\overline{a_{3,\,\tau,\,l}}$
are also injective. Since the image of $a_{3,\,\tau,\,u}$ has determinant
1 and $a_{3,\,\tau,\,u}$ is injective, if $\overline{a_{3,\,\tau,\,u}}$
were not injective, there would exist a polynomial $f\left(x\right)\in\mathbb{F}_{3}\left[x\right]$
such that $a_{3,\,\tau,\,u}\left(f\right)=-I$. However, since we
have $\left(\left(\begin{array}{cc}
1+t^{-1} & 0\\
0 & 1+t
\end{array}\right)g_{3}\left[\tau\right]\right)_{12}\ne0$, there is no such polynomial. The case $\overline{a_{3,\,\tau,\,l}}$
is proven in the same way. $\qedhere$

\noindent \end{proof}

The next lemma deals with the construction of $\mathcal{G}_{\tau}$
satisfying (4) in Theorem 3.11, preparing for the use of Lemma 3.13.

\noindent \begin{lemma}

For an element $\tau\in\mathbb{F}_{3}^{\times}$, we have
\begin{align*}
\left\langle \overline{a_{3,\,\tau,\,u}}\left(\mathbb{F}_{3}\left[x\right]\right),\,\overline{a_{3,\,\tau,\,l}}\left(\mathbb{F}_{3}\left[x\right]\right)\right\rangle  & =\overline{a_{3,\,\tau,\,u}}\left(\mathbb{F}_{3}\left[x\right]\right)*\overline{a_{3,\,\tau,\,l}}\left(\mathbb{F}_{3}\left[x\right]\right),
\end{align*}
where $*$ on the right-hand side means the (internal) free product.
Moreover, for every element $M\in\left\langle \overline{a_{3,\,\tau,\,u}}\left(\mathbb{F}_{3}\left[x\right]\right),\,\overline{a_{3,\,\tau,\,l}}\left(\mathbb{F}_{3}\left[x\right]\right)\right\rangle $,
$M$ is balanced, $\tau\left(M\right)=\tau$, and $s\left(M\right)=-1$.\end{lemma}
\begin{proof}

\noindent We use the usual ping-pong lemma, considering the action
by right multiplication in $\mathcal{Q}_{3}$. Define $S_{\tau,\,u}$
(resp. $S_{\tau,\,l}$) to be the set of elements $M\in\mathcal{Q}_{3}$
such that $\tau\left(M\right)=\tau$, $\mathrm{rd}\left(M\right)\ge2$,
and $M$ is upper-balanced (resp. lower-balanced). By definition,
it is direct that $\overline{a_{3,\,\tau,\,u}}\left(\mathbb{F}_{3}\left[x\right]\right)\subset S_{\tau,\,u}$
and $\overline{a_{3,\,\tau,\,l}}\left(\mathbb{F}_{3}\left[x\right]\right)\subset S_{\tau,\,l}$.
For a nonzero polynomial $f\left(x\right)\in\mathbb{F}_{3}\left[x\right]$
and for elements $M_{1}\in S_{\tau,\,l}$, $M_{2}\in S_{\tau,\,u}$,
to establish the free product, it suffices to show that
\begin{align}
 & M_{1}\overline{a_{3,\,\tau,\,u}}\left[f\right]\in S_{\tau,\,u},\\
 & M_{2}\overline{a_{3,\,\tau,\,l}}\left[f\right]\in S_{\tau,\,l}.
\end{align}

Take a representative $B=\left(\begin{array}{cc}
g_{1} & g_{2}\\
-\Phi\overline{g_{2}} & \overline{g_{1}}
\end{array}\right)$ of $M_{1}$, such that the first row entries are
\begin{align*}
 & g_{1}=t^{m_{1}-1}+\cdots+\tau bt^{n_{1}},\\
 & g_{2}=-\tau t^{m_{1}}+\cdots+bt^{n_{1}},
\end{align*}
where the coefficients are normalized for some $b\in\mathbb{F}_{3}^{\times}$
and rewritten satisfying the functional equation (15). Suppose $\deg\left(f\right)=n$
for some nonnegative integer $n$. For some unit $u\in\mathbb{F}_{3}^{\times}$,
we may write $f\left(t+t^{-1}\right)=u\left(t^{n}+\cdots+t^{-n}\right)$.
Expanding, we have
\begin{align*}
 & \,Ba_{3,\,\tau,\,u}\left(f\right)\\
= & \,B\left(I+u\left(t^{n}+\cdots+t^{-n}\right)\left(\begin{array}{cc}
t-t^{-1} & \tau\left(1+t^{-1}\right)\\
-\Phi\tau\left(1+t\right) & t^{-1}-t
\end{array}\right)\right)\\
= & \,u\left(\begin{array}{cc}
-t^{m_{1}-n-2}+\cdots-\tau bt^{n_{1}+n+2} & *\\
\tau t^{m_{1}-n-2}+\cdots-bt^{n_{1}+n+1} & *
\end{array}\right)^{T},
\end{align*}
where we computed the lowest and the highest degree terms. Thus, $Ba_{3,\,\tau,\,u}\left[f\right]$
is upper-balanced and we establish (32). (33) is proven in the same
way.

Moreover, in the computation above, we see that the signature remains
invariant as $\tau b$. Given that for any additive generator has
signature $-1$, we conclude the proof by induction, based on the
normal form of free products. $\qedhere$

\noindent \end{proof}

We are ready to prove Theorem 3.11 for the case $p=3$.

\noindent \begin{theorem}

The quaternionic group $\mathcal{Q}_{3}$ is the free product of three
subgroups $\mathcal{G}_{3,\,0}$, $\mathcal{G}_{3,\,1}$ and $\mathcal{G}_{3,\,2}$,
where $\mathcal{G}_{3,\,0}=\left\langle \overline{g_{3}\left[0\right]}\right\rangle $
and $\mathcal{G}_{3,\,\tau}=\left\langle a_{3,\,\tau,\,u}\left(\mathbb{F}_{3}\left[x\right]\right),\,a_{3,\,\tau,\,l}\left(\mathbb{F}_{3}\left[x\right]\right)\right\rangle $
for $\tau=1,\,2$. Moreover, for $\tau=1,\,2$, the subgroup $\mathcal{G}_{3,\,\tau}$
is the free product of two subgroups $a_{3,\,\tau,\,u}\left(\mathbb{F}_{3}\left[x\right]\right)$
and $a_{3,\,\tau,\,l}\left(\mathbb{F}_{3}\left[x\right]\right)$,
both of which are isomorphic to the additive group $\mathbb{F}_{3}\left[x\right]$.

\noindent \end{theorem}
\begin{proof}

\noindent By Lemma 3.19, we already know that $\mathcal{G}_{3,\,\tau}$
is the free product of $a_{3,\,\tau,\,u}\left(\mathbb{F}_{3}\left[x\right]\right)$
and $a_{3,\,\tau,\,l}\left(\mathbb{F}_{3}\left[x\right]\right)$ for
$\tau=1,\,2$. Thus, it suffices to show that 
\begin{description}
\item [{(a)}] \noindent the subgroups $\mathcal{G}_{3,\,0}$, $\mathcal{G}_{3,\,1}$
and $\mathcal{G}_{3,\,2}$ generate the whole group $\mathcal{Q}_{3}$,
and 
\item [{(b)}] \noindent the subgroup generated by $\mathcal{G}_{3,\,0}$,
$\mathcal{G}_{3,\,1}$ and $\mathcal{G}_{3,\,2}$ is the free product
of them.
\end{description}
\noindent \begin{proofofa}

We use the induction on the relative degree. Take an element $M\in\mathcal{Q}_{3}$
such that $\mathrm{rd}\left(M\right)=1$ and choose a representative
$B$. We may assume that $B$ is balanced by replacing it by the balanced
companion by Lemma 3.3. By normalizing the entry $B_{11}$ multiplying
by some scalar $t^{l_{0}}$, the only possible candidates for the
representative are
\begin{align*}
B & =\left(\begin{array}{cc}
a_{0}+a_{1}t & b_{0}t^{l}\\
\cdots & \cdots
\end{array}\right),
\end{align*}
where the integer $l$ is 0 if $B$ is upper-balnced, and is $1$
if lower-balanced. Suppose $l=0$. From the functional equation (15),
for some unit $\sigma\in\mathbb{F}_{3}^{\times}$, we have
\begin{align*}
\sigma= & \,\left(a_{0}+a_{1}t\right)\left(a_{0}+a_{1}t^{-1}\right)+\left(1+t+t^{-1}\right)b_{0}^{2}\\
= & \,\left(a_{0}a_{1}+b_{0}^{2}\right)\left(t+t^{-1}\right)+a_{0}^{2}+a_{1}^{2}+b_{0}^{2},
\end{align*}
from which we deduce $a_{0}a_{1}+b_{0}^{2}=0$ by comparing the coefficients
in both sides. By substituting $a_{0}=-b_{0}^{2}a_{1}^{-1}$, we have
\begin{align*}
a_{1}^{-1}B & =\left(\begin{array}{cc}
t-\left(b_{0}a_{1}^{-1}\right)^{2} & \left(b_{0}a_{1}^{-1}\right)\\
* & *
\end{array}\right).
\end{align*}

For any $\left(b_{0}a_{1}^{-1}\right)\in\mathbb{F}_{3}^{\times}$,
its square is 1, which implies $1+\left(b_{0}a_{1}^{-1}\right)^{2}+\left(b_{0}a_{1}^{-1}\right)^{4}$
is 0. Therefore, there is no element $M\in\mathcal{Q}_{3}$ such that
$\mathrm{rd}\left(M\right)=1$. Direct computation shows that the
only elements $M\in\mathcal{Q}_{3}$ such that $\mathrm{rd}\left(M\right)=2$
are exactly $\overline{a_{3,\,1,\,u}}\left(\pm1\right)$, $\overline{a_{3,\,1,\,l}}\left(\pm1\right)$,
$\overline{a_{3,\,2,\,u}}\left(\pm1\right)$, $\overline{a_{3,\,2,\,l}}\left(\pm1\right)$,
and some product of them with $\overline{g_{3}}\left[0\right]$.

For an integer $N\ge2$, suppose that the subgroups $\mathcal{G}_{0}$,
$\mathcal{G}_{1}$ and $\mathcal{G}_{2}$ generate every element $M_{0}\in\mathcal{Q}_{3}$
such that $\mathrm{rd}\left(M_{0}\right)\le N$, and suppose $M\in\mathcal{Q}_{3}$
such that $\mathrm{rd}\left(M\right)=N+1$. Choose a representative
$B$ of $M$ in the form
\begin{align*}
B & =\left(\begin{array}{cc}
g_{1} & g_{2}\\
-\Phi\overline{g_{2}} & \overline{g_{1}}
\end{array}\right).
\end{align*}

We assume that $B$ is balanced by replacing $B$ with the balanced
companion by Lemma 3.3. Suppose $B$ is lower-balanced and $\tau\left(B\right)=\tau$
for a unit $\tau\in\mathbb{F}_{3}^{\times}$. We write its entries
as
\begin{align*}
 & g_{1}=t^{m_{1}-1}+a_{1}t^{m_{1}}+\cdots+a_{2}t^{n_{1}-1}+\tau bt^{n_{1}},\\
 & g_{2}=-\tau t^{m_{1}}+a_{3}t^{m_{1}+1}+\cdots+a_{4}t^{n_{1}-1}+bt^{n_{1}},
\end{align*}
where for some $b\in\mathbb{F}_{3}^{\times}$, we normalized the first
coefficient of $g_{1}$ as 1.

A singular matrix $Bg_{3}\left[\tau\right]=B\left(\begin{array}{cc}
t-1 & \tau\\
-\Phi\tau & t^{-1}-1
\end{array}\right)$ is given by
\begin{align*}
Bg_{3}\left[\tau\right] & =\left(\begin{array}{cc}
\left(-a_{1}-1-\tau a_{3}\right)t^{m_{1}}+\cdots+\left(a_{2}+\tau b-\tau a_{4}\right)t^{n_{1}} & *\\
\left(\tau a_{1}+\tau+a_{3}\right)t^{m_{1}}+\cdots+\left(\tau a_{2}-a_{4}+b\right)t^{n_{1}-1} & *
\end{array}\right)^{T},
\end{align*}
where we have cancelled the highest degree terms and the lowest degree
terms once, and the relative degree has decreased. We claim that
\begin{align}
\tau b\left(-a_{1}-1-\tau a_{3}\right) & =\left(a_{2}+\tau b-\tau a_{4}\right).
\end{align}

From the functional equation (15) for $B$, for some unit $u\in\mathbb{F}_{3}^{\times}$,
we have
\begin{align*}
u= & \,g_{1}\overline{g_{1}}+\Phi g_{2}\overline{g_{2}}\\
= & \,\cdots+\left(a_{2}+\tau ba_{1}-\tau b-\tau a_{4}+ba_{3}\right)\left(t^{n_{1}-m_{1}}+t^{m_{1}-n_{1}}\right),
\end{align*}
which implies that
\[
a_{2}+\tau ba_{1}-\tau b-\tau a_{4}+ba_{3}=0,
\]
by comparing the coefficients in both sides. Therefore, we have (34)
by rearrangement. Suppose that $-a_{1}-1-\tau a_{3}\ne0$. Then, we
claim that
\begin{align}
\mathrm{rd}\left(Ba_{3,\,\tau,\,l}\left(\left(-a_{1}-1-\tau a_{3}\right)^{-1}\right)\right) & \le N-1.
\end{align}

To prove the claim, by expanding we have
\begin{align*}
 & \,\,Ba_{3,\,\tau,\,l}\left(\left(-a_{1}-1-\tau a_{3}\right)^{-1}\right)\\
= & \,B\left(I+\left(-a_{1}-1-\tau a_{3}\right)^{-1}g_{3}\left[\tau\right]\left(\begin{array}{cc}
-1-t^{-1} & 0\\
0 & -1-t
\end{array}\right)\right)\\
= & B+\left(-a_{1}-1-\tau a_{3}\right)^{-1}Bg_{3}\left[\tau\right]\left(\begin{array}{cc}
-1-t^{-1} & 0\\
0 & -1-t
\end{array}\right),
\end{align*}
where the first equality follows from Lemma 3.7. Thus, we have the
inequality (35) by comparing the coefficients of the highest degree
terms and the lowest degree terms for each entry.

On the other hand, suppose that $-a_{1}-1-\tau a_{3}=0$. Consider
the Laurent polynomial entry
\begin{align*}
\left(Bg_{3}\left[\tau\right]\right)_{11} & =\tau_{l}t^{m_{0}}+\cdots+\tau_{h}t^{n_{0}},
\end{align*}
where $m_{0}$ is the power of $t$ in the lowest degree term, $n_{0}$
is the power of $t$ in the highest degree term, and the coefficients
$\tau_{l},\,\tau_{h}$ are units. Denote by $m$ the integer
\begin{align*}
\min\left(\left(m_{0}-m_{1}\right),\,\left(n_{1}-n_{0}\right)\right),
\end{align*}
satisfying $m\ge1$ by the equality (34). Without loss of generality,
suppose $m$ is equal to $\left(m_{0}-m_{1}\right)$. Then, the $\left(1,\,1\right)$-entry
of a matrix $\tau_{l}^{-1}\left(t^{-m}+t^{m}\right)Bg_{3}\left[\tau\right]$
has $t^{m_{1}}$ as the lowest degree term. Choose a polynomial $f\left(x\right)\in\mathbb{F}_{3}\left[x\right]$
such that $f\left(t+t^{-1}\right)=\tau_{l}^{-1}\left(t^{-m+1}+t^{m-1}\right)$.
As in the proof of (35), we have
\begin{align*}
\mathrm{rd}\left(Ba_{3,\,\tau,\,l}\left(f\right)\right) & \le N,
\end{align*}
which concludes the proof. The case where $B$ is upper-balanced is
proven in the same way. $\qed$

\noindent \end{proofofa}
\begin{proofofb}

\noindent We use the usual ping-pong lemma, considering the action
by right multiplication in $\mathcal{Q}_{3}$. For each unit $\tau\in\mathbb{F}_{3}^{\times}$,
define $S_{\tau}$ to be the set of balanced elements $M\in\mathcal{Q}_{3}$
such that $\tau\left(M\right)=\tau$, and define $S_{0}$ to be the
set of unbalanced elements in $\mathcal{Q}_{3}$. Lemma 3.19 ensures
that for each $\tau\in\mathbb{F}_{3}$, $\mathcal{G}_{3,\,\tau}\subset S_{\tau}$.
The cases involving $S_{0}$ are trivial, so we consider only balanced
matrices. Without loss of generality, for elements $M_{1}\in S_{1}$
and $M_{2}\in\mathcal{G}_{3,\,2}$, it suffices to show that
\begin{align}
 & M_{1}M_{2}\in S_{2}.
\end{align}

Since $\tau\left(M_{1}\right)=1$ and $\tau\left(M_{2}\right)=2$,
we have
\begin{align*}
-\tau\left(M_{1}\right)\tau\left(M_{2}\right) & =1=-\frac{\tau\left(M_{1}\right)}{\tau\left(M_{2}\right)}.
\end{align*}

However, we have $\sigma\left(M_{2}\right)=-1$ by Lemma 3.19, which
satisfies the conditions (2), (4), (6), (8) on signature in Lemma
3.13. Therefore, by Lemma 3.13, we have (36) for all cases. The case
where $M_{1}$ is upper-balanced is proven in the same way. $\qedhere$

\noindent \end{proofofb}
\end{proof}

Before concluding this section, we present the postponed proof of
Lemma 2.18, as announced in Section 2.

\noindent \begin{prooflemma218}

By Theorem 3.20, the quaternionic group $\mathcal{Q}_{3}$ is generated
by $\overline{g_{3}}\left[0\right]$ and the additive generators.
Note that $\det\left(g_{3}\left[0\right]\right)=1$ by (16). By Lemma
3.17, every additive generator has a representative with determinant
1. Therefore, every element $M\in\mathcal{Q}_{3}$ has a representative
with determinant 1. $\qed$

\noindent \end{prooflemma218}

\section{The Cases $p\equiv1$ (mod 6)}

Throughout this section, unless otherwise stated, we assume that the
prime $p$ satisfies $p\equiv1$ (mod 6). For any $r\in\mathbb{F}_{p}$
such that $r^{4}+r^{2}+1\ne0$, define the \emph{type group} $\mathcal{G}_{p,\,r}$
to be the infinite cyclic group $\left\langle \overline{g_{p}}\left[r\right]\right\rangle $.
For any $\tau\in\mathbb{F}_{p}$ such that $\tau^{4}+\tau^{2}+1=0$,
define the \emph{type group} $\mathcal{G}_{p,\,\tau}$ to be generated
by all affine generators of type $\tau$ and a stable generator $\overline{s_{p,\,\tau}}\left[a\right]$
for any element $a\in\mathbb{F}_{p}\backslash\left\{ 3\tau+\tau^{-1}\right\} $.
The goal of this section is to prove the following.

\noindent \begin{theorem}

The quaternionic group $\mathcal{Q}_{p}$ is the free product of type
groups $\mathcal{G}_{p,\,0},\cdots,\,\mathcal{G}_{p,\,p-1}$. For
any $\tau\in\mathbb{F}_{p}$ such that $\tau^{4}+\tau^{2}+1=0$, the
set of upper (or lower) additive generators for $\tau$ generates
an abelian group isomorphic to $\mathbb{F}_{p}\left[x\right]$, and
the set of upper (or lower) multiplicative generators for $\tau$
generates another abelian group isomorphic to $\mathbb{F}_{p}^{\times}$.
The set of upper (or lower) affine generators for $\tau$ generates
a group isomoprhic to $\mathrm{Aff}\left(\mathbb{F}_{p}\left[x\right]\right)\cong\mathbb{F}_{p}\left[x\right]\rtimes\mathbb{F}_{p}^{\times}$,
where a multiplicative generator acts on additive generators by multiplication
by a unit in $\mathbb{F}_{p}^{\times}$. The group generated by all
affine generators for $\tau$ is the free product of the group generated
by the upper ones and the group generated by the lower ones.

Finally, the type group $\mathcal{G}_{p,\,\tau}$ is the (internal)
HNN extension isomorphic to
\begin{align*}
\mathcal{G}_{p,\,\tau} & \cong\left(\mathrm{Aff}\left(\mathbb{F}_{p}\left[x\right]\right)*\mathrm{Aff}\left(\mathbb{F}_{p}\left[x\right]\right)\right)*_{\alpha_{\tau}},
\end{align*}
where the free product inside corresponds to the group generated by
the affine generators, and the stable letter is the stable generator
$\overline{s_{p,\,\tau}}\left[a\right]$ such that for any $b\in\mathbb{F}_{p}\backslash\left\{ \pm\left(\tau+2\tau^{-1}\right)\right\} $,
inducing the isomorphism corresponding to $\alpha_{\tau}$ by
\begin{align*}
\overline{s_{p,\,\tau}}\left[a\right]\overline{m_{p,\,\tau,\,u}}\left[b\right]\overline{s_{p,\,\tau}}\left[a\right]^{-1} & =\overline{m_{p,\,\tau,\,l}}\left[b\right].
\end{align*}

The results are independent of the choice of $a\in\mathbb{F}_{p}\backslash\left\{ 3\tau+\tau^{-1}\right\} $
for $\overline{s_{p,\,\tau}}\left[a\right]$.

\noindent \end{theorem}

To prove Theorem 4.1, we begin by establishing the structure of the
subgroups $\mathcal{G}_{p,\,\tau}$. The process for handling the
additive generators is similar to the case $p=3$.

\noindent \begin{lemma}

For an element $\tau\in\mathbb{F}_{p}$ such that $\tau^{4}+\tau^{2}+1=0$
and a polynomial $f\left(x\right)\in\mathbb{F}_{p}\left[x\right]$,
we have the following:
\begin{align*}
\det\left(a_{p,\,\tau,\,u}\left[f\right]\right) & =1=\det\left(a_{p,\,\tau,\,l}\left[f\right]\right).
\end{align*}
\end{lemma}
\begin{proof}

\noindent Direct by computation. $\qedhere$

\noindent \end{proof}
\begin{lemma}

For an element $\tau\in\mathbb{F}_{p}$ such that $\tau^{4}+\tau^{2}+1=0$
and for a pair of polynomials $f_{1}\left(x\right),\,f_{2}\left(x\right)\in\mathbb{F}_{p}\left[x\right]$,
we have the following identities:
\begin{align}
 & a_{p,\,\tau,\,u}\left[f_{1}\right]a_{p,\,\tau,\,u}\left[f_{2}\right]=a_{p,\,\tau,\,u}\left[f_{1}+f_{2}\right],\\
 & a_{p,\,\tau,\,l}\left[f_{1}\right]a_{p,\,\tau,\,l}\left[f_{2}\right]=a_{p,\,\tau,\,l}\left[f_{1}+f_{2}\right].
\end{align}
\end{lemma}
\begin{proof}

\noindent By definition, we have $\det\left(g_{p}\left[\tau\right]\right)=0$,
which implies $g_{p}\left[\tau\right]g_{p}\left[\tau\right]^{\dagger}=0=g_{p}\left[\tau\right]^{\dagger}g_{p}\left[\tau\right]$.
We prove the first identity (37) as follows.
\begin{align*}
 & \,a_{p,\,\tau,\,u}\left[f_{1}\right]a_{p,\,\tau,\,u}\left[f_{2}\right]\\
= & \,\left(I+f_{1}\left(t+t^{-1}\right)g_{p}\left[\tau\right]^{\dagger}\left(\begin{array}{cc}
t-t^{-1} & 0\\
0 & 0
\end{array}\right)g_{p}\left[\tau\right]\right)\cdot\\
 & \left(I+f_{2}\left(t+t^{-1}\right)g_{p}\left[\tau\right]^{\dagger}\left(\begin{array}{cc}
t-t^{-1} & 0\\
0 & 0
\end{array}\right)g_{p}\left[\tau\right]\right)\\
= & \,a_{p,\,\tau,\,u}\left[f_{1}+f_{2}\right]+\left(f_{1}f_{2}\right)\left(t+t^{-1}\right)g_{p}\left[\tau\right]^{\dagger}\left(\begin{array}{cc}
t-t^{-1} & 0\\
0 & 0
\end{array}\right)g_{p}\left[\tau\right]g_{p}\left[\tau\right]^{\dagger}\left(\begin{array}{cc}
t-t^{-1} & 0\\
0 & 0
\end{array}\right)g_{p}\left[\tau\right]\\
= & \,a_{p,\,\tau,\,u}\left[f_{1}+f_{2}\right],
\end{align*}
where we used the identity $g_{p}\left[\tau\right]g_{p}\left[\tau\right]^{\dagger}=0$
for the last equality. The case (38) for $a_{p,\,\tau,\,l}\left[f\right]$
is proven in the same way. $\qedhere$

\noindent \end{proof}
\begin{corollary}

For an element $\tau\in\mathbb{F}_{p}$ such that $\tau^{4}+\tau^{2}+1=0$,
we define a map
\begin{align*}
a_{p,\,\tau,\,u} & :\mathbb{F}_{p}\left[x\right]\to\mathrm{SL}\left(2,\,\mathbb{F}_{p}\left[t,\,t^{-1}\right]\right)
\end{align*}
by $f\mapsto a_{p,\,\tau,\,u}\left[f\right]$, and another map
\begin{align*}
a_{p,\,\tau,\,l} & :\mathbb{F}_{p}\left[x\right]\to\mathrm{SL}\left(2,\,\mathbb{F}_{p}\left[t,\,t^{-1}\right]\right)
\end{align*}
by $f\mapsto a_{p,\,\tau,\,l}\left[f\right]$. Then, the maps $a_{p,\,\tau,\,u}$
and $a_{p,\,\tau,\,l}$ are injective group homomorphisms. The projectivizations
$\overline{a_{p,\,\tau,\,u}}:=\pi_{p}\circ a_{p,\,\tau,\,u}$ and
$\overline{a_{p,\,\tau,\,l}}:=\pi_{p}\circ a_{p,\,\tau,\,l}$ are
still injective. Let us call the image $a_{p,\,\tau,\,u}\left(\mathbb{F}_{p}\left[x\right]\right)$
(resp. $a_{p,\,\tau,\,l}\left(\mathbb{F}_{p}\left[x\right]\right)$)
the \emph{upper} (resp. \emph{lower}) \emph{additive subgroup} for
$\tau$.

\noindent \end{corollary}
\begin{proof}

\noindent Direct by using Lemma 4.2 and Lemma 4.3, as in the proof
of Corollary 3.18. $\qedhere$

\noindent \end{proof}

We turn our attention to the multiplicative generators.

\noindent \begin{lemma}

Suppose an element $\tau\in\mathbb{F}_{p}$ satisfies $\tau^{4}+\tau^{2}+1=0$.
Define a map
\begin{align*}
h_{\tau} & :\mathbb{F}_{p}\backslash\left\{ \pm\left(\tau+2\tau^{-1}\right)\right\} \to\mathbb{F}_{p}^{\times}\backslash\left\{ 1\right\} 
\end{align*}
by $a\mapsto\frac{a-\left(\tau+2\tau^{-1}\right)}{a+\left(\tau+2\tau^{-1}\right)}$.
Define maps
\begin{align*}
 & \overline{f_{p,\,\tau,\,u}}:\overline{m_{p,\,\tau,\,u}}\left[\mathbb{F}_{p}\backslash\left\{ \pm\left(\tau+2\tau^{-1}\right)\right\} \right]\cup\left\{ \left[\begin{array}{cc}
1 & 0\\
0 & 1
\end{array}\right]\right\} \to\mathbb{F}_{p}^{\times},\\
 & \overline{f_{p,\,\tau,\,l}}:\overline{m_{p,\,\tau,\,l}}\left[\mathbb{F}_{p}\backslash\left\{ \pm\left(\tau+2\tau^{-1}\right)\right\} \right]\cup\left\{ \left[\begin{array}{cc}
1 & 0\\
0 & 1
\end{array}\right]\right\} \to\mathbb{F}_{p}^{\times},
\end{align*}
as $\left[\begin{array}{cc}
1 & 0\\
0 & 1
\end{array}\right]\mapsto1$, $\overline{m_{p,\,\tau,\,l}}\left[a\right]\mapsto h_{\tau}\left(a\right)$
(resp. $\left[\begin{array}{cc}
1 & 0\\
0 & 1
\end{array}\right]\mapsto1$, $\overline{m_{p,\,\tau,\,u}}\left[a\right]\mapsto h_{\tau}\left(a\right)$).
Then, $\overline{f_{p,\,\tau,\,l}}$ and $\overline{f_{p,\,\tau,\,u}}$
are group isomoprhisms. Let us call the domain of $\overline{f_{p,\,\tau,\,u}}$
(resp. $\overline{f_{p,\,\tau,\,l}}$) the \emph{upper} (resp. \emph{lower})
\emph{multiplicative subgroup} for $\tau$.

\noindent \end{lemma}
\begin{proof}

\noindent For $a_{1},\,a_{2}\in\mathbb{F}_{p}\backslash\left\{ \pm\left(\tau+2\tau^{-1}\right)\right\} $,
by calculating the product of two upper multiplicative generators
by representatives, we have
\begin{align*}
 & \,m_{p,\,\tau,\,u}\left[a_{1}\right]m_{p,\,\tau,\,u}\left[a_{2}\right]\\
= & \,\left(\begin{array}{cc}
\frac{-\tau+a_{1}t+\tau t^{2}}{t} & \frac{1+\tau^{2}t}{t}\\
-\Phi\left(\tau^{2}+t\right) & \frac{\tau+a_{1}t-\tau t^{2}}{t}
\end{array}\right)\left(\begin{array}{cc}
\frac{-\tau+a_{2}t+\tau t^{2}}{t} & \frac{1+\tau^{2}t}{t}\\
-\Phi\left(\tau^{2}+t\right) & \frac{\tau+a_{2}t-\tau t^{2}}{t}
\end{array}\right)\\
= & \,\left(\begin{array}{cc}
\left(a_{1}+a_{2}\right)\left(-\tau t^{-1}+\tau t\right)+a_{1}a_{2}-3\tau^{2} & *\\
\left(a_{1}+a_{2}\right)\left(t^{-1}+t\right) & *
\end{array}\right)^{T},
\end{align*}
which implies that when $\frac{a_{1}a_{2}-3\tau^{2}}{a_{1}+a_{2}}\ne\pm\left(\tau+2\tau^{-1}\right)$,

\noindent 
\begin{align}
 & \overline{m_{p,\,\tau,\,u}}\left[a_{1}\right]\overline{m_{p,\,\tau,\,u}}\left[a_{2}\right]=\overline{m_{p,\,\tau,\,u}}\left(\frac{a_{1}a_{2}-3\tau^{2}}{a_{1}+a_{2}}\right),\;a_{1}\ne-a_{2},\\
 & \overline{m_{p,\,\tau,\,u}}\left[a_{1}\right]\overline{m_{p,\,\tau,\,u}}\left[-a_{1}\right]=\left[\begin{array}{cc}
1 & 0\\
0 & 1
\end{array}\right].
\end{align}

Furthermore, the cases where $\frac{a_{1}a_{2}-3\tau^{2}}{a_{1}+a_{2}}=\pm\left(\tau+2\tau^{-1}\right)$
are impossible. It is because by using the assumption $\left(\tau+2\tau^{-1}\right)^{2}=-3\tau^{2}$,
the condition $\frac{a_{1}a_{2}-3\tau^{2}}{a_{1}+a_{2}}=\pm\left(\tau+2\tau^{-1}\right)$
under the condition $a_{1}\ne-a_{2}$ is simplified to
\begin{align*}
\left(a_{1}\pm\left(\tau+2\tau^{-1}\right)\right)\left(a_{2}\pm\left(\tau+2\tau^{-1}\right)\right) & =0.
\end{align*}

In summary, the equalities (39) and (40) imply that there is a multiplication
on the set $\overline{m_{p,\,\tau,\,l}}\left[\mathbb{F}_{p}\backslash\left\{ \pm\left(\tau+2\tau^{-1}\right)\right\} \right]\cup\left\{ \left[\begin{array}{cc}
1 & 0\\
0 & 1
\end{array}\right]\right\} $ inherited from the usual matrix multiplication.

On the other hand, for $a_{1},\,a_{2}\in\mathbb{F}_{p}\backslash\left\{ \pm\left(\tau+2\tau^{-1}\right)\right\} $
such that $a_{1}\ne-a_{2}$, the map
\begin{align*}
h_{\tau}\left(a\right) & =\frac{a-\left(\tau+2\tau^{-1}\right)}{a+\left(\tau+2\tau^{-1}\right)}
\end{align*}
gives that
\begin{align*}
 & \,h_{\tau}\left(\frac{a_{1}a_{2}-3\tau^{2}}{a_{1}+a_{2}}\right)\\
= & \,\frac{a_{1}a_{2}+\left(\tau+2\tau^{-1}\right)^{2}-\left(a_{1}+a_{2}\right)\left(\tau+2\tau^{-1}\right)}{a_{1}a_{2}+\left(\tau+2\tau^{-1}\right)^{2}+\left(a_{1}+a_{2}\right)\left(\tau+2\tau^{-1}\right)}\\
= & \,\left(\frac{a_{1}-\left(\tau+2\tau^{-1}\right)}{a_{1}+\left(\tau+2\tau^{-1}\right)}\right)\left(\frac{a_{2}-\left(\tau+2\tau^{-1}\right)}{a_{2}+\left(\tau+2\tau^{-1}\right)}\right),\\
= & \,h_{\tau}\left(a_{1}\right)h_{\tau}\left(a_{2}\right),
\end{align*}
which establishes that $\overline{f_{\tau,\,u}}$ is a group homomorphism.
The injectivity and the surjectivity are canonical, and we conclude
the proof for $\overline{m_{p,\,\tau,\,u}}$. The case for $\overline{m_{p,\,\tau,\,l}}$
is proven in the same way, deriving the same equalities as (39) and
(40) only with the subscript $u$ replaced with the subscript $l$.
$\qedhere$

\noindent \end{proof}
\begin{lemma}

Suppose an element $\tau\in\mathbb{F}_{p}$ satisfies $\tau^{4}+\tau^{2}+1=0$.
For $a\in\mathbb{F}_{p}\backslash\left\{ \pm\left(\tau+2\tau^{-1}\right)\right\} $,
we have the following identities:
\begin{align*}
 & \overline{m_{p,\,\tau,\,u}}\left[a\right]\overline{a_{p,\tau,\,u}}\left[f\right]\overline{m_{p,\,\tau,\,u}}\left[a\right]^{-1}=\overline{a_{p,\,\tau,\,u}}\left[h_{\tau}\left(a\right)f\right],\\
 & \overline{m_{p,\,\tau,\,l}}\left[a\right]\overline{a_{p,\,\tau,\,l}}\left[f\right]\overline{m_{p,\,\tau,\,l}}\left[a\right]^{-1}=\overline{a_{p,\,\tau,\,l}}\left[h_{\tau}\left(a\right)^{-1}f\right],
\end{align*}
where $h_{\tau}\left(a\right)=\frac{a-\left(\tau+2\tau^{-1}\right)}{a+\left(\tau+2\tau^{-1}\right)}$,
defined in Lemma 4.5. In particular, the subgroup generated by the
upper (resp. lower) multiplicative subgroup for $\tau$ and the upper
(resp. lower) additive subgroup for $\tau$ is an inner semidirect
product isomorphic to $\mathrm{Aff}\left(\mathbb{F}_{p}\left[x\right]\right)$.
Let us refer to the semidirect product as the \emph{upper} (resp.
\emph{lower}) \emph{affine subgroup} for $\tau$, denoted by $\mathrm{Aff}_{p,\,\tau,\,u}$
(resp. $\mathrm{Aff}_{p,\,\tau,\,l}$).

\noindent \end{lemma}
\begin{proof}

\noindent Direct by computation. $\qedhere$

\noindent \end{proof}
\begin{definition}

Suppose an element $\tau\in\mathbb{F}_{p}$ satisfies $\tau^{4}+\tau^{2}+1=0$.
Choose an element $M\in\mathcal{Q}_{p}$ such that $\tau\left(M\right)=\tau$
and $\mathrm{rd}\left(M\right)\ge3$. Choose a representative $B=\left(\begin{array}{cc}
g_{1} & g_{2}\\
-\Phi\overline{g_{2}} & \overline{g_{1}}
\end{array}\right)$ of $M$, and write its entries as
\begin{align*}
 & g_{1}=a_{0}t^{m_{1}}+a_{1}t^{m_{1}+1}+\cdots+a_{n_{1}-m_{1}-1}t^{n_{1}-1}+a_{n_{1}-m_{1}}t^{n_{1}},\\
 & g_{2}=b_{0}t^{m_{2}}+b_{1}t^{m_{2}+1}+\cdots+b_{n_{2}-m_{2}-1}t^{n_{2}-1}+b_{n_{2}-m_{2}}t^{n_{2}}.
\end{align*}

Suppose that $B$ is upper-balanced. Define the \emph{second-order
type }of $B$ (or $M$) for $\tau$ as
\begin{align*}
\tau2_{\tau}\left(B\right) & :=\left(a_{1}+\tau b_{1}\right)b_{0}^{-1}+\tau.
\end{align*}

On the other hand, suppose that $B$ is lower-balanced. Define the
\emph{second-order type} of $B$ (or $M$) for $\tau$ as
\begin{align*}
\tau2_{\tau}\left(B\right) & :=\left(b_{1}+\tau a_{1}\right)a_{0}^{-1}.
\end{align*}

Moreover, define the \emph{second-order signature} of $B$ (or $M$)
for $\tau$ as
\begin{itemize}
\item $\sigma2_{\tau}\left(B\right):=\left(-a_{n_{1}-m_{1}-1}+b_{1}\tau\right)b_{0}^{-1}+\tau,$
when $B$ is evenly upper-balanced,
\item $\sigma2_{\tau}\left(B\right):=\left(a_{n_{1}-m_{1}-1}\tau-b_{1}\right)\tau b_{0}^{-1},$
when $B$ is oddly upper-balanced,
\item $\sigma2_{\tau}\left(B\right):=\left(a_{n_{1}-m_{1}-1}\tau^{-1}-b_{1}\right)a_{0}^{-1}+\tau,$
when $B$ is oddly lower-balanced,
\item $\sigma2_{\tau}\left(B\right):=\left(-a_{n_{1}-m_{1}-1}\tau+b_{1}\right)a_{0}^{-1},$
when $B$ is evenly lower-balanced.
\end{itemize}
\noindent \end{definition}

We introduced the two values in Definition 4.7 to distinguish between
various generators of the same type $\tau$ such that $\tau^{4}+\tau^{2}+1=0$.
The second-order type is particularly useful when we aim to algorithmically
express an element $M$ having type $\tau$ using affine generators
and a stable generator (Lemma 4.18). The introduction of the second-order
signature serves a technical purpose to establish the internal HNN
extension (Lemma 4.16).

For example, for an element $\tau\in\mathbb{F}_{p}$ such that $\tau^{4}+\tau^{2}+1=0$
and for a nonzero polynomial $f$, the matrices in Definition 3.7
are
\begin{align}
 & a_{p,\,\tau,\,u}\left[f\right]=\nonumber \\
 & I+f\tau^{2}\left(\begin{array}{cc}
\left(t^{-2}+t^{-1}-t-t^{2}\right) & \left(-\tau^{-1}t^{-2}+\tau t^{-1}+\tau^{-1}-\tau t\right)\\
-\Phi\tau\left(-\tau t^{-1}+\tau^{-1}+\tau t-\tau^{-1}t^{2}\right) & \left(-t^{-2}-t^{-1}+t+t^{2}\right)
\end{array}\right),\\
 & a_{p,\,\tau,\,l}\left[f\right]=\nonumber \\
 & I+f\tau^{2}\left(\begin{array}{cc}
\left(t^{-2}+t^{-1}-t-t^{2}\right) & \left(-\tau t^{-1}+\tau^{-1}+\tau t-\tau^{-1}t^{2}\right)\\
-\Phi\tau\left(-\tau^{-1}t^{-2}+\tau t^{-1}+\tau^{-1}-\tau t\right) & \left(-t^{-2}-t^{-1}+t+t^{2}\right)
\end{array}\right).
\end{align}

By (41) and (42), both has type $\tau$, signature $-1$, second-order
type $-\tau^{3}$, and second-order signature $\tau+\tau^{-1}$.

On the other hand, for an element $a\in\mathbb{F}_{p}\backslash\left\{ \pm\left(\tau+2\tau^{-1}\right)\right\} $,
put $\xi_{\tau,\,a,\,u}$ as $\tau^{2}+a\tau+2$. For a nonzero polynomial
$f$, we have
\begin{align}
 & \,a_{p,\,\tau,\,u}\left[f\right]m_{p,\,\tau,\,u}\left[a\right]\nonumber \\
= & \,\left(\begin{array}{cc}
f\tau\xi_{\tau,\,a,\,u}t^{-2}+\tau\left(-1+f\xi_{\tau,\,a,\,u}\right)t^{-1}+\cdots-\tau\left(-1+f\xi_{\tau,\,a,\,u}\right)t-f\tau\xi_{\tau,\,a,\,u}t^{2} & *\\
-f\xi_{\tau,\,a,\,u}t^{-2}+\left(1+f\tau^{2}\xi_{\tau,\,a,\,u}\right)t^{-1}+\left(f\xi_{\tau,\,a,\,u}+\tau^{2}\right)-f\tau^{2}\xi_{\tau,\,a,\,u}t & *
\end{array}\right)^{T}.
\end{align}

By (43), the element $\overline{a_{p,\,\tau,\,u}}\left[f\right]\overline{m_{p,\,\tau,\,u}}\left[a\right]$
is upper-balanced and has type $\tau$, signature $-1$, second-order
type $-\tau^{3}$, and second-order signature $\tau+\tau^{-1}$. Assuming
the same conditions, put $\xi_{\tau,\,a,\,l}$ as $\tau^{2}-a\tau+2$.
We also have
\begin{align}
 & \,a_{p,\,\tau,\,l}\left[f\right]m_{p,\,\tau,\,l}\left[a\right]\nonumber \\
= & \,\left(\begin{array}{cc}
-f\tau\xi_{\tau,\,a,\,l}t^{-2}-\tau\left(1+f\xi_{\tau,\,a,\,l}\right)t^{-1}+\cdots+\tau\left(1+f\xi_{\tau,\,a,\,l}\right)t+f\tau\xi_{\tau,\,a,\,l}t^{2} & *\\
f\tau^{2}\xi_{\tau,\,a,\,l}t^{-1}+\left(\tau^{2}-f\xi_{\tau,\,a,\,l}\right)+\left(1-\tau^{2}\xi_{\tau,\,a,\,l}\right)t+f\xi_{\tau,\,a}t^{2} & *
\end{array}\right)^{T}.
\end{align}

By (44), the element $\overline{a_{p,\,\tau,\,l}}\left[f\right]\overline{m_{p,\,\tau,\,l}}\left[a\right]$
is lower-balanced and has type $\tau$, signature $-1$, second-order
type $-\tau^{3}$, and second-order signature $\tau+\tau^{-1}$.

\noindent \begin{lemma}

Suppose an element $\tau\in\mathbb{F}_{p}$ satisfies $\tau^{4}+\tau^{2}+1=0$.
We have
\begin{align*}
\left\langle \mathrm{Aff}_{p,\,\tau,\,u},\,\mathrm{Aff}_{p,\,\tau,\,l}\right\rangle  & =\mathrm{Aff}_{p,\,\tau,\,u}*\mathrm{Aff}_{p,\,\tau,\,l},
\end{align*}
where $*$ on the right-hand side denotes the internal free product.
Let us refer to the free prodcut as the \emph{base group} for $\tau$,
denoted by $G_{p,\,\tau}$.

Choose an element $M\in G_{p,\,\tau}$. Then, $M$ is balanced and
$\tau\left(M\right)=\tau$. When $\mathrm{rd}\left(M\right)\ge3$,
we also obtain that $\tau2_{\tau}\left(M\right)=-\tau^{3}$ and $\sigma2_{\tau}\left(M\right)=\tau+\tau^{-1}$.
Moreover, when $M$ is evenly balanced, we have $\sigma\left(M\right)=-1$.
When $M$ is oddly upper-balanced, we have $\sigma\left(M\right)=-\tau^{-2}$.
When $M$ is oddly lower-balanced, we have $\sigma\left(M\right)=-\tau^{2}$.\end{lemma}
\begin{proof}

\noindent We use the usual ping-pong lemma, considering the action
by right multiplication in $\mathcal{Q}_{p}$. Define $S_{p,\,\tau,\,u}$
to be the set of elements $M\in\mathcal{Q}_{p}$ such that $\tau\left(M\right)=\tau$,
$\mathrm{rd}\left(M\right)\ge3$ and $M$ is upper-balanced, and define
$S_{p,\,\tau,\,l}$ to be the set of elements $M\in\mathcal{Q}_{p}$
such that $\tau\left(M\right)=\tau$, $\mathrm{rd}\left(M\right)\ge3$
and $M$ is lower-balanced. By definition, $S_{p,\,\tau,\,u}$ contains
the upper additive subgroup and $S_{p,\,\tau,\,l}$ contains the lower
additive subgroup. For elements $M_{1}\in S_{p,\,\tau,\,l}$ and $M_{2}\in S_{p,\,\tau,\,u}$,
for a polynomial $f\left(x\right)\in\mathbb{F}_{p}\left[x\right]$
and for an element $a\in\mathbb{F}_{p}\backslash\left\{ \pm\left(\tau+2\tau^{-1}\right)\right\} $,
to establish the free product, it suffices to show that
\begin{align}
 & M_{1}\overline{a_{p,\,\tau,\,u}}\left[f\right]\in S_{p,\,\tau,\,u},\\
 & M_{1}\overline{a_{p,\,\tau,\,u}}\left[f\right]\overline{m_{p,\,\tau,\,u}}\left[a\right]\in S_{p,\,\tau,\,u},\\
 & M_{2}\overline{a_{p,\,\tau,\,l}}\left[f\right]\in S_{p,\,\tau,\,l},\\
 & M_{2}\overline{a_{p,\,\tau,\,l}}\left[f\right]\overline{m_{p,\,\tau,\,l}}\left[a\right]\in S_{p,\,\tau,\,l},
\end{align}
where we used the fact that the upper (or lower) affine subgroup is
a semidirect product by Lemma 4.6.

We first prove (45). Take a representative $B=\left(\begin{array}{cc}
g_{1} & g_{2}\\
-\Phi\overline{g_{2}} & \overline{g_{1}}
\end{array}\right)$ of $M_{1}$, such that the first row entries are
\begin{align}
 & g_{1}=t^{m_{1}-1}+a_{1}t^{m_{1}}+\cdots+a_{2}t^{n_{1}-1}+\tau b_{3}t^{n_{1}},\\
 & g_{2}=-\tau t^{m_{1}}+b_{1}t^{m_{1}+1}+\cdots+b_{2}t^{n_{1}-1}+b_{3}t^{n_{1}},
\end{align}
where $b_{3}\ne0$ and the coefficients are rewritten satisfying the
functional equation (15). From the equation (15), we also have
\begin{align}
b_{2} & =\tau^{-1}a_{2}+b_{3}\left(-1+a_{1}+b_{1}\tau^{-1}\right).
\end{align}

Suppose $f\ne0$ and $\deg\left(f\right)=n$ for some nonnegative
integer $n$. For some units $u_{1},\,u_{2}\in\mathbb{F}_{3}^{\times}$,
we may write $f\left(t+t^{-1}\right)=u_{1}\left(t^{-n}+u_{2}t^{-n+1}+\cdots+u_{2}t^{n-1}+t^{n}\right)$.
Expanding by (41), we have

\noindent 
\begin{align*}
 & \,Ba_{p,\,\tau,\,u}\left(f\right)\\
= & \,u\tau^{2}B\left(\begin{array}{cc}
t^{-n-2}+\left(1+u_{2}\right)t^{-n-1}+\cdots-\left(1+u_{2}\right)t^{n+1}-t^{n+2} & *\\
-\tau^{-1}t^{-n-2}+\left(\tau-u_{2}\tau^{-1}\right)t^{-n-1}+\cdots+\left(\tau^{-1}-u_{2}\tau\right)t^{n}-\tau t^{n+1} & *
\end{array}\right)^{T}\\
= & \,u\tau^{2}\left(\begin{array}{cc}
t^{m_{1}-n-3}+\left(2+u_{2}+a_{1}+\tau^{-2}\right)t^{m_{1}-n-2}+\cdots+a_{2}'t^{n_{1}+n+2}+\tau^{-1}b_{3}t^{n_{1}+n+3} & *\\
-\tau^{-1}t^{m_{1}-n-3}+\left(2\tau-\left(a_{1}+u_{2}\right)\tau^{-1}\right)t^{m_{1}-n-2}+\cdots+b_{2}'t^{n_{1}+n+1}+b_{3}t^{n_{1}+n+2} & *
\end{array}\right)^{T},
\end{align*}

\noindent where $a_{2}'=\tau^{-1}\left(b_{2}+b_{3}+b_{3}u_{2}\right)-2b_{3}\tau$
and $b_{2}'=b_{2}+b_{3}\left(u_{2}+1-\tau^{2}\right)$.

Therefore, the matrix $Ba_{p,\,\tau,\,u}\left(f\right)$ is upper-balanced
and has type $\tau$, signature $\tau^{-1}b_{3}$, second-order type
$-\tau^{3}$, and a relative degree that is not smaller than that
of $B$. In particular, we have (45). The proof of (46) follows in
the same manner.

On the other hand, we prove (47). Take a representative $B=\left(\begin{array}{cc}
g_{1} & g_{2}\\
-\Phi\overline{g_{2}} & \overline{g_{1}}
\end{array}\right)$ of $M_{2}$, such that the first row entries are
\begin{align*}
 & g_{1}=-\tau t^{m_{1}}+a_{1}t^{m_{1}+1}+\cdots+a_{2}t^{n_{1}}+b_{3}t^{n_{1}+1},\\
 & g_{2}=t^{m_{1}}+b_{1}t^{m_{1}+1}+\cdots+b_{2}t^{n_{1}-1}+\tau b_{3}t^{n_{1}},
\end{align*}
where $b_{3}\ne0$ and the coefficients are rewritten satisfying the
functional equation (15). Suppose $f\ne0$ and $\deg\left(f\right)=n$
for some nonnegative integer $n$. For some units $u_{1},\,u_{2}\in\mathbb{F}_{3}^{\times}$,
we may write $f\left(t+t^{-1}\right)=u_{1}\left(t^{-n}+u_{2}t^{-n+1}+\cdots+u_{2}t^{n-1}+t^{n}\right)$.
Expanding by (42), we have

\noindent 
\begin{align*}
 & \,Ba_{p,\,\tau,\,l}\left(f\right)\\
= & \,u\tau^{2}B\left(\begin{array}{cc}
t^{-n-2}+\left(1+u_{2}\right)t^{-n-1}+\cdots-\left(1+u_{2}\right)t^{n+1}-t^{n+2} & *\\
-\tau t^{-n-1}+\left(\tau^{-1}-u_{2}\tau\right)t^{-n}+\cdots+\left(\tau-u_{2}\tau^{-1}\right)t^{n+1}-\tau^{-1}t^{n+2} & *
\end{array}\right)^{T}\\
= & \,u\tau^{2}\left(\begin{array}{cc}
\tau^{-1}t^{m_{1}-n-3}+\left(1+u_{2}+b_{1}-2\tau^{2}\right)\tau^{-1}t^{m_{1}-n-2}+\cdots+a_{2}'t^{n_{1}+n+2}-b_{3}t^{n_{1}+n+3} & *\\
-t^{m_{1}-n-2}-\left(1+b_{1}+u_{2}-\tau^{2}\right)t^{m_{1}-n-1}+\cdots+b_{2}'t^{n_{1}+n+2}-b_{3}\tau^{-1}t^{n_{1}+n+3} & *
\end{array}\right)^{T},
\end{align*}

\noindent where $a_{2}'=b_{3}\tau^{2}-b_{3}\left(1+u_{2}\right)-a_{2}$
and $b_{2}'=b_{3}\left(2\tau-u_{2}\tau^{-1}\right)-a_{2}\tau^{-1}$.

Therefore, the matrix $Ba_{p,\,\tau,\,l}\left(f\right)$ is lower-balanced
and has type $\tau$, signature $-\tau b_{3}$, second-order type
$-\tau^{3}$, and relative degree not smaller than that of $B$. In
particular, we have (47). The proof of (48) is done in the same way.
We have established the free product.

We use induction to prove the rest. Every multiplicative generator
is evenly balanced, with type $\tau$ and signature $-1$. For any
nonzero polynomial $f\left(x\right)\in\mathbb{F}_{p}\left[x\right]$,
every additive generator $\overline{a_{p,\,\tau,\,u}}\left[f\right]$
(or $\overline{a_{p,\,\tau,\,l}}\left[f\right]$) is also evenly balanced,
with type $\tau$, signature $-1$, second-order type $-\tau^{3}$,
and second-order signature $\tau+\tau^{-1}$ as seen from (41) and
(42). When we proved (45) and (46) above, if $M_{1}$ is evenly (resp.
oddly) balanced with signature $\sigma$, for any $A\in\mathrm{Aff}_{p,\,\tau,\,u}$,
$M_{1}A$ is oddly (resp. evenly) balanced with signature $\sigma\tau^{-2}$
(resp. $\sigma\tau^{2}$), while the type and the second-order type
remain the same. Therefore, the normal form of the free product inductively
establishes the desired properties on the balancedness, the type,
the signature, and the second-order type, by retracing the above calculations.

On the second-order signature, suppose $B$ is evenly lower-balanced
in (45). Then, from $\sigma\left(B\right)=-1$, we have $b_{3}=-\tau^{-1}$
in (49) and (50). Since $Ba_{p,\,\tau,\,u}\left(f\right)$ is oddly
upper-balanced, the second-order signature of $Ba_{p,\,\tau,\,u}\left(f\right)$
is computed to be
\begin{align*}
 & \,-\tau^{2}\left(a_{2}'\tau-\left(2\tau-\left(a_{1}+u_{2}\right)\tau^{-1}\right)\right)\\
= & \,-\tau^{2}\left(b_{2}+b_{3}+b_{3}u_{2}-2b_{3}\tau^{2}-2\tau+\left(a_{1}+u_{2}\right)\tau^{-1}\right)\\
= & \,-\tau^{2}\left(\tau^{-1}a_{2}-\tau^{-1}\left(a_{1}+u_{2}+b_{1}\tau^{-1}-2\tau^{2}\right)-2\tau+\left(a_{1}+u_{2}\right)\tau^{-1}\right)\\
= & \,b_{1}-\tau a_{2}\\
= & \,\sigma2_{\tau}\left(M_{1}\right),
\end{align*}
where the first equality follows from (51), the second from the fact
$b_{3}=-\tau^{-1}$, and the last from the assumption that $B$ is
evenly lower-balanced and Definition 4.7. Therefore, we have
\[
\sigma2_{\tau}\left(M_{1}\overline{a_{p,\,\tau,\,u}}\left[f\right]\right)=\sigma2_{\tau}\left(M_{1}\right).
\]

The other cases (46), (47) and (48) are dealt with in the same manner,
and inductively, for any $M\in G_{p,\,\tau}$ such that $\mathrm{rd}\left(M\right)\ge3$,
we conclude that $\sigma2_{\tau}\left(M\right)=\tau+\tau^{-1}$. $\qedhere$

\noindent \end{proof}

From now on, we will focus on the properties of stable generators.
For an element $\tau\in\mathbb{F}_{p}$ such that $\tau^{4}+\tau^{2}+1=0$,
recall that the stable generator for $a\in\mathbb{F}_{p}\backslash\left\{ 3\tau+\tau^{-1}\right\} $
is given by
\begin{align*}
\overline{s_{p,\,\tau}}\left[a\right] & =\left[\begin{array}{cc}
\frac{-\tau+a_{1}t+at^{2}+\tau^{-1}t^{3}}{t} & \frac{1+a_{3}t+t^{2}}{t}\\
-\Phi\left(\frac{1+a_{3}t+t^{2}}{t}\right) & \frac{\tau^{-1}+at+a_{1}t^{2}-\tau t^{3}}{t^{2}}
\end{array}\right],
\end{align*}
where $a_{1}=\tau\left(3-a\tau+6\tau^{2}\right)$ and $a_{3}=-3\tau^{2}+\tau a-2$.

\noindent \begin{lemma}

Suppose that an element $\tau\in\mathbb{F}_{p}$ satisfies $\tau^{4}+\tau^{2}+1=0$.
For a pair of elements $a_{1}\in\mathbb{F}_{p}\backslash\left\{ \pm\left(\tau+2\tau^{-1}\right)\right\} $
and $a_{2}\in\mathbb{F}_{p}\backslash\left\{ 3\tau+\tau^{-1}\right\} $,
we have
\begin{align}
 & \overline{m_{p,\,\tau,\,l}}\left[a_{1}\right]\overline{s_{p,\,\tau}}\left[a_{2}\right]=\overline{s_{p,\,\tau}}\left[a_{3}\right],\\
 & \overline{s_{p,\,\tau}}\left[a_{2}\right]\overline{m_{p,\,\tau,\,u}}\left[a_{1}\right]=\overline{s_{p,\,\tau}}\left[a_{3}\right],
\end{align}
where $a_{3}$ is given by
\begin{align*}
a_{3} & =\frac{2\tau\left(\tau^{2}+5\right)-a_{2}\left(\tau^{2}+a_{1}\tau+2\right)}{\tau^{2}-a_{1}\tau+2},
\end{align*}
such that $\tau^{2}-a_{1}\tau+2\ne0$ and $a_{2}\ne a_{3}$.

\noindent \end{lemma}
\begin{proof}

\noindent The formulas (52) and (53) are direct from computation.
The denominator $\tau^{2}-a_{1}\tau+2$ is nonzero by the assumption
$a_{1}\ne\tau+2\tau^{-1}$. By the definition of $a_{3}$, the condition
$a_{2}=a_{3}$ implies
\begin{align}
a_{2} & =\frac{\tau\left(5+\tau^{2}\right)}{\tau^{2}+2}.
\end{align}

However, the equality (54) implies $a_{2}=3\tau+\tau^{-1}$, since
we have
\begin{align*}
\frac{\tau\left(5+\tau^{2}\right)}{\tau^{2}+2}-\left(3\tau+\tau^{-1}\right) & =-\frac{2\left(1+\tau^{2}+\tau^{4}\right)}{\tau\left(\tau^{2}+2\right)},
\end{align*}
where the right-hand side is 0 by the assumption $\tau^{4}+\tau^{2}+1=0$.
$\qedhere$

\noindent \end{proof}

\noindent \begin{lemma}

Suppose an element $\tau\in\mathbb{F}_{p}$ satisfies $\tau^{4}+\tau^{2}+1=0$.
For a pair of elements $a_{1},\,a_{2}\in\mathbb{F}_{p}\backslash\left\{ 3\tau+\tau^{-1}\right\} $
such that $a_{1}\ne a_{2}$, we have
\begin{align}
 & \overline{s_{p,\,\tau}}\left[a_{1}\right]^{-1}\overline{s_{p,\,\tau}}\left[a_{2}\right]=\overline{m_{p,\,\tau,\,u}}\left[-a_{3}\right],\\
 & \overline{s_{p,\,\tau}}\left[a_{1}\right]\overline{s_{p,\,\tau}}\left[a_{2}\right]^{-1}=\overline{m_{p,\,\tau,\,l}}\left[a_{3}\right],
\end{align}
where $a_{3}$ is given by
\begin{align*}
a_{3} & =\frac{3\left(6+4\tau^{2}+\left(a_{1}+a_{2}\right)\tau^{3}\right)}{\left(a_{1}-a_{2}\right)\left(\tau^{2}+2\right)},
\end{align*}
such that $a_{3}\ne\pm\tau\left(2\tau^{2}+1\right)$. Note that $\overline{m_{p,\,\tau,\,u}}\left[-a_{3}\right]=\overline{m_{p,\,\tau,\,u}}\left[a_{3}\right]^{-1}$
by (40).

\noindent \end{lemma}
\begin{proof}

\noindent The formulas (55) and (56) are direct from computation.
Suppose that $a_{3}=\tau\left(2\tau^{2}+1\right)$. By the definition
of $a_{3}$, we have
\begin{align*}
 & \,0=-9-6\tau^{2}+a_{1}\tau\left(1+\tau^{2}+\tau^{4}\right)-a_{2}\tau\left(1+4\tau^{2}+\tau^{4}\right)\\
\Leftrightarrow & \,a_{2}=-3\tau^{-3}-2\tau^{-1}\\
\Leftrightarrow & \,a_{2}=3\tau+\tau^{-1},
\end{align*}

\noindent which contradicts the assumption $a_{2}\ne3\tau+\tau^{-1}$.

On the other hand, suppose that $a_{3}=-\tau\left(2\tau^{2}+1\right)$.
By the definition of $a_{3}$, we have
\begin{align*}
 & \,0=9+6\tau^{2}+a_{1}\tau\left(1+4\tau^{2}+\tau^{4}\right)-a_{2}b\left(1+\tau^{2}+\tau^{4}\right)\\
\Leftrightarrow & \,a_{1}=-3\tau^{-3}-2\tau^{-1}\\
\Leftrightarrow & \,a_{1}=3\tau+\tau^{-1},
\end{align*}
which contradicts the assumption $a_{1}\ne3\tau+\tau^{-1}$. $\qedhere$

\noindent \end{proof}
\begin{corollary}

Suppose an element $\tau\in\mathbb{F}_{p}$ satisfies $\tau^{4}+\tau^{2}+1=0$.
For a pair of elements $a_{1}\in\mathbb{F}_{p}\backslash\left\{ \pm\left(\tau+2\tau^{-1}\right)\right\} $
and $a_{2}\in\mathbb{F}_{p}\backslash\left\{ 3\tau+\tau^{-1}\right\} $,
we have
\begin{align}
 & \overline{s_{p,\,\tau}}\left[a_{2}\right]\overline{m_{p,\,\tau,\,u}}\left[a_{1}\right]\overline{s_{p,\,\tau}}\left[a_{2}\right]^{-1}=\overline{m_{p,\,\tau,\,l}}\left[a_{1}\right].
\end{align}
\end{corollary}
\begin{proof}

\noindent By (53) in Lemma 4.9, we have
\begin{align*}
\overline{s_{p,\,\tau}}\left[a_{2}\right]\overline{m_{p,\,\tau,\,u}}\left[a_{1}\right]\overline{s_{p,\,\tau}}\left[a_{2}\right]^{-1} & =\overline{s_{p,\,\tau}}\left[a_{3}\right]\overline{s_{p,\,\tau}}\left[a_{2}\right]^{-1},
\end{align*}
where $a_{3}=\frac{2\tau\left(\tau^{2}+5\right)-a_{2}\left(\tau^{2}+a_{1}\tau+2\right)}{\tau^{2}-a_{1}\tau+2}.$
By (56) in Lemma 4.10, we have
\begin{align*}
\overline{s_{p,\,\tau}}\left[a_{3}\right]\overline{s_{p,\,\tau}}\left[a_{2}\right]^{-1} & =\overline{m_{p,\,\tau,\,l}}\left[a_{4}\right],
\end{align*}
where $a_{4}=\frac{3\left(6+4\tau^{2}+\left(a_{3}+a_{2}\right)\tau^{3}\right)}{\left(a_{3}-a_{2}\right)\left(\tau^{2}+2\right)}$.
Expanding by substituting $a_{3}$ expressed by $a_{1}$ and $a_{2}$
into $a_{4}$, we have
\begin{align*}
 & a_{4}\\
= & \,\frac{3\left(6+4\tau^{2}+\left(a_{3}+a_{2}\right)\tau^{3}\right)}{\left(a_{3}-a_{2}\right)\left(2+\tau^{2}\right)}\\
= & \,\frac{3a_{1}\tau\left(3+2\tau^{2}+a_{2}\tau^{3}\right)}{\left(\tau\left(\tau^{2}+5\right)-a_{2}\left(\tau^{2}+2\right)\right)\left(\tau^{2}+2\right)}\\
= & \,-\frac{a_{1}\tau\left(3+2\tau^{2}+a_{2}\tau^{3}\right)}{a_{2}\left(1+\tau^{2}\right)-3\tau-2\tau^{3}},
\end{align*}
where from the assumption $\tau^{4}+\tau^{2}+1=0$, we have simplified
the numerator by the second equality and the denominator by the last.
The same assumption justifies that $a_{1}=a_{4}$ or
\begin{align*}
 & \,\tau\left(3+2\tau^{2}+a_{2}\tau^{3}\right)=-a_{2}\left(1+\tau^{2}\right)+3\tau+2\tau^{3}\\
\Leftrightarrow & \,a_{2}\left(1+\tau^{2}+\tau^{4}\right)=0,
\end{align*}
which establishes (57). $\qedhere$

\noindent \end{proof}

Corollary 4.11 ensures that the conjugation on the upper multiplicative
subgroup by a stable generator induces the isomorphism to the lower
multiplicative subgroup. Note that (57) holds independently of the
choice of $a_{2}$ for $\overline{s_{p,\,\tau}}\left[a_{2}\right]$.
Since the multiplicative subgroup is abelian by Lemma 4.5, one may
choose any $a_{2}$ in $\mathbb{F}_{p}\backslash\left\{ 3\tau+\tau^{-1}\right\} $
by (53) in Lemma 4.9.

To prove that a stable generator and the affine subgroups together
generate an HNN extension, we cannot use the usual ping-pong lemma.
The following Fenchel\textendash Nielsen theorem provides a ping-pong
lemma for HNN extensions.

\noindent \begin{theorem}

Let $G$ be a group, $G_{0}$ be a subgroup of $G$, and $\alpha:H\to K$
be an isomorphism between two subgroups of $G_{0}$. Suppose that
an element $s\in G\backslash G_{0}$ induces the isomorphism by conjugation
as
\begin{align*}
shs^{-1} & =\alpha\left(h\right),\;h\in H,
\end{align*}
and $G$ is generated by $G_{0}$ and $s$. Suppose further that $G$
acts on a set $\Omega$, which contains non-empty disjoint subsets
$X_{0},\,Z_{s},\,Z_{\overline{s}}$ such that
\begin{enumerate}
\item $s\left(X_{0}\cup Z_{s}\right)\subset Z_{s}$ and $s^{-1}\left(X_{0}\cup Z_{\overline{s}}\right)\subset Z_{\overline{s}}$;
\item $h\left(Z_{s}\right)\subset Z_{s}$ for each $h\in H$ and $k\left(Z_{\overline{s}}\right)\subset Z_{\overline{s}}$
for each $k\in K$;
\item $\left(G_{0}\backslash H\right)\left(Z_{s}\right)\subset X_{0}$ and
$\left(G_{0}\backslash K\right)\left(Z_{\overline{s}}\right)\subset X_{0}$;
\item both the restrictions of the actions of $H$ to $Z_{s}$ and $K$
to $Z_{\overline{s}}$ are faithful;
\item the union of the images $\left(G_{0}\backslash H\right)\left(Z_{s}\right)\cup\left(G_{0}\backslash K\right)\left(Z_{\overline{s}}\right)$
misses a point in $X_{0}$.
\end{enumerate}
Then, the group $G$ is the internal HNN extension of the subgroup
$G_{0}$ with $s$ as the stable letter.

\noindent \end{theorem}
\begin{proof}

\noindent See \citep[�VII.D.12]{MR0959135} or \citep[Lemma 5.6]{MR4735233}.
$\qedhere$

\noindent \end{proof}

\noindent \begin{lemma}

Suppose an element $\tau\in\mathbb{F}_{p}$ satisfies $\tau^{4}+\tau^{2}+1=0$.
Choose an element $M\in\mathcal{Q}_{p}$ such that $M$ is upper-balanced,
$\mathrm{rd}\left(M\right)\ge3$, $\tau\left(M\right)=\tau$ and $\tau2_{\tau}\left(M\right)=-\left(\tau+3\tau^{-1}\right)$.
Then, for any $a\in\mathbb{F}_{p}\backslash\left\{ \pm\left(\tau+2\tau^{-1}\right)\right\} $,
$M\overline{m_{p,\,\tau,\,u}}\left[a\right]$ is upper-balanced, and
has the same relative degree, type, signature, and second-order type
as $M$.

On the other hand, Choose an element $M\in\mathcal{Q}_{p}$ such that
$M$ is lower-balanced, $\mathrm{rd}\left(M\right)\ge3$, $\tau\left(M\right)=\tau$
and $\tau2_{\tau}\left(M\right)=-\left(\tau+3\tau^{-1}\right)$. Then,
for any $a\in\mathbb{F}_{p}\backslash\left\{ \pm\left(\tau+2\tau^{-1}\right)\right\} $,
$M\overline{m_{p,\,\tau,\,l}}\left[a\right]$ is lower-balanced, and
has the same relative degree, type, signature, and second-order type
as $M$.

\noindent \end{lemma}
\begin{proof}

\noindent Suppose that $M$ is upper-balanced. Choose a representative
$B=\left(\begin{array}{cc}
g_{1} & g_{2}\\
-\Phi\overline{g_{2}} & \overline{g_{1}}
\end{array}\right)$ of $M$ such that
\begin{align*}
 & g_{1}=-\tau t^{m_{1}}+a_{1}t^{m_{1}+1}+\cdots+a_{2}t^{n_{1}}+\tau^{-1}b_{3}t^{n_{1}+1},\\
 & g_{2}=t^{m_{1}}+b_{1}t^{m_{1}+1}+\cdots+b_{2}t^{n_{1}-1}+b_{3}t^{n_{1}}.
\end{align*}

By the functional equation (15), we have $b_{2}=a_{2}\tau-a_{1}b_{3}\tau^{-1}-\left(1+b_{1}\right)b_{3}$.
Then, the entry $\left(Bm_{p,\,\tau,\,u}\left[a\right]\right)_{11}$
is computed to be
\begin{align*}
 & \,-\left(1+\left(a+a_{1}\right)\tau+\left(1+b_{1}\right)\tau^{2}\right)t^{m_{1}}-\left(1-aa_{1}+b_{1}+\left(2+b_{1}\right)\tau^{2}\right)t^{m_{1}+1}\\
 & \,+\cdots+\tau^{-1}b_{3}\left(a+a_{1}-\tau^{3}+b_{1}\tau\right)t^{n_{1}+1}.
\end{align*}

At first, suppose that the coefficient of $t^{m_{1}}$ is nonzero.
Then, the signature remains the same from the assumption $\tau^{4}+\tau^{2}+1=0$.
In succession, the entry $\left(Bm_{p,\,\tau,\,u}\left[a\right]\right)_{12}$
is computed to be
\begin{align*}
 & \,\left(a+a_{1}-\tau^{3}+b_{1}\tau\right)t^{m_{1}}+\left(ab_{1}+a_{1}\tau^{2}-\tau\right)t^{m_{1}+1}\\
 & \,+\cdots+\tau^{-1}b_{3}\left(1+\left(a+a_{1}\right)\tau+\left(1+b_{1}\right)\tau^{2}\right)t^{n_{1}}.
\end{align*}

Comparing the terms, we also see that $Bm_{p,\,\tau,\,u}\left[a\right]$
is upper-balanced and the type remains $\tau$. The second-order type
also remains $-\tau-3\tau^{-1}$ as
\begin{align*}
 & \,\left(a+a_{1}-\tau^{3}+b_{1}\tau\right)^{-1}\left(-\left(1-aa_{1}+b_{1}+\left(2+b_{1}\right)\tau^{2}\right)+\tau\left(ab_{1}+a_{1}\tau^{2}-\tau\right)\right)+\tau\\
= & \,\left(a+a_{1}-\tau^{3}+b_{1}\tau\right)^{-1}\left(a\left(a_{1}+\left(1+b_{1}\right)\tau\right)-b_{1}-a_{1}\tau^{-1}-2\tau^{2}\right)\\
= & \,\left(\tau^{2}-a\tau+2\right)^{-1}\left(a\left(3+\tau^{2}\right)-4\tau-5\tau^{-1}\right)\\
= & \,-\tau-3\tau^{-1},
\end{align*}
where the second equality follows from the assumption $a_{1}+\tau b_{1}+\tau=-\tau-3\tau^{-1}$.
On the other hand, suppose the coefficient of $t^{m_{1}}$ is zero.
Then, we have
\begin{align*}
 & \,a\\
= & \,-a_{1}+\tau^{3}-b_{1}\tau\\
= & \,\tau+2\tau^{-1},
\end{align*}
where the second equality follows from the assumption $a_{1}+\tau b_{1}+\tau=-\tau-3\tau^{-1}$.
This contradicts the fact that $a$ is chosen outside $\left\{ \pm\left(\tau+2\tau^{-1}\right)\right\} $.
The case where $M$ is lower-balanced is proven in the same way. $\qedhere$

\noindent \end{proof}
\begin{lemma}

Suppose an element $\tau\in\mathbb{F}_{p}$ satisfies $\tau^{4}+\tau^{2}+1=0$.
Choose an element $M\in\mathcal{Q}_{p}$ such that $M$ is upper-balanced,
$\mathrm{rd}\left(M\right)\ge2$ and $\tau\left(M\right)=\tau$. Then,
for any $a\in\mathbb{F}_{p}\backslash\left\{ 3\tau+\tau^{-1}\right\} $,
$M\overline{s_{p,\,\tau}}\left[a\right]$ is upper-balanced, and has
the same type and signature as $M$. Moreover, we have
\begin{align}
 & \,\mathrm{rd}\left(M\overline{s_{p,\,\tau}}\left[a\right]\right)=\mathrm{rd}\left(M\right)+4,\\
 & \,\tau2_{\tau}\left(M\overline{s_{p,\,\tau}}\left[a\right]\right)=-\left(\tau+3\tau^{-1}\right).
\end{align}

On the other hand, choose an element $M\in\mathcal{Q}_{p}$ such that
$M$ is lower-balanced, $\mathrm{rd}\left(M\right)\ge2$ and $\tau\left(M\right)=\tau$.
Then, for any $a\in\mathbb{F}_{p}\backslash\left\{ 3\tau+\tau^{-1}\right\} $,
$M\overline{s_{p,\,\tau}}\left[a\right]^{-1}$ is lower-balanced,
and has the same type and signature as $M$. We have
\begin{align}
 & \,\mathrm{rd}\left(M\overline{s_{p,\,\tau}}\left[a\right]^{-1}\right)=\mathrm{rd}\left(M\right)+4,\\
 & \,\tau2_{\tau}\left(M\overline{s_{p,\,\tau}}\left[a\right]^{-1}\right)=-\left(\tau+3\tau^{-1}\right).
\end{align}
\end{lemma}
\begin{proof}

\noindent Suppose $M$ is upper-balanced. Choose a representative
$B=\left(\begin{array}{cc}
g_{1} & g_{2}\\
-\Phi\overline{g_{2}} & \overline{g_{1}}
\end{array}\right)$ of $M$ such that
\begin{align*}
 & g_{1}=-\tau t^{m_{1}}+a_{1}t^{m_{1}+1}+\cdots+\tau^{-1}b_{2}t^{n_{1}+1},\\
 & g_{2}=t^{m_{1}}+b_{1}t^{m_{1}+1}+\cdots+b_{2}t^{n_{1}}.
\end{align*}

Then, the entry $\left(Bs_{p,\,\tau}\left[a\right]\right)_{11}$ is
computed to be
\begin{align*}
-t^{m_{1}-2}-\left(-1+a\tau-4\tau^{2}+b_{1}\right)t^{m_{1}-1}+\cdots+\tau^{-2}b_{2}t^{n_{1}+3},
\end{align*}

\noindent where the signature remains the same. The entry $\left(Bs_{p,\,\tau}\left[a\right]\right)_{12}$
is computed to be
\begin{align*}
\tau^{-1}t^{m_{1}-2}+\left(a-\tau+b_{1}\tau^{-1}\right)t^{m_{1}-1}+\cdots+\tau^{-1}b_{2}t^{n_{1}+2},
\end{align*}

\noindent where $Bs_{p,\,\tau}\left[a\right]$ is upper-balanced,
satisfies (58), and the type remains $\tau$. The value $\tau2_{\tau}\left(Bs_{p,\,\tau}\left[a\right]\right)$
is computed to be $-\left(\tau+3\tau^{-1}\right)$, by using the assumption
$\tau^{4}+\tau^{2}+1=0$, establishing (59). In the case where $M$
is lower-balanced, (60) and (61) are proven in the same way. $\qedhere$

\noindent \end{proof}
\begin{lemma}

Suppose an element $\tau\in\mathbb{F}_{p}$ satisfies $\tau^{4}+\tau^{2}+1=0$.
Choose an element $M\in\mathcal{Q}_{p}$ such that $M$ is lower-balanced,
$\mathrm{rd}\left(M\right)\ge3$, $\tau\left(M\right)=\tau$ and $\tau2_{\tau}\left(M\right)\ne-\left(\tau+3\tau^{-1}\right)$.
Then, for any $a\in\mathbb{F}_{p}\backslash\left\{ 3\tau+\tau^{-1}\right\} $,
$M\overline{s_{p,\,\tau}}\left[a\right]$ is upper-balanced, has type
$\tau$, and satisfies
\begin{align*}
 & \mathrm{rd}\left(M\overline{s_{p,\,\tau}}\left[a\right]\right)=\mathrm{rd}\left(M\right)+1,\,\tau\left(M\overline{s_{p,\,\tau}}\left[a\right]\right)=\tau,\\
 & \sigma\left(M\overline{s_{p,\,\tau}}\left[a\right]\right)=\tau^{-2}\sigma\left(M\right),\,\tau2_{\tau}\left(M\overline{s_{p,\,\tau}}\left[a\right]\right)=-\left(\tau+3\tau^{-1}\right).
\end{align*}

On the other hand, choose an element $M\in\mathcal{Q}_{p}$ such that
$M$ is upper-balanced, $\mathrm{rd}\left(M\right)\ge3$, $\tau\left(M\right)=\tau$
and $\tau2_{\tau}\left(M\right)\ne-\left(\tau+3\tau^{-1}\right)$.
Then, for any $a\in\mathbb{F}_{p}\backslash\left\{ 3\tau+\tau^{-1}\right\} $,
$M\overline{s_{p,\,\tau}}\left[a\right]^{-1}$ is lower-balanced,
has type $\tau$, and satisfies
\begin{align*}
 & \mathrm{rd}\left(M\overline{s_{p,\,\tau}}\left[a\right]\right)=\mathrm{rd}\left(M\right)+1,\,\tau\left(M\overline{s_{p,\,\tau}}\left[a\right]\right)=\tau,\\
 & \sigma\left(M\overline{s_{p,\,\tau}}\left[a\right]^{-1}\right)=\tau^{2}\sigma\left(M\right),\,\tau2_{\tau}\left(M\overline{s_{p,\,\tau}}\left[a\right]\right)=-\left(\tau+3\tau^{-1}\right).
\end{align*}
\end{lemma}
\begin{proof}

\noindent Suppose $M$ is upper-balanced. Choose a representative
$B=\left(\begin{array}{cc}
g_{1} & g_{2}\\
-\Phi\overline{g_{2}} & \overline{g_{1}}
\end{array}\right)$ of $M$ such that
\begin{align*}
 & g_{1}=-\tau t^{m_{1}}+a_{1}t^{m_{1}+1}+\cdots+a_{2}t^{n_{1}}+\tau^{-1}b_{3}t^{n_{1}+1},\\
 & g_{2}=t^{m_{1}}+b_{1}t^{m_{1}+1}+\cdots+b_{2}t^{n_{1}-1}+b_{3}t^{n_{1}}.
\end{align*}

By the functional equation (15), we have $b_{2}=a_{2}\tau-a_{1}b_{3}\tau^{-1}-\left(1+b_{1}\right)b_{3}$.
By expressing $b_{2}$ in $a_{1},\,a_{2},\,b_{3}$, the entry $\left(Bs_{p,\,\tau}\left[a\right]^{\dagger}\right)_{11}$
is computed to be
\begin{align*}
 & \,\left(-1+a_{1}\tau^{-1}+b_{1}-3\tau^{2}\right)t^{m_{1}-1}+\left(6+aa_{1}-a\tau^{-1}+ab_{1}\tau-3b_{1}\tau^{2}\right)t^{m_{1}}\\
 & \,+\cdots+b_{3}\left(1-a_{1}\tau^{-1}-b_{1}+3\tau^{2}\right)t^{n_{1}+1}.
\end{align*}

By the assumption $\tau2_{\tau}\left(M\right)\ne-\left(\tau+3\tau^{-1}\right)$,
we see that the coefficient of $t^{m_{1}-1}$ is nonzero and the signature
is $-b_{3}$. The entry $\left(Bs_{p,\,\tau}\left[a\right]^{\dagger}\right)_{12}$
is computed to be
\begin{align*}
 & \,\left(\tau-a_{1}-b_{1}\tau+3\tau^{3}\right)t^{m_{1}}+\left(a+2a_{1}+\tau-aa_{1}\tau+3a_{1}\tau^{2}+3b_{1}\tau-ab_{1}\tau^{2}+6b_{1}\tau^{3}\right)t^{m_{1}+1}\\
 & \,+\cdots+b_{3}\tau^{-2}\left(-a_{1}-b_{1}\tau+\tau+3\tau^{3}\right)t^{n_{1}+1}.
\end{align*}

Therefore, the matrix $Bs_{p,\,\tau}\left[a\right]^{\dagger}$ is
lower-balanced and has type $\tau$. The second-order type is directly
computed to be $-\left(\tau+3\tau^{-1}\right)$ by using the coefficients.
The case where $M$ is lower-balanced is proven in the same way. $\qedhere$

\noindent \end{proof}

Here is the key lemma used to establish the HNN extension. Recall
that we defined the base group $G_{p,\,\tau}$ for $\tau$ in Lemma
4.8.

\noindent \begin{lemma}

Suppose an element $\tau\in\mathbb{F}_{p}$ satisfies $\tau^{4}+\tau^{2}+1=0$.
Choose an element $M_{1}\in\mathcal{Q}_{p}$ such that $M_{1}$ is
balanced, $\mathrm{rd}\left(M_{1}\right)\ge3$, $\tau\left(M_{1}\right)=\tau$
and $\tau2_{\tau}\left(M_{1}\right)=-\left(\tau+3\tau^{-1}\right)$.
Choose an element $M_{2}\in G_{p,\,\tau}$.
\begin{description}
\item [{(a)}] Suppose $M_{1}$ is oddly upper-balanced and $\sigma\left(M_{1}\right)=-\tau^{-2}$.
The relative degree of $M_{1}M_{2}$ increases in all three cases
below.
\end{description}
\begin{enumerate}
\item When $M_{2}$ is upper-balanced and $\mathrm{rd}\left(M_{2}\right)\ge3$,
$\sigma\left(M_{1}M_{2}\right)=-\tau^{-2}$, $M_{1}M_{2}$ is oddly
upper-balanced, and $\tau2_{\tau}\left(M_{1}M_{2}\right)=-\tau^{3}$.
\item When $M_{2}$ is lower-balanced and $\mathrm{rd}\left(M_{2}\right)\ge3$,
$\sigma\left(M_{1}M_{2}\right)=-1$, $M_{1}M_{2}$ is evenly lower-balanced,
and $\tau2_{\tau}\left(M_{1}M_{2}\right)=-\tau^{3}$.
\item When $M_{2}$ is a lower multiplicative generator $\overline{m_{p,\,\tau,\,l}}\left[a\right]$,
$\sigma\left(M_{1}M_{2}\right)=-1$, $M_{1}M_{2}$ is evenly lower-balanced,
and $\tau2_{\tau}\left(M_{1}M_{2}\right)=-a-\tau^{-1}$, which is
different from both $-\left(\tau+3\tau^{-1}\right)$ and $-\tau^{3}$.
\end{enumerate}
\begin{description}
\item [{(b)}] Suppose $M_{1}$ is evenly upper-balanced and $\sigma\left(M_{1}\right)=-1$.
The relative degree of $M_{1}M_{2}$ increases in all three cases
below.
\end{description}
\begin{enumerate}
\item When $M_{2}$ is upper-balanced and $\mathrm{rd}\left(M_{2}\right)\ge3$,
$\sigma\left(M_{1}M_{2}\right)=-1$, $M_{1}M_{2}$ is evenly upper-balanced,
and $\tau2_{\tau}\left(M_{1}M_{2}\right)=-\tau^{3}$.
\item When $M_{2}$ is lower-balanced and $\mathrm{rd}\left(M_{2}\right)\ge3$,
$\sigma\left(M_{1}M_{2}\right)=-\tau^{2}$, $M_{1}M_{2}$ is oddly
lower-balanced, and $\tau2_{\tau}\left(M_{1}M_{2}\right)=-\tau^{3}$.
\item When $M_{2}$ is a lower multiplicative generator $\overline{m_{p,\,\tau,\,l}}\left[a\right]$,
$\sigma\left(M_{1}M_{2}\right)=-\tau^{2}$, $M_{1}M_{2}$ is oddly
lower-balanced, and $\tau2_{\tau}\left(M_{1}M_{2}\right)=-a-\tau^{-1}$,
which is different from both $-\left(\tau+3\tau^{-1}\right)$ and
$-\tau^{3}$.
\end{enumerate}
\begin{description}
\item [{(c)}] Suppose $M_{1}$ is oddly lower-balanced and $\sigma\left(M_{1}\right)=-\tau^{2}$.
The relative degree of $M_{1}M_{2}$ increases in all three cases
below.
\end{description}
\begin{enumerate}
\item When $M_{2}$ is upper-balanced and $\mathrm{rd}\left(M_{2}\right)\ge3$,
$\sigma\left(M_{1}M_{2}\right)=-1$, $M_{1}M_{2}$ is evenly upper-balanced,
and $\tau2_{\tau}\left(M_{1}M_{2}\right)=-\tau^{3}$.
\item When $M_{2}$ is lower-balanced and $\mathrm{rd}\left(M_{2}\right)\ge3$,
$\sigma\left(M_{1}M_{2}\right)=-\tau^{2}$, $M_{1}M_{2}$ is oddly
lower-balanced, and $\tau2_{\tau}\left(M_{1}M_{2}\right)=-\tau^{3}$.
\item When $M_{2}$ is a upper multiplicative generator $\overline{m_{p,\,\tau,\,u}}\left[a\right]$,
$\sigma\left(M_{1}M_{2}\right)=-1$, $M_{1}M_{2}$ is evenly upper-balanced,
and $\tau2_{\tau}\left(M_{1}M_{2}\right)=a-\tau^{-1}$, which is different
from both $-\left(\tau+3\tau^{-1}\right)$ and $-\tau^{3}$.
\end{enumerate}
\begin{description}
\item [{(d)}] Suppose $M_{1}$ is evenly lower-balanced and $\sigma\left(M_{1}\right)=-1$.
The relative degree of $M_{1}M_{2}$ increases in all three cases
below.
\end{description}
\begin{enumerate}
\item When $M_{2}$ is upper-balanced and $\mathrm{rd}\left(M_{2}\right)\ge3$,
$\sigma\left(M_{1}M_{2}\right)=-\tau^{-2}$, $M_{1}M_{2}$ is oddly
upper-balanced, and $\tau2_{\tau}\left(M_{1}M_{2}\right)=-\tau^{3}$.
\item When $M_{2}$ is lower-balanced and $\mathrm{rd}\left(M_{2}\right)\ge3$,
$\sigma\left(M_{1}M_{2}\right)=-1$, $M_{1}M_{2}$ is evenly lower-balanced,
and $\tau2_{\tau}\left(M_{1}M_{2}\right)=-\tau^{3}$.
\item When $M_{2}$ is a upper multiplicative generator $\overline{m_{p,\,\tau,\,u}}\left[a\right]$,
$\sigma\left(M_{1}M_{2}\right)=-\tau^{-2}$, $M_{1}M_{2}$ is oddly
upper-balanced, and $\tau2_{\tau}\left(M_{1}M_{2}\right)=a-\tau^{-1}$,
which is different from both $-\left(\tau+3\tau^{-1}\right)$ and
$-\tau^{3}$.
\end{enumerate}
\noindent \end{lemma}
\begin{proof}

\noindent Assume the case (a)-1. We choose representatives $B_{1}=\left(\begin{array}{cc}
g_{1,\,1} & g_{1,\,2}\\
-\Phi\overline{g_{1,\,2}} & \overline{g_{1,\,1}}
\end{array}\right)$ for $M_{1}$ and $B_{2}=\left(\begin{array}{cc}
g_{2,\,1} & g_{2,\,2}\\
-\Phi\overline{g_{2,\,2}} & \overline{g_{2,\,1}}
\end{array}\right)$ for $M_{2}$ such that
\begin{align*}
 & g_{1,\,1}=-\tau t^{m_{1}}+a_{1,\,1}t^{m_{1}+1}+\cdots+a_{1,\,2}t^{n_{1}}+\tau^{-1}t^{n_{1}+1},\\
 & g_{1,\,2}=t^{m_{1}}+b_{1,\,1}t^{m_{1}+1}+\cdots+b_{1,\,2}t^{n_{1}-1}+t^{n_{1}},\\
 & g_{2,\,1}=-\tau t^{m_{2}}+a_{2,\,1}t^{m_{2}+1}+\cdots+a_{2,\,2}t^{n_{2}}+\tau^{-1}b_{2,\,3}t^{n_{2}+1},\\
 & g_{2,\,2}=t^{m_{2}}+b_{2,\,1}t^{m_{2}+1}+\cdots+b_{2,\,2}t^{n_{2}-1}+b_{2,\,3}t^{n_{2}},
\end{align*}
where the functional equation (15) implies
\begin{align}
 & b_{1,\,2}=a_{1,\,2}\tau-1-a_{1,\,1}\tau^{-1}-b_{1,\,1},\\
 & b_{2,\,2}=a_{2,\,2}\tau-b_{2,\,3}\left(1+a_{2,\,1}\tau^{-1}+b_{2,\,1}\right).
\end{align}

The entry $\left(B_{1}B_{2}\right)_{11}$ is computed as
\begin{align}
 & \,\tau^{2}t^{m_{1}+m_{2}}-b_{2,\,3}t^{m_{1}-n_{2}-1}-\left(\tau a_{2,\,1}+\tau a_{1,\,1}\right)t^{m_{1}+m_{2}+1}\nonumber \\
 & -\left(b_{2,\,2}+b_{2,\,3}+b_{1,\,1}b_{2,\,3}\right)t^{m_{1}-n_{2}}+\cdots+\left(\tau^{-1}a_{2,\,2}+\tau^{-1}a_{1,\,2}b_{2,\,3}\right)t^{n_{1}+n_{2}+1}\nonumber \\
 & -\left(1+b_{2,\,1}+b_{1,\,2}\right)t^{n_{1}-m_{2}}+\tau^{-2}b_{2,\,3}t^{n_{1}+n_{2}+2}-t^{n_{1}-m_{2}+1},
\end{align}
while the entry $\left(B_{1}B_{2}\right)_{12}$ is computed as
\begin{align}
 & \,-\tau t^{m_{1}+m_{2}}+\tau^{-1}b_{2,\,3}t^{m_{1}-n_{2}-1}+\left(a_{1,\,1}-\tau b_{2,\,1}\right)t^{m_{1}+m_{2}+1}+\left(a_{2,\,2}+\tau^{-1}b_{1,\,1}b_{2,\,3}\right)t^{m_{1}-n_{2}}+\cdots\nonumber \\
 & +\left(\tau^{-1}b_{2,\,2}+a_{1,\,2}b_{2,\,3}\right)t^{n_{1}+n_{2}}+\left(a_{2,\,1}-\tau b_{1,\,2}\right)t^{n_{1}-m_{2}-1}+\tau^{-1}b_{2,\,3}t^{n_{1}+n_{2}+1}-\tau t^{n_{1}-m_{2}}.
\end{align}

By comparing (64) and (65), if $m_{2}\ne-n_{2}-1$, we obtain all
the desired results. Suppose $m_{2}=-n_{2}-1$, where $M_{2}$ is
evenly upper-balanced. By Lemma 4.8, we have $\sigma\left(M_{2}\right)=-1$
and $b_{2,\,3}=\tau^{2}$. Then, the coefficients of $t^{m_{1}+m_{2}}$,
$t^{n_{1}+n_{2}+2}$ in $\left(B_{1}B_{2}\right)_{11}$ and the coefficients
of $t^{m_{1}+m_{2}}$, $t^{n_{1}+n_{2}+1}$ in $\left(B_{1}B_{2}\right)_{12}$
are all zero. By canceling these terms, the entry $\left(B_{1}B_{2}\right)_{11}$
becomes
\begin{align}
 & \,-\left(b_{2,\,2}+b_{2,\,3}+b_{1,\,1}b_{2,\,3}+\tau a_{2,\,1}+\tau a_{1,\,1}\right)t^{m_{1}+m_{2}+1}\nonumber \\
 & +\cdots+\left(\tau^{-1}a_{2,\,2}+\tau^{-1}a_{1,\,2}b_{2,\,3}-1-b_{2,\,1}-b_{1,\,2}\right)t^{n_{1}+n_{2}+1}.
\end{align}

We claim that the coefficients of $t^{m_{1}+m_{2}+1}$, $t^{n_{1}+n_{2}+1}$
in (63) are zero if and only if
\begin{align*}
\tau2_{\tau}\left(M_{1}\right) & =\sigma2_{\tau}\left(M_{2}\right).
\end{align*}

We prove this claim as follows:
\begin{align*}
 & \,0\\
= & \,-\left(b_{2,\,2}+b_{2,\,3}+b_{1,\,1}b_{2,\,3}+\tau a_{2,\,1}+\tau a_{1,\,1}\right)\\
= & \,-\tau\left(a_{1,\,1}+b_{1,\,1}\tau+\tau+\left(a_{2,\,2}-b_{2,\,1}\tau-\tau\right)\right)\\
= & \,-\tau\left(\tau2_{\tau}\left(M_{1}\right)-\sigma2_{\tau}\left(M_{2}\right)\right),
\end{align*}
where the second equality follows from (60), and the last follows
from Definition 4.7. The assumption gives that $\tau2_{\tau}\left(M_{1}\right)=-\left(\tau+3\tau^{-1}\right)$
and Lemma 4.8 gives that $\sigma2_{\tau}\left(M_{2}\right)=\tau+\tau^{-1}$.
Therefore,
\begin{align*}
 & \,\tau2_{\tau}\left(M_{1}\right)=\sigma2_{\tau}\left(M_{2}\right)\\
\Leftrightarrow & \,\tau^{2}=-2\\
\Rightarrow & \,1+\tau^{2}+\tau^{4}=3,
\end{align*}
where the last condition implies $p=3$. Therefore, we obtain that
the coefficients of $t^{m_{1}+m_{2}+1}$, $t^{n_{1}+n_{2}+2}$ in
(63) are nonzero. On the other hand, from (62), the entry $\left(B_{1}B_{2}\right)_{12}$
becomes
\begin{align*}
 & \,\left(a_{1,\,1}-\tau b_{2,\,1}+a_{2,\,2}+\tau^{-1}b_{1,\,1}b_{2,\,3}\right)t^{m_{1}+m_{2}+1}\\
 & +\cdots+\left(\tau^{-1}b_{2,\,2}+a_{1,\,2}b_{2,\,3}+a_{2,\,1}-\tau b_{1,\,2}\right)t^{n_{1}+n_{2}},
\end{align*}
which yields that $M_{1}M_{2}$ is oddly upper-balanced and $\sigma\left(M_{1}M_{2}\right)=-\tau^{-2}$
by using (62) and $b_{2,\,3}=\tau^{2}$.

When $\mathrm{rd}\left(M_{1}\right)\ge5$ (or $\mathrm{rd}\left(M_{2}\right)\ge5$),
we need to specify the three terms starting from the highest degree
and the three terms starting from the lowest degree in the entries
of $B_{1}$ (or $B_{2}$) (see the proof of Claim 2 in the proof of
Lemma 4.18). In that situation, when one computes the matrix multiplication
$B_{1}B_{2}$, then the second-order type of $M_{1}M_{2}$ will turn
out to be the same as $\tau2_{\tau}\left(M_{2}\right)$, provided
that the coefficients of $t^{m_{1}+m_{2}+1}$, $t^{n_{1}+n_{2}+1}$
in (66) are nonzero. Since Lemma 4.8 ensures that $\tau2_{\tau}\left(M_{2}\right)=-\tau^{3}$,
we conclude that $\tau2_{\tau}\left(M_{1}M_{2}\right)=-\tau^{3}$.

The other cases are proven in the same way. The key to the calculation
is to show that when the lowest degree terms in the entries of $M_{1}M_{2}$
is canceled, as before, cancellation occurs only once under the condition
$\tau2_{\tau}\left(M_{1}\right)\ne\sigma2_{\tau}\left(M_{2}\right)$,
from Lemma 4.8. $\qedhere$

\noindent \end{proof}

Recall that in Theorem 4.1, we defined the type group $\mathcal{G}_{p,\,\tau}$
to be generated by all affine generators of type $\tau$ and a stable
generator $\overline{s_{p,\,\tau}}\left[a\right]$ for any $a\in\mathbb{F}_{p}\backslash\left\{ 3\tau+\tau^{-1}\right\} $.
We describe the structure as follows, as claimed in Theorem 4.1.

\noindent \begin{lemma}

Suppose an element $\tau\in\mathbb{F}_{p}$ satisfies $\tau^{4}+\tau^{2}+1=0$.
Then, the type group $\mathcal{G}_{p,\,\tau}$ is the internal HNN
extension with $G_{p,\,\tau}$ as the base group of the HNN extension
and $\overline{s_{p,\,\tau}}\left[a\right]$ as the stable letter
for any $a\in\mathbb{F}_{p}\backslash\left\{ 3\tau+\tau^{-1}\right\} $.

\noindent \end{lemma}
\begin{proof}

\noindent Fix an element $a\in\mathbb{F}_{p}\backslash\left\{ 3\tau+\tau^{-1}\right\} $.
We use Theorem 4.12, a variant of the ping-pong lemma. Define $Z_{s}$
(resp. $Z_{\overline{s}}$) to be the set of elements $M\in\mathcal{Q}_{p}$
such that $M$ is upper-balanced (resp. lower-balanced), $\tau\left(M\right)=\tau$,
$\mathrm{rd}\left(M\right)\ge3$, and $\tau2_{\tau}\left(M\right)=-\left(\tau+3\tau^{-1}\right)$.
Define $X_{0}$ to be the set of balanced elements $M\in\mathcal{Q}_{p}$
of type $\tau$ such that $\mathrm{rd}\left(M\right)\le2$ or $\tau2_{\tau}\left(M\right)\ne-\left(\tau+3\tau^{-1}\right)$.
Then, the sets $X_{0}$, $Z_{s}$ and $Z_{\overline{s}}$ are disjoint
and non-empty, since $X_{0}$ includes affine generators, $Z_{s}$
includes $\overline{s_{p,\,\tau}}\left[a\right]$ and $Z_{\overline{s}}$
includes $\overline{s_{p,\,\tau}}\left[a\right]^{-1}$.

In Theorem 4.12, Lemma 4.14 and Lemma 4.15 establish the condition
(1), Lemma 4.13 establishes the condition (2), Lemma 4.16 establishes
the condition (3), and Lemma 4.9 establishes the condition (4). The
only remaining condition is (5). Temporarily denote by $H$ the upper
multiplicative subgroup and denote by $K$ the lower multiplicative
subgroup. Suppose an element $M_{0}$ such that
\[
M_{0}\in\left(G_{p,\,\tau}\backslash H\right)\left(Z_{s}\right)\cup\left(G_{p,\,\tau}\backslash K\right)\left(Z_{\overline{s}}\right).
\]

Lemma 4.16 ensures that the action of $G_{p,\,\tau}\backslash H$
on $Z_{s}$ and $G_{p,\,\tau}\backslash K$ on $Z_{\overline{s}}$
increase the relative degree. Consequently, $\mathrm{rd}\left(M_{0}\right)>3$,
implying that the multiplicative generators, whose relative degree
is 2, cannot be $M_{0}$. This concludes the proof. $\qedhere$

\noindent \end{proof}

\noindent \begin{lemma}

Suppose an element $\tau\in\mathbb{F}_{p}$ satisfies $\tau^{4}+\tau^{2}+1=0$.
Choose a nontrivial element $M\in\mathcal{G}_{p,\,\tau}$ such that
$\mathrm{rd}\left(M\right)\ge3$.
\begin{description}
\item [{(a)}] Suppose $\tau2_{\tau}\left(M\right)=-\left(\tau+3\tau^{-1}\right)$.
\end{description}
\begin{enumerate}
\item When $M$ is upper-balanced, for any $a\in\mathbb{F}_{p}\backslash\left\{ 3\tau+\tau^{-1}\right\} $,
we have
\begin{align*}
\mathrm{rd}\left(M\right) & >\mathrm{rd}\left(M\overline{s_{p,\,\tau}}\left[a\right]^{-1}\right).
\end{align*}
\item When $M$ is lower-balanced, for any $a\in\mathbb{F}_{p}\backslash\left\{ 3\tau+\tau^{-1}\right\} $,
we have
\begin{align*}
\mathrm{rd}\left(M\right) & >\mathrm{rd}\left(M\overline{s_{p,\,\tau}}\left[a\right]\right).
\end{align*}
\end{enumerate}
\begin{description}
\item [{(b)}] Suppose $\tau2_{\tau}\left(M\right)=-\tau^{3}$.
\end{description}
\begin{enumerate}
\item When $M$ is upper-balanced, there exists a ponlynomial $f\left(x\right)\in\mathbb{F}_{p}\left[x\right]$
such that
\begin{align*}
\mathrm{rd}\left(M\right) & >\mathrm{rd}\left(M\overline{a_{p,\,\tau,\,u}}\left[f\right]\right).
\end{align*}
\item When $M$ is lower-balanced, there exists a ponlynomial $f\left(x\right)\in\mathbb{F}_{p}\left[x\right]$
such that
\begin{align*}
\mathrm{rd}\left(M\right) & >\mathrm{rd}\left(M\overline{a_{p,\,\tau,\,l}}\left[f\right]\right).
\end{align*}
\end{enumerate}
\begin{description}
\item [{(c)}] Suppose $\tau2_{\tau}\left(M\right)$ is not $-\left(\tau+3\tau^{-1}\right)$
nor $-\tau^{3}$. 
\end{description}
\begin{enumerate}
\item When $M$ is upper-balanced, we have
\begin{align*}
\mathrm{rd}\left(M\right) & >\mathrm{rd}\left(M\overline{m_{p,\,\tau,\,u}}\left[-\tau^{-1}-\tau2_{\tau}\left(M\right)\right]\right).
\end{align*}
\item When $M$ is lower-balanced, we have
\begin{align*}
\mathrm{rd}\left(M\right) & >\mathrm{rd}\left(M\overline{m_{p,\,\tau,\,l}}\left[\tau^{-1}+\tau2_{\tau}\left(M\right)\right]\right).
\end{align*}
\end{enumerate}
Moreover, the right-hand sides of the inequalities in (c)-1 and (c)-2
are always well-defined.

\noindent \end{lemma}
\begin{proof}

\noindent We prove (a)-1, (b)-1 and (c)-1 together. Suppose that $M$
is upper-balanced. Choose a representative $B=\left(\begin{array}{cc}
g_{1} & g_{2}\\
-\Phi\overline{g_{2}} & \overline{g_{1}}
\end{array}\right)$ of $M$ such that
\begin{align*}
 & g_{1}=-\tau t^{m_{1}}+a_{1}t^{m_{1}+1}+\cdots+a_{2}t^{n_{1}}+\tau^{-1}b_{3}t^{n_{1}+1},\\
 & g_{2}=t^{m_{1}}+b_{1}t^{m_{1}+1}+\cdots+b_{2}t^{n_{1}-1}+b_{3}t^{n_{1}},
\end{align*}

\noindent where the functional equation (15) gives that
\begin{align}
b_{2} & =a_{2}\tau-b_{3}\left(1+a_{1}\tau^{-1}+b_{1}\right).
\end{align}

First, assume that $\tau2_{\tau}\left(M\right)=-\left(\tau+3\tau^{-1}\right)$.
Then, we have
\begin{align}
 & Bs_{p,\,\tau}\left[a\right]^{\dagger}=\nonumber \\
 & \left(\begin{array}{cc}
\tau^{-1}\left(a_{1}+\tau b_{1}-\tau-3\tau^{3}\right)t^{m_{1}-1}+\cdots+\left(-a_{2}\tau+b_{2}+b_{3}\left(2+3\tau^{2}\right)\right)t^{n_{1}+1} & *\\
-\left(a_{1}+\tau b_{1}-\tau-3\tau^{3}\right)t^{m_{1}}+\cdots+\left(-a_{2}+b_{2}\tau^{-1}+b_{3}\left(2\tau^{-1}+3\tau\right)\right)t^{n_{1}+1} & *
\end{array}\right)^{T}.
\end{align}

By the assumption $\tau2_{\tau}\left(M\right)=-\left(\tau+3\tau^{-1}\right)$,
the coefficients of $t^{m_{1}-1}$ in the $\left(1,\,1\right)$ entry
and of $t^{m_{1}}$ in the $\left(1,\,2\right)$ entry in the right-hand
side of (68) vanish. From (67), the coefficient of $t^{n_{1}+1}$
in $\left(1,\,1\right)$ entry in the right-hand side of (68) becomes
\begin{align*}
 & \,\left(-a_{2}\tau+b_{2}+b_{3}\left(2+3\tau^{2}\right)\right)\\
= & \,-b_{3}\tau^{-1}\left(a_{1}+\tau b_{1}-\tau-3\tau^{3}\right),
\end{align*}
which also vanishes by the assumption $\tau2_{\tau}\left(M\right)=-\left(\tau+3\tau^{-1}\right)$.
The coefficient of $t^{n_{1}+1}$ in $\left(1,\,2\right)$ entry in
the right-hand side of (68) vanishes in the same reason. In summary,
we have established the inequality in (a)-1.

Next, assume that $\tau2_{\tau}\left(M\right)$ is not $-\left(\tau+3\tau^{-1}\right)$
nor $-\tau^{3}$. Then, we have
\begin{align}
 & Bm_{p,\,\tau,\,u}\left[a\right]=\left(\begin{array}{cc}
\left(-\left(1+a\tau+\tau\left(a_{1}+b_{1}\tau+\tau\right)\right)t^{m_{1}}+\cdots\right.\\
\left.+\left(a_{2}\tau-b_{2}-b_{3}\left(1+\tau^{2}-a\tau^{-1}\right)\right)t^{n_{1}+1}\right) & *\\
\left(\left(a-\tau^{3}+a_{1}+b_{1}\tau\right)t^{m_{1}}+\cdots\right.\\
\left.+\left(a_{2}\tau^{2}-b_{2}\tau+ab_{3}+b_{3}\tau^{-1}\right)t^{n_{1}}\right) & *
\end{array}\right)^{T}.
\end{align}

If we choose $a$ as $-\tau^{-1}-\tau2_{\tau}\left(M\right)$, the
coefficients of $t^{m_{1}}$ in the $\left(1,\,1\right)$ and $\left(1,\,2\right)$
entries in (69) vanish. Applying (67), the coefficients of $t^{n_{1}+1}$
in the $\left(1,\,1\right)$ entry and of $t^{n_{1}}$ in the $\left(1,\,2\right)$
entry in the right-hand side of (69) also vanish under the same condition
$a=-\tau^{-1}-\tau2_{\tau}\left(M\right)$. We need to verify that
$\overline{m_{p,\,\tau,\,u}}\left[a\right]$ is well-defined as a
multiplicative generator, or in other words
\begin{align}
-\tau^{-1}-\tau2_{\tau}\left(M\right) & \ne\pm\left(\tau+2\tau^{-1}\right).
\end{align}

The condition (70) holds always from the assumption $\tau2_{\tau}\left(M\right)$
is not $-\left(\tau+3\tau^{-1}\right)$ nor $-\tau^{3}$, which establishes
(c)-1.

Finally, we prove (b)-1. Assume that $\tau2_{\tau}\left(M\right)=-\tau^{3}$.
Then, we have
\begin{align}
Bg_{p} & \left[\tau\right]^{\dagger}=\left(\begin{array}{cc}
\left(a_{1}+b_{1}\tau+\tau+\tau^{3}\right)t^{m_{1}}+\cdots+\left(-a_{2}\tau^{2}+b_{2}\tau-b_{3}\tau^{3}\right)t^{n_{1}} & *\\
\left(1-a_{1}\tau-b_{1}\tau^{2}\right)t^{m_{1}+1}+\cdots+\left(-a_{2}\tau+b_{2}-b_{3}\tau^{2}\right)t^{n_{1}} & *
\end{array}\right)^{T}.
\end{align}

By the assumption $\tau2_{\tau}\left(M\right)=-\tau^{3}$ and (67),
the coefficients of $t^{m_{1}}$, $t^{n_{1}}$ in the $\left(1,\,1\right)$
entry and of $t^{m_{1}+1}$, $t^{n_{1}}$ in the $\left(1,\,2\right)$
entry in the right-hand side of (71) vanish. Therefore, we have $\mathrm{rd}\left(M\right)-3\ge\mathrm{rd}\left(Bg_{p}\left[\tau\right]^{\dagger}\right)$.

\noindent \begin{claim2}

Suppose the same condition of (b)-1. Then, for some $\alpha\in\mathbb{F}_{p}$,
we have
\begin{align}
Bg_{p}\left[\tau\right]^{\dagger} & =\left(\begin{array}{cc}
\alpha t^{m_{1}+1}+\cdots-\alpha b_{3}t^{n_{1}-1} & *\\
-\tau\alpha t^{m_{1}+2}+\cdots-\tau^{-1}\alpha b_{3}t^{n_{1}-1} & *
\end{array}\right)^{T}.
\end{align}

\noindent \end{claim2}
\begin{proofclaim2}

\noindent When $\mathrm{rd}\left(B\right)\le4$, (72) is direct by
computation. Suppose $\mathrm{rd}\left(B\right)\ge5$ and rewrite
the entries as

\noindent 
\begin{align*}
 & g_{1}=-\tau t^{m_{1}}+a_{1}t^{m_{1}+1}+a_{3}t^{m_{1}+2}+\cdots+a_{4}t^{n_{1}-1}+a_{2}t^{n_{1}}+\tau^{-1}b_{3}t^{n_{1}+1},\\
 & g_{2}=t^{m_{1}}+b_{1}t^{m_{1}+1}+b_{4}t^{m_{1}+2}+\cdots+b_{5}t^{n_{1}-2}+b_{2}t^{n_{1}-1}+b_{3}t^{n_{1}},
\end{align*}

\noindent where the functional equation (15) gives that
\begin{align}
b_{5} & =-a_{1}a_{2}+a_{4}\tau-b_{2}-b_{1}b_{2}-b_{3}-a_{3}b_{3}\tau^{-1}-b_{1}b_{3}-b_{3}b_{4}.
\end{align}

Then, we have
\begin{align}
 & Bg_{p}\left[\tau\right]^{\dagger}=\nonumber \\
 & \left(\begin{array}{cc}
\left(a_{3}+\tau\left(1-a_{1}\tau+b_{1}+b_{4}\right)\right)t^{m_{1}+1}+\cdots+\left(a_{2}+\tau\left(b_{2}+b_{3}+b_{5}-a_{4}\tau\right)\right)t^{n_{1}-1} & *\\
\left(b_{1}-a_{3}\tau-b_{4}\tau^{2}\right)t^{m_{1}+2}+\cdots+\left(-a_{4}\tau-b_{2}\tau^{2}+b_{5}\right)t^{n_{1}-1} & *
\end{array}\right)^{T},
\end{align}
where we cancelled the coefficients of $t^{m_{1}}$, $t^{n_{1}}$
in the $\left(1,\,1\right)$ entry and of $t^{m_{1}+1}$, $t^{n_{1}}$
in the $\left(1,\,2\right)$ entry by the assumption $\tau2_{\tau}\left(M\right)=-\tau^{3}$.
Applying (67) and (73), the coefficient of $t^{n_{1}-1}$ in (74)
becomes

\noindent 
\begin{align*}
 & \,a_{2}+\tau\left(b_{2}+b_{3}+b_{5}-a_{4}\tau\right)\\
= & \,a_{2}-a_{1}a_{2}\tau-a_{2}b_{1}\tau^{2}-a_{3}b_{3}+a_{1}b_{1}b_{3}+b_{1}^{2}b_{3}\tau-b_{3}b_{4}\tau\\
= & \,-b_{3}\left(a_{3}-b_{1}\tau^{-1}+b_{4}\tau\right),
\end{align*}

\noindent where the second equality follows from the assumption $a_{1}=-b_{1}\tau+\tau^{-1}$.
Therefore, comparing the coefficients in (74), we establish (72) by
choosing $\alpha=a_{3}-b_{1}\tau^{-1}+b_{4}\tau$. $\qed$

\noindent \end{proofclaim2}

We return to the proof of (b)-1, and suppose that $\alpha\ne0$ in
(72). Then, we have
\begin{align}
\left(\tau\alpha\right)^{-1}Bg_{p}\left[\tau\right]^{\dagger}\left(\begin{array}{cc}
t-t^{-1} & 0\\
0 & 0
\end{array}\right)g_{p}\left[\tau\right] & =\left(\begin{array}{cc}
\tau t^{m_{1}}+\cdots-\tau^{-1}b_{3}t^{n_{1}+1} & *\\
-t^{m_{1}}+\cdots-b_{3}t^{n_{1}} & *
\end{array}\right)^{T}.
\end{align}

From (75) and the definition of additive generators (Definition 3.6),
we have
\begin{align*}
\mathrm{rd}\left(M\right) & >\mathrm{rd}\left(M\overline{a_{p,\,\tau,\,u}}\left[\left(\tau\alpha\right)^{-1}\right]\right).
\end{align*}

On the other hand, suppose that $\alpha=0$ in (72). As in the proof
of (a) of Theorem 3.20, consider the Laurent polynomial entry
\begin{align*}
\left(Bg_{p}\left[\tau\right]^{\dagger}\right)_{11} & =\tau_{l}t^{m_{0}}+\cdots+\tau_{h}t^{n_{0}},
\end{align*}
where $m_{0}$ is the power of $t$ in the lowest degree term, $n_{0}$
is the power of $t$ in the highest degree term, and the coefficients
$\tau_{l},\,\tau_{h}$ are units. Denote by $m$ the integer
\begin{align*}
\min\left(\left(m_{0}-m_{1}-1\right),\,\left(n_{1}-n_{0}-1\right)\right),
\end{align*}
satisfying $m\ge1$ by (72). Without loss of generality, suppose $m$
is equal to $\left(m_{0}-m_{1}-1\right)$. Then, the $\left(1,\,1\right)$-entry
of a matrix $\tau_{l}^{-1}\left(t^{-m}+t^{m}\right)Bg_{p}\left[\tau\right]^{\dagger}$
has $t^{m_{1}+1}$ as the lowest degree term. Choose a polynomial
$f\left(x\right)\in\mathbb{F}_{p}\left[x\right]$ such that $f\left(t+t^{-1}\right)=\tau^{-1}\tau_{l}^{-1}\left(t^{-m+1}+t^{m-1}\right)$.
As in (75), we have $\mathrm{rd}\left(M\right)>\mathrm{rd}\left(M\overline{a_{p,\,\tau,\,u}}\left[f\right]\right)$,
which establishes (b)-1.

The cases (a)-2, (b)-2 and (c)-2 are proven in the same way. $\qedhere$

\noindent \end{proof}

\noindent \begin{prooftheorem41}

\noindent By Lemma 4.17, for an element $\tau\in\mathbb{F}_{p}$ satisfies
$\tau^{4}+\tau^{2}+1=0$, we already know that the type group $\mathcal{G}_{p,\,\tau}$
is an internal HNN extension with $G_{p,\,\tau}$ as the base group
and $\overline{s_{p,\,\tau}}\left[a\right]$ as the stable letter
for any $a\in\mathbb{F}_{p}\backslash\left\{ 3\tau+\tau^{-1}\right\} $.
Thus, it suffices to show that
\begin{description}
\item [{(a)}] \noindent the subgroups $\mathcal{G}_{p,\,0},\cdots,\,\mathcal{G}_{p,\,p-1}$
generate the whole group $\mathcal{Q}_{p}$, and 
\item [{(b)}] \noindent the subgroup generated by $\mathcal{G}_{p,\,0},\cdots,\,\mathcal{G}_{p,\,p-1}$
is the free product of them.
\end{description}
\noindent \begin{proofofa}

\noindent We use the induction on the relative degree. Choose an element
$M\in\mathcal{Q}_{p}$. We may assume that $M$ is balanced by Lemma
3.3. When $\mathrm{rd}\left(M\right)=1$, as in the proof of (a) of
Theorem 3.20, the only possible candidates have form $\overline{g_{p}}\left[0\right]^{k}\overline{g_{p}}\left[r\right]^{\pm1}$
for an integer $k$ and $r\in\mathbb{F}_{p}^{\times}$ such that $r^{4}+r^{2}+1\ne0$.
For an integer $N\ge1$, suppose that the subgroups $\mathcal{G}_{p,\,0},\cdots,\,\mathcal{G}_{p,\,p-1}$
generate every element $M_{0}\in\mathcal{Q}_{p}$ such that $\mathrm{rd}\left(M_{0}\right)\le N$,
and suppose $M\in\mathcal{Q}_{p}$ such that $\mathrm{rd}\left(M\right)=N+1$.
We may assume that $M$ is balanced by Lemma 3.3. Suppose that $\tau\left(M\right)$
satisfies
\begin{align*}
\tau\left(M\right)^{4}+\tau\left(M\right)^{2}+1 & \ne0.
\end{align*}

When $M$ is upper-balanced, we have $\mathrm{rd}\left(M\overline{g_{p}}\left[\tau\left(M\right)\right]^{-1}\right)<\mathrm{rd}\left(M\right)$.
When $M$ is lower-balanced, we have $\mathrm{rd}\left(M\overline{g_{p}}\left[\tau\left(M\right)\right]\right)<\mathrm{rd}\left(M\right)$.

Suppose that $\tau\left(M\right)$ satisfies
\[
\tau\left(M\right)^{4}+\tau\left(M\right)^{2}+1=0.
\]

Let $\tau$ denote the type $\tau\left(M\right)$. When $\mathrm{rd}\left(M\right)=2$,
direct computation shows that the only possible candidates are a product
of $\overline{g_{p}}\left[0\right]^{k}$ and a multiplicative generator
for $\tau$, for some integer $k$. Suppose $\mathrm{rd}\left(M\right)\ge3$.
For any possible cases of the second-order type $\tau2_{\tau}\left(M\right)$,
Lemma 4.18 ensures that there exists an element $M'\in\mathcal{G}_{\tau}$
such that $\mathrm{rd}\left(MM'\right)<\mathrm{rd}\left(M\right)$.
We conclude the proof by induction. $\qed$

\noindent \end{proofofa}
\begin{proofofb}

\noindent We use the usual ping-pong lemma. Define $S_{p,\,0}$ to
be the set of unbalanced elements in $\mathcal{Q}_{p}$ and define
$S_{p,\,r}$ to be the set of balanced elements having type $r$ in
$\mathcal{Q}_{p}$. For a pair of integers $r_{1},\,r_{2}$ such that
$r_{1}\ne r_{2}$ and a pair of nontrivial elements $M_{1}\in S_{p,\,r_{1}}$
and $M_{2}\in\mathcal{G}_{p,\,r_{2}}$, it suffices to show that
\begin{align*}
M_{1}M_{2} & \in S_{p,\,r_{2}}.
\end{align*}

By Lemma 3.13, it suffices to consider the signatures of $M_{2}$.

Suppose $r_{2}$ satisfies that $r_{2}^{4}+r_{2}^{2}+1\not\equiv0$
(mod $p$). Then, the group $\mathcal{G}_{p,\,r_{2}}$ is generated
by the elementary generator $\overline{g_{p}}\left[r_{2}\right]$.
Since any nontrivial power of $\overline{g_{p}}\left[r_{2}\right]$
is oddly balanced, and has signature $-r_{2}^{-2}$ when upper-balanced
or $-r_{2}^{2}$ when lower-balanced (\citep[the proof of Theorem 4.9,][]{lee2024salters}),
the conditions (4) and (6) in Lemma 3.13 reduces to $r_{1}\ne r_{2},$
as supposed.

On the other hand, suppose $r_{2}$ satisfies that $r_{2}^{4}+r_{2}^{2}+1\equiv0$
(mod $p$). Then, by Lemma 4.17, the group $\mathcal{G}_{p,\,r_{2}}$
is an HNN extension. We consider Britton's normal form of elements
\begin{align}
g_{0}s^{i_{1}}g_{1}s^{i_{2}}\cdots g_{n-1}s^{i_{n}}g_{n},
\end{align}
where $s$ is a stable generator, $i_{j_{1}}\in\mathbb{Z}$ and $g_{j_{2}}\in G_{p,\,r_{2}}$
for every pair of integers $j_{1},\,j_{2}$ such that $1\le j_{1}\le n$
and $0\le j_{2}\le n$.

By Lemma 4.8, a nontrivial element in the base group $G_{p,\,r_{2}}$
has signature $-1$ when evenly balanced, $-r_{2}^{-2}$ when oddly
upper-balanced, and $-r_{2}^{2}$ when oddly lower-balanced. Therefore,
Lemma 4.13, Lemma 4.14, Lemma 4.15 and Lemma 4.16, inductively from
(76), ensure that any nontrivial element in the group $\mathcal{G}_{p,\,r_{2}}$
also has signature $-1$ when evenly balanced, $-r_{2}^{-2}$ when
oddly upper-balanced, and $-r_{2}^{2}$ when oddly lower-balanced.

Consequently, the conditions (2), (4), (6) and (8) in Lemma 3.13 are
reduced to $r_{1}\ne r_{2}$, as supposed. $\qed$

\noindent \end{proofofb}
\end{prooftheorem41}

\noindent \begin{remark}

The quaternionic group $\mathcal{Q}_{p}$ is closely related to the
projective unitary group defined by the relation
\[
\overline{M}\left[\begin{array}{cc}
1 & 0\\
0 & \Phi
\end{array}\right]M^{T}=\left[\begin{array}{cc}
1 & 0\\
0 & \Phi
\end{array}\right].
\]

In fact, it is more natural to define the \emph{projective similitude
group} as
\begin{align*}
\mathcal{\mathcal{S}}_{p} & :=\left\{ M\in\mathrm{PGL}\left(2,\,\mathbb{F}_{p}\left[t,\,t^{-1}\right]\right)\;:\;\overline{M}\left[\begin{array}{cc}
1 & 0\\
0 & \Phi
\end{array}\right]M^{T}=k\left[\begin{array}{cc}
1 & 0\\
0 & \Phi
\end{array}\right],\:k\in\mathbb{F}_{p}^{\times}\right\} ,
\end{align*}
where the similitude relation is a generalization of the unitarity
relation.

Then, the quaternionic group $\mathcal{Q}_{p}$ is an index 4 subgroup
of $\mathcal{S}_{p}$ when $p>2$, and an index 2 subgroup of $\mathcal{S}_{p}$
when $p=2$. The coset representatives besides the identity are
\begin{align*}
\left[\begin{array}{cc}
1 & 0\\
0 & -1
\end{array}\right],\,\left[\begin{array}{cc}
1 & 0\\
0 & t
\end{array}\right],\,\left[\begin{array}{cc}
1 & 0\\
0 & -t
\end{array}\right].
\end{align*}
\end{remark}
\begin{remark}

As a final remark of this section, we mention that the techniques
we have used here are easily generalizable to more general similitude
groups with similitude relation
\begin{align*}
\overline{X}\left(\begin{array}{cc}
1 & 0\\
0 & a+b\left(t+t^{-1}\right)
\end{array}\right)X^{T} & =k\left(\begin{array}{cc}
1 & 0\\
0 & a+b\left(t+t^{-1}\right)
\end{array}\right).
\end{align*}

Temporarily denote by $\Phi_{a,\,b}$ the value $a+b\left(t+t^{-1}\right)$.
In general, for any field $F$, its projectivization has the corresponding
\emph{quaternionic group}
\begin{align*}
\mathcal{Q}\left[F,\,a,\,b\right] & :=\,\left\{ M\in\mathrm{PGL}\left(2,\,F\left[t,\,t^{-1}\right]\right)\::\:M=\left[\begin{array}{cc}
g_{1} & g_{2}\\
-\Phi_{a,\,b}\overline{g_{2}} & \overline{g_{1}}
\end{array}\right],\,\mathrm{where}\;g_{1},\,g_{2}\in F\left[t,\,t^{-1}\right]\right\} 
\end{align*}
as a subgroup of index 2 or 4. The structure of this group is analyzable
in the same way as Theorem 3.11. For example, the elementary generators
for this group are given by the projectivization of
\begin{align*}
g\left[r\right] & =\left(\begin{array}{cc}
bt-r^{2} & r\\
-r\Phi_{a,\,b} & bt^{-1}-r^{2}
\end{array}\right),
\end{align*}
whose determinant is $r^{4}+ar^{2}+b^{2}$. If there is no solution
to the equation $x^{4}+ax^{2}+b^{2}=0$ over the base field $F$,
the elementary generators freely generate the entire quaternionic
group $\mathcal{Q}\left[F,\,a,\,b\right]$. If there is a solution
to the equation, its multiplicity must be 1, 2, or 4. The case of
multiplicity 4 is possible only when the field characteristic is 2.
The cases where the root multipilcity is 2 or 4 are dealt with analogously
to the case $p=3$ in Theorem 3.11. The other case is analogous to
the cases $p\equiv1$ (mod 6) in Theorem 3.11.

Moreover, for a palindromic Laurent polynomial $f_{1}\in F\left[t,\,t^{-1}\right]$
and any Laurent polynomial $f_{2}\in F\left[t,\,t^{-1}\right]$, the
injection given by
\begin{align*}
\left[\begin{array}{cc}
g_{1} & g_{2}\\
-\overline{g_{2}}\left(f_{1}f_{2}\overline{f_{2}}\right) & \overline{g_{1}}
\end{array}\right] & \mapsto\left[\begin{array}{cc}
g_{1} & f_{2}g_{2}\\
-\overline{f_{2}g_{2}}f_{1} & \overline{g_{1}}
\end{array}\right]
\end{align*}
makes more general analysis possible, if one takes a finite algebraic
extension of $F$.

\noindent \end{remark}

\section{Proof of Theorem 1.1}

In Section 2, for each prime $p$, we defined the formal Burau group
$\mathcal{B}_{p}$, containing the Burau image group $\beta_{3,\,p}\left(B_{3}\right)$
and defined two homomorphisms $\overline{\phi_{p}}$ and $\overline{\rho_{p}}$,
where the images have entries in a ring of characteristic $p$. On
the other hand, we may evaluate the entries of the original Burau
image group $\beta_{3}\left(B_{3}\right)$ at $t=-1$. As in Section
2, we define the evaluation map:
\begin{align*}
\mathrm{eval}_{-1} & :\beta_{3}\left(B_{3}\right)\to\mathrm{GL}\left(3,\,\mathbb{Z}\right).
\end{align*}

For a matrix $B\in\mathrm{eval}_{-1}\left(\beta_{3}\left(B_{3}\right)\right)$,
we write the entries in terms of the four corner entries as
\begin{align*}
B & =\left(\begin{array}{ccc}
B_{11} & 1-B_{11}-B_{13} & B_{13}\\
-1+B_{11}+B_{31} & 3-B_{11}-B_{13}-B_{31}-B_{33} & -1+B_{13}+B_{33}\\
B_{31} & 1-B_{31}-B_{33} & B_{33}
\end{array}\right).
\end{align*}

\noindent \begin{definition}

Define $\rho:\mathrm{eval}_{-1}\left(\beta_{3}\left(B_{3}\right)\right)\to\mathrm{GL}\left(2,\,\mathbb{Z}\right)$
by
\begin{align*}
B & \mapsto\left(\begin{array}{cc}
1-B_{13} & 1-B_{11}\\
1-B_{33} & 1-B_{31}
\end{array}\right).
\end{align*}
\end{definition}
\begin{lemma}

The map $\rho$ is an injective group homomorphism preserving the
determinant.

\noindent \end{lemma}
\begin{proof}

\noindent See \citep[Lemma 3.8]{lee2024salters}. $\qedhere$

\noindent \end{proof}

Put $\pi$ as the natural projection
\begin{align*}
\pi & :\mathrm{GL}\left(2,\,\mathbb{Z}\left[t,\,t^{-1}\right]\right)\to\mathrm{PGL}\left(2,\,\mathbb{Z}\left[t,\,t^{-1}\right]\right).
\end{align*}

Define a map $\overline{\rho}$ to be the composition $\pi\circ\rho\circ\mathrm{eval}_{-1}$.
Recall that we defined $\Delta$ to be a braid $\left(\sigma_{1}\sigma_{2}\sigma_{1}\right)^{2}$
in $B_{3}$.

\noindent \begin{lemma}

The kernel of $\overline{\rho}:\beta_{3}\left(B_{3}\right)\to\mathrm{PSL}\left(2,\mathbb{Z}\right)$
is the center of $\beta_{3}\left(B_{3}\right)$, generated by $\beta_{3}\left(\Delta\right)$.
In other words, $\overline{\rho}$ induces the short exact sequence:
\begin{align*}
1 & \longrightarrow\mathbb{Z}\longrightarrow\beta_{3}\left(B_{3}\right)\longrightarrow\mathrm{PSL}\left(2,\mathbb{Z}\right)\longrightarrow1.
\end{align*}

\noindent \end{lemma}
\begin{proof}

\noindent See \citep[Lemma 3.9]{lee2024salters}. $\qedhere$

\noindent \end{proof}

In Section 1, for each prime $p$, we defined a map
\begin{align*}
\mu_{p} & :\mathrm{PSL}\left(2,\,\mathbb{Z}\right)\to\mathrm{PGL}\left(2,\,\mathbb{F}_{p}\left[t,\,t^{-1},\,\left(1+t\right)^{-1}\right]\right),
\end{align*}
to be
\begin{align*}
\mu_{p}\left[\begin{array}{cc}
1 & 0\\
1 & 1
\end{array}\right] & =\left[\begin{array}{cc}
1 & 0\\
0 & -t
\end{array}\right],\;\mu_{p}\left[\begin{array}{cc}
1 & -1\\
0 & 1
\end{array}\right]=\left[\begin{array}{cc}
-\frac{t^{2}}{1+t} & \frac{t}{1+t}\\
\frac{1+t+t^{2}}{1+t} & \frac{1}{1+t}
\end{array}\right].
\end{align*}

One can easily verify the well-definedness of $\mu_{p}$ by computing
a finite set of relations in the presentation of the modular group.
Recall that from (13) and (14) we have
\begin{align}
 & \overline{\phi_{p}}\left(\beta_{3,\,p}\left(\sigma_{1}\right)\right)=\left[\begin{array}{cc}
-\frac{t^{2}}{1+t} & \frac{t}{1+t}\\
\frac{1+t+t^{2}}{1+t} & \frac{1}{1+t}
\end{array}\right],\;\overline{\phi_{p}}\left(\beta_{3,\,p}\left(\sigma_{2}\right)\right)=\left[\begin{array}{cc}
1 & 0\\
0 & -t
\end{array}\right],\\
 & \overline{\rho_{p}}\left(\beta_{3,\,p}\left(\sigma_{1}\right)\right).=\left[\begin{array}{cc}
1 & -1\\
0 & 1
\end{array}\right],\;\overline{\rho_{p}}\left(\beta_{3,\,p}\left(\sigma_{2}\right)\right)=\left[\begin{array}{cc}
1 & 0\\
1 & 1
\end{array}\right].
\end{align}

Define a map
\begin{align*}
M_{p} & :\mathrm{GL}\left(3,\,\mathbb{Z}\left[t,\,t^{-1}\right]\right)\to\mathrm{GL}\left(3,\,\mathbb{F}_{p}\left[t,\,t^{-1}\right]\right)
\end{align*}
to be the usual modulo $p$ map at each entry. Recall that in Section
1 we also defined another modulo $p$ map
\begin{align*}
m_{p}:\mathrm{PGL}\left(2,\,\mathbb{Z}\left[t,\,t^{-1},\,\left(1+t\right)^{-1}\right]\right) & \to\mathrm{PGL}\left(2,\,\mathbb{F}_{p}\left[t,\,t^{-1},\,\left(1+t\right)^{-1}\right]\right).
\end{align*}

From Corollary 2.11 and Lemma 5.3, we have
\begin{align*}
\ker\overline{\rho} & =\left\langle \beta_{3}\left(\Delta\right)\right\rangle =\ker\overline{\phi_{p}}|_{\beta_{3,\,p}},
\end{align*}
from which we can reconstruct the map $\mu_{p}$ taking the cosets
$\left\langle \beta_{3}\left(\Delta\right)\right\rangle A$ in $\beta_{3}\left(\Delta\right)$
to $\overline{\phi_{p}}\left(M_{p}\left(A\right)\right)$. Here is
one of the key lemmas to prove Theorem 1.1.

\noindent \begin{lemma}

The map $\mu_{p}$ is injective if and only if the representation
$\beta_{3,\,p}$ is faithful.

\noindent \end{lemma}
\begin{proof}

\noindent By the construction of the map $\mu_{p}$, we have a commutative
diagram:
\begin{align*}
\xymatrix{1\ar[r] & \mathbb{Z}\ar[r]\ar[d]^{=} & \beta_{3}\left(B_{3}\right)\ar[d]^{M_{p}}\ar[r]^{\overline{\rho}} & \mathrm{PSL}\left(2,\mathbb{Z}\right)\ar[d]^{\mu_{p}}\ar[r] & 1\\
1\ar[r] & \mathbb{Z}\ar[r] & \beta_{3,\,p}\left(B_{3}\right)\ar[r]^{\overline{\phi_{p}}} & \overline{\phi_{p}}\left(\beta_{3,\,p}\left(B_{3}\right)\right)\ar[r] & 1
}
\end{align*}
where the map $M_{p}$ restricted on the kernel is an isomorphism.
The four lemma establishes the result desired. $\qedhere$

\noindent \end{proof}

\noindent \begin{definition}

For each prime $p$, define two isomorphic subgroups $\Gamma^{0}\left(p\right)$
and $\Gamma_{0}\left(p\right)$ of the modular group $\mathrm{PSL}\left(2,\,\mathbb{Z}\right)$
as
\begin{align*}
 & \Gamma^{0}\left(p\right):=\left\{ \left[\begin{array}{cc}
a & b\\
c & d
\end{array}\right]\in\mathrm{PSL}\left(2,\,\mathbb{Z}\right)\;:\;b\equiv0\;(\mathrm{mod}\:p)\right\} ,\\
 & \Gamma_{0}\left(p\right):=\left\{ \left[\begin{array}{cc}
a & b\\
c & d
\end{array}\right]\in\mathrm{PSL}\left(2,\,\mathbb{Z}\right)\;:\;c\equiv0\;(\mathrm{mod}\:p)\right\} .
\end{align*}

The group $\Gamma_{0}\left(p\right)$ is called the \emph{Hecke congruence
subgroup} of level $p$.

\noindent \end{definition}

Lemma 2.13 implies that for an element $M\in\mathrm{PSL}\left(2,\,\mathbb{Z}\right)$,
the subgroup $\Gamma^{0}\left(p\right)$ includes $M$ if and only
if $\mu_{p}\left(M\right)$ is contained in $\mathrm{PGL}\left(2,\,\mathbb{F}_{p}\left[t,\,t^{-1}\right]\right)$.
The following provides a computable version of Lemma 5.4.

\noindent \begin{lemma}

The map $\mu_{p}$ is injective if and only if the restriction $\mu_{p}|_{\Gamma^{0}\left(p\right)}$
is injective.

\noindent \end{lemma}
\begin{proof}

\noindent The set of cosets $\left\{ \Gamma^{0}\left(p\right)M\right\} _{M\in\mathrm{PSL}\left(2,\,\mathbb{Z}\right)}$
has a bijection to the projective line over the finite field $\mathbb{F}_{p}$.
Therefore, there are exactly $p+1$ cosets with representatives:
\begin{align*}
\left[\begin{array}{cc}
1 & 0\\
0 & 1
\end{array}\right],\,\left[\begin{array}{cc}
1 & -1\\
0 & 1
\end{array}\right],\cdots,\,\left[\begin{array}{cc}
1 & -\left(p-1\right)\\
0 & 1
\end{array}\right],\,\left[\begin{array}{cc}
0 & -1\\
1 & 0
\end{array}\right].
\end{align*}

From the fact that
\begin{align}
\overline{\rho_{p}}\left(\beta_{3,\,p}\left(\sigma_{1}\sigma_{2}\sigma_{1}\right)\right) & =\left[\begin{array}{cc}
0 & -1\\
1 & 0
\end{array}\right],
\end{align}
the images of these representatives under $\mu_{p}$ are
\begin{align*}
\left[\begin{array}{cc}
1 & 0\\
0 & 1
\end{array}\right],\,\left[\begin{array}{cc}
-\frac{t^{2}}{1+t} & \frac{t}{1+t}\\
\frac{1+t+t^{2}}{1+t} & \frac{1}{1+t}
\end{array}\right],\cdots,\,\left[\begin{array}{cc}
-\frac{t^{2}}{1+t} & \frac{t}{1+t}\\
\frac{1+t+t^{2}}{1+t} & \frac{1}{1+t}
\end{array}\right]^{p-1},\,\left[\begin{array}{cc}
\frac{t}{1+t} & \frac{t}{1+t}\\
\frac{1+t+t^{2}}{1+t} & -\frac{t}{1+t}
\end{array}\right],
\end{align*}
where for each integer $n$ such that $1\le n\le p-1$, we have
\begin{align}
\left[\begin{array}{cc}
-\frac{t^{2}}{1+t} & \frac{t}{1+t}\\
\frac{1+t+t^{2}}{1+t} & \frac{1}{1+t}
\end{array}\right]^{n} & =\left[\begin{array}{cc}
\frac{t+\left(-t\right)^{n}\left(1+t+t^{2}\right)}{\left(1+t\right)^{2}} & \frac{t+\left(-t\right)^{n+1}}{\left(1+t\right)^{2}}\\
\frac{\left(1+t+t^{2}\right)\left(1-\left(-t\right)^{n}\right)}{\left(1+t\right)^{2}} & \frac{1+t+t^{2}+t\left(-t\right)^{n}}{\left(1+t\right)^{2}}
\end{array}\right],
\end{align}
which is canonically proven through diagonalization. On the other
hand, we also have
\begin{equation}
\left[\begin{array}{cc}
\frac{t}{1+t} & \frac{t}{1+t}\\
\frac{1+t+t^{2}}{1+t} & -\frac{t}{1+t}
\end{array}\right]^{-1}\left[\begin{array}{cc}
-\frac{t^{2}}{1+t} & \frac{t}{1+t}\\
\frac{1+t+t^{2}}{1+t} & \frac{1}{1+t}
\end{array}\right]^{n}=\left[\begin{array}{cc}
-\frac{1}{t+t^{2}} & -\frac{1}{t+t^{2}}\\
\frac{\left(1+t+t^{2}\right)t\left(-t\right)^{n-3}}{1+t} & \frac{\left(-t\right)^{n}}{t+t^{2}}
\end{array}\right],
\end{equation}
which directly follows from (80). By Lemma 2.13, it suffices to check
whether these elements are in $\mathrm{PGL}\left(2,\,\mathbb{F}_{p}\left[t,\,t^{-1}\right]\right)$.
The $\left(1,\,2\right)$ entry in the right-hand side of (81) is
not a Laurent polynomial, and by L'Hopital's rule, the $\left(1,\,2\right)$
entry in the right-hand side of (80) is a Laurent polynomial over
$\mathbb{F}_{p}$ if and only if
\begin{align*}
1-\left(n+1\right)\left(-t\right)^{n}|_{t=-1}=0,
\end{align*}
or $p$ divides $n$. Since the elements in (80) and (81) have determinant
$\left(-t\right)^{k}$ for some integer $k$ up to square unit, we
have exactly $p+1$ cosets of $\mu_{p}\left(\Gamma^{0}\left(p\right)\right)$.
Consequently, the injectivity of $\mu_{p}|_{\Gamma^{0}\left(p\right)}$
implies the injectivity of $\mu_{p}$. $\qedhere$

\noindent \end{proof}
\begin{corollary}

The representation $\beta_{3,\,p}$ is faithful if and only if the
restriction $\mu_{p}|_{\Gamma^{0}\left(p\right)}$ is injective.

\noindent \end{corollary}
\begin{proof}

\noindent Direct from Lemma 5.4 and Lemma 5.6. $\qedhere$

\noindent \end{proof}

To apply Corollary 5.7, we need to understand the structure of the
group $\Gamma^{0}\left(p\right)$. Rademacher \citep{MR3069525} proved
that the group $\Gamma_{0}\left(p\right)$, isomorphic to $\Gamma^{0}\left(p\right)$,
is a free product of cyclic groups, whose order could be infinite,
2, or 3. Following Rademacher's notation, but transposed, in $\mathrm{PSL}\left(2,\,\mathbb{Z}\right)$
we define

\begin{equation}
S:=\left[\begin{array}{cc}
1 & 0\\
1 & 1
\end{array}\right],\;T:=\left[\begin{array}{cc}
0 & -1\\
1 & 0
\end{array}\right].
\end{equation}

Fix a prime $p$. For an integer $k$ such that $1\le k\le p-1$,
define another unique integer $k_{*}$ such that $1\le k_{*}\le p-1$
and $kk_{*}\equiv-1$ (mod $p$). For each integer $k$ such that
$1\le k\le p-1$, define $V_{k}=T^{-1}S^{-k_{*}}TS^{k}T$. Then, we
have
\begin{align*}
V_{k} & =\left[\begin{array}{cc}
-k_{*} & kk_{*}+1\\
-1 & k
\end{array}\right],
\end{align*}
which is included in $\Gamma^{0}\left(p\right)$.

\noindent \begin{theorem}

For each prime $p$, the group $\Gamma^{0}\left(p\right)$ is generated
by $S,\,V_{1},\,V_{2},\,\cdots,V_{p-1}$. When $p>3$, we need only
$S$ and $2\left\lfloor \frac{p}{12}\right\rfloor +2$ elements among
$V_{1},\cdots,\,V_{p-1}$ to generate the entire group. It is a minimal
set of generators, and the relations in that case are only
\begin{align*}
V_{\chi_{1}}^{2} & =1,\;V_{\chi_{2}}^{2}=1
\end{align*}
when $p\equiv1$ (mod 4) and
\begin{align*}
V_{\lambda_{1}}^{3} & =1,\;V_{\lambda_{2}}^{3}=1
\end{align*}
when $p\equiv1$ (mod 3), where the subscripts are defined to the
integers satisfying $\chi_{i}^{2}\equiv-1$, $\left(2\lambda_{i}-1\right)^{2}\equiv-3$
(mod $p$) for $i=1,\,2$.

\noindent \end{theorem}
\begin{proof}

\noindent See \citep[Satz, p. 146]{MR3069525}. $\qedhere$

\noindent \end{proof}

From now on, let $C_{l}$ denote the cyclic group of order $l$, and
let $F_{l}$ denote the free group of rank $l$. Define a quotient
map
\begin{align*}
\omega & :\mathrm{PSL}\left(2,\,\mathbb{Z}\right)\to C_{2}
\end{align*}
by adding relations (cf. (78))
\begin{align*}
\overline{\rho}\left(\beta_{3}\left(\sigma_{1}\right)\right)^{2} & =1,\;\overline{\rho}\left(\beta_{3}\left(\sigma_{1}\right)\right)=\overline{\rho}\left(\beta_{3}\left(\sigma_{2}\right)\right).
\end{align*}

Then, the map $\omega$ can also be alternatively defined by taking
the generating set $\left\{ S,\,T\right\} $ in (82) by adding relations
\begin{align*}
S^{2} & =1,\;S=T,
\end{align*}
since $S=\overline{\rho}\left(\beta_{3}\left(\sigma_{2}\right)\right)$
and $T=\overline{\rho}\left(\beta_{3}\left(\sigma_{1}\sigma_{2}\sigma_{1}\right)\right)$.
For each prime $p$, define a group
\begin{align*}
\mathcal{K}_{p} & :=\ker\left(\omega|_{\Gamma^{0}\left(p\right)}\right).
\end{align*}

\noindent \begin{lemma}

Suppose $p\equiv2,\,5$ (mod 6). Then, $\mathcal{K}_{p}$ is free.
When $p=2$, the rank is 2. When $p\equiv5$ (mod 6), the rank is
$\frac{p+4}{3}$. On the other hand, the group $\mathcal{K}_{3}$
is a free product isomorphic to $\mathbb{Z}*C_{3}*C_{3}$. For the
cases where $p\equiv1$ (mod 6), the group $\mathcal{K}_{p}$ is a
free product of isomorphic to $F_{\frac{p-4}{3}}*C_{3}*C_{3}*C_{3}*C_{3}$.
In each case, there is an algorithm to decide the generating set of
the kernel, as words in $S$ and $\left\{ V_{i}\right\} _{1\le i\le p-1}$.

\noindent \end{lemma}
\begin{proof}

\noindent We use the usual Reidemeister\textendash Schreier process.
When an element $M\in\mathrm{PSL}\left(2,\,\mathbb{Z}\right)$ is
written in $S$ and $T$, the image $\omega\left(M\right)$ is nontrivial
if and only if the sum of exponents is odd. We choose the representatives
the identity $I$ and $S$ in the corresponding transversals.

For an integer $\chi_{i}$ defined in Theorem 5.8, we have $\left(\chi_{i}\right)_{*}=\chi_{i}$.
Therefore, $\omega\left(V_{\chi_{i}}\right)\ne1$. Using the represenative
$S$, for each generator $V_{\chi_{i}}$ of $\Gamma^{0}\left(p\right)$,
in the kernel we have two generators
\begin{align*}
SV_{\chi_{i}} & ,\,V_{\chi_{i}}S^{-1},
\end{align*}
where we see that the relation $V_{\chi_{i}}^{2}=1$ is trivialized
in the kernel, deleting one generator $V_{\chi_{i}}S^{-1}$.

On the other hand, for an integer $\lambda_{i}$ defined in Theorem
5.8, we have $\left(\lambda{}_{i}\right)_{*}=\lambda_{i}-1$ from
the condition $\left(2\lambda_{i}-1\right)^{2}\equiv-3$ (mod $p$).
Therefore, $\omega\left(V_{\lambda_{i}}\right)=1$. Using the representative
$S$, for each generator $V_{\lambda_{i}}$ of $\Gamma^{0}\left(p\right)$,
in the kernel we have two generators
\begin{align*}
V_{\lambda_{i}},\,SV_{\lambda_{i}}S^{-1},
\end{align*}
where we see that the relation $V_{\lambda_{i}}^{3}=1$ gives another
relation $\left(SV_{\lambda_{i}}S^{-1}\right)^{3}=1$ in the kernel.
Since a free generator $V_{i}$ adds two free generators in the kernel
regardless of the image under $\omega$, we establish the result desired
by computing the rank of each kernel. $\qedhere$

\noindent \end{proof}
\begin{lemma}

The representation $\beta_{3,\,p}$ is faithful if and only if the
restriction $\mu_{p}|_{\mathcal{K}_{p}}$ is injective. Moreover,
we have $\mu_{p}\left(\mathcal{K}_{p}\right)\subset\mathcal{Q}_{p}^{1}$.

\noindent \end{lemma}
\begin{proof}

The image $\mu_{p}\left(\mathcal{K}_{p}\right)$ is an index 2 subgroup
of $\mu_{p}\left(\Gamma^{0}\left(p\right)\right)$, since a nontrivial
coset representative is
\begin{align*}
\mu_{p}\left(S\right) & =\left[\begin{array}{cc}
1 & 0\\
0 & -t
\end{array}\right].
\end{align*}

Therefore, the injectivity of $\mu_{p}|_{\mathcal{K}_{p}}$ is equivalent
to the injectivity of $\mu_{p}|_{\Gamma^{0}\left(p\right)}$. Corollary
5.7 concludes the proof on the injectivity.

By the construction of the map $\omega$, any element $M\in\mu_{p}\left(\mathcal{K}_{p}\right)$
has determinant $\left(-t\right)^{2k}$ for some integer $k$ up to
square unit. Thus, $M$ has a representative $B$ such that $\det\left(B\right)=1$.
$\qedhere$

\noindent \end{proof}

Based on the structure theorem of $\mathcal{Q}_{p}$ (Theorem 3.11),
we will consider the structure of the image $\mu_{p}\left(\mathcal{K}_{p}\right)$.
The key algorithm we use is Stallings' folding process as follows.

\noindent \begin{theorem}

Let $F$ be a free group. For a finite set $S$ of elements in $F$,
there is an algorithm to find a minimal generating set of the group
generated by $S$. In particular, we can effectively compute the rank
of $\left\langle S\right\rangle $.

\noindent \end{theorem}
\begin{proof}

\noindent See \citep{MR0695906,MR2130178,MR1882114}. $\qedhere$

\noindent \end{proof}

We need a finite computational setting to apply Theorem 5.11 effectively.
Recall that in Lemma 4.6, we defined the upper (resp. lower) affine
subgroup, as the semidirect product of the upper (resp. lower) multiplicative
subgroup acting on the upper (resp. lower) additive subgroup. Suppose
a prime $p$ satisfies $p\equiv1$ (mod 6). By Lemma 4.6, for $a\in\mathbb{F}_{p}\backslash\left\{ \pm\left(\tau+2\tau^{-1}\right)\right\} $
and a polynomial $f\left(x\right)\in\mathbb{F}_{p}\left[x\right]$,
the action of $\overline{m_{p,\,\tau,\,u}}\left[a\right]$ on $\overline{a_{p,\,\tau,\,u}}\left[f\right]$
is represented by a scalar multiplication $\overline{a_{p,\,\tau,\,u}}\left[h_{\tau}\left(a\right)f\right]$
on the polynomial $f$, where $h_{\tau}\left(a\right)=\frac{a-\left(\tau+2\tau^{-1}\right)}{a+\left(\tau+2\tau^{-1}\right)}$.
Similarly, the action of $\overline{m_{p,\,\tau,\,l}}\left[a\right]$
on $\overline{a_{p,\,\tau,\,l}}\left[f\right]$ is represented by
a scalar multiplication $\overline{a_{p,\,\tau,\,l}}\left[h_{\tau}\left(a\right)^{-1}f\right]$
on the polynomial $f$. In any cases, the degree of $f$ is invariant.

\noindent \begin{definition}

Suppose a prime $p$ satisfies $p\equiv1$ (mod 6). Choose an element
$\tau\in\mathbb{F}_{p}$ such that $\tau^{4}+\tau^{2}+1=0$. For any
integer $k>0$, we define the \emph{upper} \emph{affine subgroup}
of level $k$ for $\tau$ to be
\begin{align*}
\mathrm{Aff}_{p,\,\tau,\,u}\left[k\right] & :=\left\langle \overline{a_{p,\,\tau,\,u}}\left[f\right]\::\:\deg\left(f\right)\le k-1\right\rangle \rtimes M_{p,\,\tau,\,u},
\end{align*}
where $M_{p,\,\tau,\,u}$ is the upper multiplicative subgroup for
$\tau$ (cf. Lemma 4.5) and define the \emph{lower} \emph{affine subgroup}
of level $k$ for $\tau$ to be
\begin{align*}
\mathrm{Aff}_{p,\,\tau,\,l}\left[k\right] & :=\left\langle \overline{a_{p,\,\tau,\,l}}\left[f\right]\::\:\deg\left(f\right)\le k-1\right\rangle \rtimes M_{p,\,\tau,\,l},
\end{align*}
where $M_{p,\,\tau,\,l}$ is the lower multiplicative subgroup for
$\tau$ (cf. Lemma 4.5).

Formally, define $\mathrm{Aff}_{p,\,\tau,\,u}\left[0\right]$ and
$\mathrm{Aff}_{p,\,\tau,\,l}\left[0\right]$ to be the trivial group.
Denote by $E_{p^{k}}$ the elementary abelian group of order $p^{k}$.
Then, the groups $\mathrm{Aff}_{p,\,\tau,\,u}\left[k\right]$ and
$\mathrm{Aff}_{p,\,\tau,\,l}\left[k\right]$ are isomorphic to $E_{p^{k}}\rtimes C_{p-1}$.
Define the \emph{base group} of level $k$ for $\tau$ to be
\begin{align*}
G_{p,\,\tau}\left[k\right] & :=\left\langle \mathrm{Aff}_{p,\,\tau,\,u}\left[k\right],\,\mathrm{Aff}_{p,\,\tau,\,l}\left[k\right]\right\rangle .
\end{align*}

By Lemma 4.8, the group $G_{p,\,\tau}\left[k\right]$ is the free
product of the two affine subgroups of level $k$. For any $a\in\mathbb{F}_{p}\backslash\left\{ 3\tau+\tau^{-1}\right\} $,
define the \emph{type group} of level $k$ for $\tau$ to be
\begin{align*}
\mathcal{G}_{p,\,\tau}\left[k\right] & :=\left\langle G_{p,\,\tau}\left[k\right],\,\overline{s_{p,\,\tau}}\left[a\right]\right\rangle ,
\end{align*}
which is invariant with respect to the choice of $a$. By Lemma 4.17,
$\mathcal{G}_{p,\,\tau}\left[k\right]$ is the HNN extension with
$G_{p,\,\tau}\left[k\right]$ as the base group and $\overline{s_{p,\,\tau}}\left[a\right]$
as the stable letter.

Finally, define the \emph{quaternionic group} of level $k$, denoted
by $\mathcal{Q}_{p}\left[k\right]$, to be generated by two sets of
groups:
\begin{enumerate}
\item The type groups $\mathcal{G}_{p,\,r}$ where $r$ satisfies the condition
$r^{4}+r^{2}+1\ne0$.
\item The type groups of level $k$, $\mathcal{G}_{p,\,\tau}\left[k\right]$,
where $\tau$ satisfies the condition $\tau^{4}+\tau^{2}+1=0$.
\end{enumerate}
\noindent \end{definition}

See \citep{MR1954121,MR1239551} for an overview of the fundamentals
of Bass\textendash Serre theory.

\noindent \begin{definition}

For a prime $p$ being 1 modulo 6, $\tau\in\mathbb{F}_{p}$ such that
$\tau^{4}+\tau^{2}+1=0$, and a nonnegative integer $k$, define a
graph of groups $\Gamma_{p,\,\tau}\left[k\right]$ over a graph $X_{p,\,\tau}\left[k\right]$
having two vertices $v_{1},\,v_{2}$ and two edges $e_{1},\,e_{2}$,
such that
\begin{enumerate}
\item the vertex $v_{1}$ has $\mathrm{Aff}_{p,\,\tau,\,u}\left[k\right]$
as the vertex group $G_{v_{1}}$, and the vertex $v_{2}$ has $\mathrm{Aff}_{p,\,\tau,\,l}\left[k\right]$
as the vertex group $G_{v_{2}}$,
\item each edge connects the two vertices $v_{1},\,v_{2}$,
\item the edge $e_{1}$ has the trivial edge group, and the edge $e_{2}$
has the cyclic group $C_{p-1}$ as the edge group, and
\item the assigning maps $\phi_{1}:C_{p-1}\to G_{v_{1}}$ and $\phi_{2}:C_{p-1}\to G_{v_{2}}$
have the image $M_{p,\,\tau,\,u}$ and $M_{p,\,\tau,\,l}$, respectively,
such that for each element $a\in\mathbb{F}_{p}\backslash\left\{ \pm\left(\tau+2\tau^{-1}\right)\right\} $,
\begin{align*}
\phi_{2}\circ\phi_{1}^{-1}\left(\overline{m_{p,\,\tau,\,u}}\left[a\right]\right) & =\overline{m_{p,\,\tau,\,l}}\left[a\right].
\end{align*}
\end{enumerate}
\noindent \end{definition}

By construction, the group $\mathcal{G}_{p,\,\tau}\left[k\right]$
is the fundamental group of the graph of groups $\Gamma_{p,\,\tau}\left[k\right]$,
by choosing the spanning tree containing $e_{1}$. Therefore, the
group $\mathcal{G}_{p,\,\tau}\left[k\right]$ is virtually free, because
the graph $X_{p,\,\tau}\left[k\right]$ is finite and each vertex
group is finite \citep[Proposition 11, p. 120]{MR1954121}.

However, the general method involving HNN extensions to construct
a homomorphism from $\mathcal{G}_{p,\,\tau}\left[k\right]$ whose
kernel is free employs a symmetric group in its range, which causes
inefficiency in real computation. Therefore, in this case, we suggest
another method based on the invariance of the fundamental group with
respect to the choice of the spanning tree.

\noindent \begin{definition}

Suppose a prime $p$ satisfies $p\equiv1$ (mod 6) and an element
$\tau\in\mathbb{F}_{p}$ satisfies $\tau^{4}+\tau^{2}+1=0$. Choose
a nonnegative integer $k$. For any $a\in\mathbb{F}_{p}\backslash\left\{ 3\tau+\tau^{-1}\right\} $,
define a group
\begin{align*}
H_{p,\,\tau}\left[k\right] & :=\left\langle \overline{s_{p,\,\tau}}\left[a\right]\mathrm{Aff}_{p,\,\tau,\,u}\left[k\right]\overline{s_{p,\,\tau}}\left[a\right]^{-1},\,\mathrm{Aff}_{p,\,\tau,\,l}\left[k\right]\right\rangle .
\end{align*}

By the relation (57), the group $H_{p,\,\tau}\left[k\right]$ is the
amalgamated free product of the two subgroups
\begin{align*}
\overline{s_{p,\,\tau}}\left[a\right]\mathrm{Aff}_{p,\,\tau,\,u}\left[k\right]\overline{s_{p,\,\tau}}\left[a\right]^{-1},\,\mathrm{Aff}_{p,\,\tau,\,l}\left[k\right].
\end{align*}

Define a semidirect product $K_{p}\left[k\right]$ to be
\begin{align*}
\left(E_{p^{k}}\times E_{p^{k}}\right)\rtimes\mathbb{F}_{p}^{\times},
\end{align*}
where the action of $b\in\mathbb{F}_{p}^{\times}$ on the left Cartesian
factor is given by $f\mapsto b^{-1}f$ and the action on the right
factor is given by $f\mapsto bf$, where we represent an element $f\in E_{p^{k}}$
as a polynomial in $\mathbb{F}_{p}\left[x\right]$. For an element
$a\in\mathbb{F}_{p}^{\times}$ and a pair of polynomials $f_{1},\,f_{2}\in\mathbb{F}_{p}\left[x\right]$
such that $\deg\left(f_{1}\right)\le k-1$ and $\deg\left(f_{2}\right)\le k-1$,
we write each element in $K_{p}\left[k\right]$ as
\begin{align*}
\left(f_{1},\,f_{2},\,b\right),
\end{align*}
where $f_{1}$ is in the left factor and $f_{2}$ in the right factor.

Define a map
\begin{align*}
\Lambda_{p,\,\tau}\left[k\right] & :H_{p,\,\tau}\left[k\right]\to K_{p}\left[k\right],
\end{align*}
as for each polynomials $f\in\mathbb{F}_{p}\left[x\right]$ such that
$\deg\left(f\right)\le k-1$ and $c\in\mathbb{F}_{p}\backslash\left\{ \pm\left(\tau+2\tau^{-1}\right)\right\} $,
\begin{align*}
 & \overline{s_{p,\,\tau}}\left[a\right]\overline{a_{p,\,\tau,\,u}}\left[f\right]\overline{s_{p,\,\tau}}\left[a\right]^{-1}\mapsto\left(f,\,0,\,1\right),\\
 & \overline{a_{p,\,\tau,\,l}}\left[f\right]\mapsto\left(0,\,f,\,1\right),\\
 & \overline{m_{p,\,\tau,\,l}}\left[c\right]\mapsto\left(0,\,0,\,h_{\tau}\left(c\right)\right).
\end{align*}

By choosing the spanning tree containing $e_{2}$ of $\Gamma_{p,\,\tau}\left[k\right]$,
we have
\begin{align*}
\mathcal{G}_{p,\,\tau}\left[k\right] & =\left\langle \overline{s_{p,\,\tau}}\left[a\right]\right\rangle *H_{p,\,\tau}\left[k\right].
\end{align*}

Therefore, we extend the map $\Lambda_{p,\,\tau}\left[k\right]$ to
$\mathcal{G}_{p,\,\tau}\left[k\right]$ by defining
\begin{align*}
\widehat{\Lambda_{p,\,\tau}}\left[k\right] & :\mathcal{G}_{p,\,\tau}\left[k\right]\to K_{p}\left[k\right],
\end{align*}
as $\overline{s_{p,\,\tau}}\left[a\right]\mapsto1$ and for $h\in H_{p,\,\tau}\left[k\right]$,
$h\mapsto\Lambda_{p,\,\tau}\left[k\right]\left(h\right)$.

Finally, for the four distinct solutions $\tau_{i}$, $1\le i\le4$,
to the equation $x^{4}+x^{2}+1=0$ over $\mathbb{F}_{p}$, extend
further the map $\widehat{\Lambda_{p,\,\tau}}\left[k\right]$ to the
group $\mathcal{Q}_{p}\left[k\right]$ by defining
\begin{align*}
\widehat{\Lambda_{p}}\left[k\right] & :\mathcal{Q}_{p}\left[k\right]\to K_{p}\left[k\right],
\end{align*}
as $g\mapsto1$ for an element $g\in\mathcal{G}_{p,\,r}$, where $r$
satisfies $r^{4}+r^{2}+1\ne0$, and $h\mapsto\widehat{\Lambda_{p,\,\tau}}\left[k\right]\left(h\right)$
for an element $h\in\mathcal{G}_{p,\,\tau}\left[k\right]$.

\noindent \end{definition}

We directly observe that $K_{p}\left[k\right]$ is a finite group
of order $p^{2k}\left(p-1\right)$.

\noindent \begin{lemma}

Suppose $p$ satisfies $p\equiv1$ (mod 6) and an element $\tau\in\mathbb{F}_{p}$
satisfies $\tau^{4}+\tau^{2}+1=0$. Then, for each nonnegative integer
$k$, the subgroup $\ker\left(\Lambda_{p,\,\tau}\left[k\right]\right)$
is a free subgroup of $H_{p,\,\tau}\left[k\right]$. Moreover, $\ker\left(\widehat{\Lambda_{p}}\left[k\right]\right)$
is a free subgroup of $\mathcal{Q}_{p}\left[k\right]$, and there
is an effective way to compute a free generating set.

\noindent \end{lemma}
\begin{proof}

\noindent Temporarily denote by $K$ the subgroup $\ker\left(\Lambda_{p,\,\tau}\left[k\right]\right)$.
Then, by construction, $K$ meets any conjugate of the lower multiplicative
subgroup only trivially. By the Kurosh subgroup theorem for amalgamated
free products \citep[Theorem 14, p. 56]{MR1954121}, $K$ is a free
product of a free subgroup of $H_{p,\,\tau}\left[k\right]$ and conjugates
of subgroups of the factors
\begin{align*}
\overline{s_{p,\,\tau}}\left[a\right]\mathrm{Aff}_{p,\,\tau,\,u}\left[k\right]\overline{s_{p,\,\tau}}\left[a\right]^{-1},\,\mathrm{Aff}_{p,\,\tau,\,l}\left[k\right].
\end{align*}

Suppose $L$ is a conjugate by $M_{0}\in H_{p,\,\tau}\left[k\right]$
of a subgroup of the left factor
\begin{align*}
\overline{s_{p,\,\tau}}\left[a\right]\mathrm{Aff}_{p,\,\tau,\,u}\left[k\right]\overline{s_{p,\,\tau}}\left[a\right]^{-1},
\end{align*}
and suppose $M$ is an element in the intersection of $K$ and $L$.
Then, the left factor includes $M_{0}^{-1}MM_{0}$ and we have
\begin{align*}
\Lambda_{p,\,\tau}\left[k\right]\left(M_{0}^{-1}MM_{0}\right) & =\left(0,\,0,\,1\right),
\end{align*}
which implies $M$ is the identity element. Thus, $K$ meets any conjugate
of the left factor only trivially. In the same way, $K$ meets any
conjugate of the right factor only trivially. Consequently, we have
established that $K$ is free.

Based on the freeness of $K$ and the finiteness of $K_{p}\left[k\right]$,
the group $\ker\left(\widehat{\Lambda_{p,\,\tau}}\left[k\right]\right)$
is free. By Theorem 3.11, and by applying the usual Kurosh subgroup
theorem for free products, we conclude that $\ker\left(\widehat{\Lambda_{p}}\left[k\right]\right)$
is also free. Since the group $K_{p}\left[k\right]$ is finite, we
can find effectively the free generating set by using the Reidemeister\textendash Schreier
process. $\qedhere$

\noindent \end{proof}

Suppose a prime $p$ satisfies $p\equiv1$ (mod 6). Lemma 5.9 ensures
that the group $\mathcal{K}_{p}$ is isomorphic to $F_{\frac{p-4}{3}}*C_{3}*C_{3}*C_{3}*C_{3}$.
Fix a generator $g$ of the order 3 cyclic group $C_{3}$. Define
a map $\epsilon_{p}:\mathcal{K}_{p}\to C_{3}$ as $h\mapsto1$ when
$h$ is included in the free group factor, and an element corresponding
to $g$ in the cyclic factors to the generator $g$ of the range $C_{3}$.

\noindent \begin{lemma}

Suppose a prime $p$ satisfies $p\equiv1$ (mod 6). Then, $\ker\left(\epsilon_{p}\right)$
is free of rank $p+2$.

\noindent \end{lemma}
\begin{proof}

\noindent The freeness is direct from the usual Kurosh subgroup theorem.
The rank is computed by Lemma 5.9. $\qedhere$

\noindent \end{proof}

We present the proof of Theorem 1.1, except for the case $p=3$ for
simplicity. For the case $p=3$, we will directly prove the faithfulness
later in Theorem 5.18.

\noindent \begin{prooftheorem11}

\noindent Fix any prime $p$ such that $p\ne3$. For clarity, we divide
the two cases according to the condition of prime $p$.
\begin{enumerate}
\item When $p\equiv2,\,5$ (mod 6).
\item When $p\equiv1$ (mod 6.
\end{enumerate}
By Lemma 5.10, it suffices to decide whether the restriction $\mu_{p}|_{\mathcal{K}_{p}}$
is injective. We present the algorithm, with input $p$ and possible
outputs ``true'' or ``false'', as follows.

\noindent \begin{cases1}
\begin{description}
\item [{$ $}]~
\item [{(a)}] Find a minimal generating set $S,\,V_{a_{1}},\,\cdots,\,V_{a_{n}}$
of the congruence subgroup $\Gamma^{0}\left(p\right)$ (see \citep{MR3069525,MR1137534,doan2023optimal}).
The rank is given by Theorem 5.8.
\item [{(b)}] Based on (a), find a minimal generating set $\kappa_{1},\,\kappa_{2},\cdots,\,\kappa_{m}$
of $\mathcal{K}_{p}$. By Lemma 5.9, it is a free group isomorphic
to $F_{2}$ (when $p=2$) or $F_{\frac{p+4}{3}}$ (when $p\equiv5$
(mod 6)). Since finitely generated free groups are Hopfian, it suffices
to decide whether the image $\mu_{p}\left(\mathcal{K}_{p}\right)$
is isomorphic to $\mathcal{K}_{p}$.
\item [{(c)}] Compute the images $\mu_{p}\left(\kappa_{1}\right),\,\mu_{p}\left(\kappa_{2}\right),\cdots,\,\mu_{p}\left(\kappa_{m}\right)$.
By Lemma 5.10, these are included in the quaternionic group $\mathcal{Q}_{p}$,
which is freely generated by $p$ elementary generators.
\item [{(d)}] For each image $g$ of the generator of $\mathcal{K}_{p}$,
check the balancedness and compute the type $\tau\left(g\right)$.
If $g$ is not balanced, multiply some power of $\overline{g_{p}}\left[0\right]$
on the right to replace it with the balanced companion (Lemma 3.3).
If $g$ is upper-balanced (resp. lower-balanced), we have a new element
$g_{1}=g\overline{g_{p}}\left[\tau\left(g\right)\right]^{-1}$ (resp.
$g_{1}=g\overline{g_{p}}\left[\tau\left(g\right)\right]$) such that
$\mathrm{rd}\left(g\right)>\mathrm{rd}\left(g_{1}\right)$. Repeat
this process, expressing the images $\mu_{p}\left(\kappa_{1}\right),\,\mu_{p}\left(\kappa_{2}\right),\cdots,\,\mu_{p}\left(\kappa_{m}\right)$
as words $w_{1},\,w_{2},\cdots,\,w_{m}$ in the set of elementary
generators.
\item [{(d)}] Compute the rank of the free subgroup generated by $w_{1},\,w_{2},\cdots,\,w_{m}$
by Stallings' folding process (Theorem 5.11). If the rank is the same
as that of $\mathcal{K}_{p}$, terminate the algorithm with output
``true''. Otherwise, terminate the algorithm with output ``false''.
\end{description}
\noindent \end{cases1}
\begin{cases2}
\begin{description}
\item [{$ $}]~
\item [{(a)}] Find a minimal generating set $S,\,V_{a_{1}},\,\cdots,\,V_{a_{n}}$
of the congruence subgroup $\Gamma^{0}\left(p\right)$ (see \citep{MR3069525,MR1137534,doan2023optimal}).
The rank is given by Theorem 5.8. By Lemma 5.9, the group $\mathcal{K}_{p}$
is a free product isomorphic to $F_{\frac{p-4}{3}}*C_{3}*C_{3}*C_{3}*C_{3}$.
\item [{(b)}] Compute the images of the cyclic factors $C_{3}$ of $\mathcal{K}_{p}$
under $\mu_{p}$. If at least one of them is trivial, terminate here
the algorithm with output ``false''. Suppose these are nontrivial.
\item [{(c)}] Based on (a), find a minimal generating set $\kappa_{1},\,\kappa_{2},\cdots,\,\kappa_{p+2}$
of $\ker\left(\epsilon_{p}\right)$, which is free of rank $p+2$
by Lemma 5.16.
\item [{(d)}] Compute the images $\mu_{p}\left(\kappa_{1}\right),\,\mu_{p}\left(\kappa_{2}\right),\cdots,\,\mu_{p}\left(\kappa_{p+2}\right)$.
By Lemma 5.10, these are included in the quaternionic group $\mathcal{Q}_{p}$.
From Lemma 4.18, for each image $g$ of the generator of $\mathcal{K}_{p}$,
replace $g$ with its balanced companion (Lemma 3.3) and find an elementary,
affine or stable generator (or its inverse) $h$ of a type group $\mathcal{G}_{p,\,\tau\left(g\right)}$
such that $\mathrm{rd}\left(g\right)>\mathrm{rd}\left(gh\right)$
by computing the type $\tau\left(g\right)$ and the second-order type
$\tau2_{\tau\left(g\right)}\left(g\right)$ if needed. We have a new
element $g_{1}=gh$. Repeat this process, expressing the images $\mu_{p}\left(\kappa_{1}\right),\,\mu_{p}\left(\kappa_{2}\right),\cdots,\,\mu_{p}\left(\kappa_{p+2}\right)$
as words $w_{1},\,w_{2},\cdots,\,w_{p+2}$ in the set of the elementary,
affine and stable generators of $\mathcal{Q}_{p}$.
\item [{(e)}] Find the integer $k$ such that $k-1$ is the maximal degree
of $f\in\mathbb{F}_{p}\left[x\right]$ among the additive generators
$\overline{a_{p,\,\tau,\,u}}\left[f\right]$, $\overline{a_{p,\,\tau,\,l}}\left[f\right]$
that appear in the words $w_{1},\,w_{2},\cdots,\,w_{p+2}$. If there
is no such generator in the words, choose $k$ as $0$. Define the
subgroup $W$ of $\mathcal{Q}_{p}$ to be generated by $w_{1},\,w_{2},\cdots,\,w_{p+2}$.
Since finitely generated free groups are Hopfian, it suffices to decide
whether $W$ is free of rank $p+2$.
\item [{(f)}] Compute a group presentation of the image group $\widehat{\Lambda_{p}}\left[k\right]\left(W\right)$
by employing the Todd\textendash Coxeter algorithm in conjunction
with the Reidemeister\textendash Schreier process. This terminates
always in finite time since the range of $\widehat{\Lambda_{p}}\left[k\right]$
is a finite group $K_{p}\left[k\right]$. Denote by $K_{W}$ the group
$\ker\left(\widehat{\Lambda_{p}}\left[k\right]|_{W}\right)$, and
denote by $N$ the order of the group $\widehat{\Lambda_{p}}\left[k\right]\left(W\right)$.
\item [{(g)}] Compute a generating set of $K_{W}$ by the Reidemeister\textendash Schreier
process. The group $K_{W}$ is also a free group since a free group
$\ker\left(\widehat{\Lambda_{p}}\left[k\right]\right)$ (Lemma 5.15)
contains $K_{W}$. If $W$ is free, the Schreier index formula indicates
that the rank of the free subgroup $K_{W}$ is $1+N\left(p+1\right)$.
\item [{(h)}] By Stallings' folding process (Theorem 5.11), compute the
rank of the subgroup $K_{W}$ of $\ker\left(\widehat{\Lambda_{p}}\left[k\right]\right)$.
If the rank is exactly $1+N\left(p+1\right)$, terminate the algorithm
with output ``true''. Otherwise, terminate the algorithm with output
``false''.
\end{description}
We conclude the proof by showing that $\mu_{p}|_{\mathcal{K}_{p}}$
is injective if and only if the algorithm terminates with output ``true''.
In cases (1), the rank of the free subgroup of $\mathcal{Q}_{p}$
generated by $w_{1},\,w_{2},\cdots,\,w_{m}$ is at most the rank of
$\mathcal{K}_{p}$, thus the condition that the rank is full implies
that the image $\mu_{p}\left(\mathcal{K}_{p}\right)$ is isomorphic
to $\mathcal{K}_{p}$. In cases (2), we consider the following commutative
diagram:
\begin{align*}
\xymatrix{1\ar[r] & F_{1+N\left(p+1\right)}\ar[r]\ar@{->>}[d] & F_{p+2}\ar@{->>}[d]\ar[r] & \widehat{\Lambda_{p}}\left[k\right]\left(W\right)\ar[d]^{\Vert}\ar[r] & 1\\
1\ar[r] & K_{W}\ar[r] & W\ar[r] & \widehat{\Lambda_{p}}\left[k\right]\left(W\right)\ar[r] & 1
}
\end{align*}
where the bottom sequence is induced by the map $\widehat{\Lambda_{p}}\left[k\right]$,
the surjection $F_{p+2}\to W$ is given by the universal property
of free groups, and the surjection $F_{1+N\left(p+1\right)}\to K_{W}$
is the restriction given by the Reidemeister\textendash Schreier process.
Since $K_{W}$ is free, the surjection on the left is an isomorphism
if and only if the rank of $K_{W}$ is $1+N\left(p+1\right)$. The
result desired follows from the five lemma. $\qed$

\noindent \end{cases2}
\end{prooftheorem11}
\begin{remark}

Groves\textendash Wilton \citep[Corollary 4.3]{MR2516172} proved
that there exists an algorithm that, given as input a presentation
for a group $G$ and a solution to the word problem in $G$, determines
whether or not $G$ is free. In our case, since the images under the
map $\mu_{p}$ are linear, and the word problem is solvable for finitely
generated linear groups, this algorithm provides an alternative way
to prove Theorem 1.1 for the cases (2).

One can find a finite presentation of $W$ from those of $\widehat{\Lambda_{p}}\left[k\right]\left(W\right)$
and $K_{W}$, and determine whether $W$ is free using the Groves\textendash Wilton
algorithm. If $W$ is finitely presented and free, we can effectively
find its rank by abelianizing it. However, this method of proof also
requires a presentation of the free subgroup $K_{W}$, and we still
need Stallings' folding process.

It is also noteworthy that for each prime $p$, the quaternionic group
$\mathcal{Q}_{p}$ has solvable membership problem, utilizing a similiar
method to that employed in the proof of Theorem 1.1. For the cases
where $p\equiv1$ (mod 6), we have a sequence of finitely generated
subgroups $\mathcal{Q}_{p}\left[k\right]$, $k=0,\,1,\,\cdots$ of
$\mathcal{Q}_{p}$. A finite set of elements $S$ of $\mathcal{Q}_{p}$
is contained in some $\mathcal{Q}_{p}\left[k\right]$. Because $\mathcal{Q}_{p}\left[k\right]$
is also a fundamental group of a finite graph of groups, where all
vertex groups and edge groups are finite, $\mathcal{Q}_{p}\left[k\right]$
has solvable membership problem \citep[Corollary 5.15]{MR2130178}.

\noindent \end{remark}

As mentioned in Section 1, we demonstrate specific examples where
$\beta_{3,\,p}$ is faithful. We begin by listing the generating set
and relations of $\Gamma^{0}\left(p\right)$ for small $p$, obtained
by Rademacher \citep[p. 147]{MR3069525}.
\begin{itemize}
\item $\Gamma^{0}\left(2\right)$: generated by $S,\,V_{1}$; having a relation
$V_{1}^{2}=1$.
\item $\Gamma^{0}\left(3\right)$: generated by $S,\,V_{2}$; having a relation
$V_{2}^{3}=1$.
\item $\Gamma^{0}\left(5\right)$: generated by $S,\,V_{2},\,V_{3}$; having
relations $V_{2}^{2}=1$, $V_{3}^{2}=1$.
\item $\Gamma^{0}\left(7\right)$: generated by $S,\,V_{3},\,V_{5}$; having
relation $V_{3}^{3}=1$, $V_{5}^{3}=1$.
\item $\Gamma^{0}\left(11\right)$: generated by $S,\,V_{4},\,V_{6}$; no
relation.
\item $\Gamma^{0}\left(13\right)$: generated by $S,\,V_{4},\,V_{5},\,V_{8},\,V_{10}$;
having relation $V_{5}^{2}=1$, $V_{8}^{2}=1$, $V_{4}^{3}=1$, $V_{10}^{3}=1$.
\end{itemize}
\noindent \begin{theorem}

The representation $\beta_{3,\,p}$ is faithful when $p\le13$.

\noindent \end{theorem}
\begin{proof}

Recall that $S=\overline{\rho}\left(\beta_{3}\left(\sigma_{2}\right)\right)$,
$T=\overline{\rho}\left(\beta_{3}\left(\sigma_{1}\sigma_{2}\sigma_{1}\right)\right)$
and $V_{k}=T^{-1}S^{-k_{*}}TS^{k}T$. From (77) and (78), $\mu_{p}$
maps
\begin{align*}
S & \mapsto\left[\begin{array}{cc}
1 & 0\\
0 & -t
\end{array}\right],\;T\mapsto\left[\begin{array}{cc}
\frac{t}{1+t} & \frac{t}{1+t}\\
\frac{1+t+t^{2}}{1+t} & -\frac{t}{1+t}
\end{array}\right].
\end{align*}

For the definitions of the generators involved, refer to Definition
2.15 for the elementary generators; Definition 3.5 and Definition
3.6 for the additive generators; Definition 3.7 for the multiplicative
generators; Definition 3.9 and Lemma 3.10 for the stable generators.

Suppose $p=2$. The group $\mathcal{K}_{2}$ is freely generated by
$S^{2}$ and $SV_{1}$, which are mapped to $\overline{g_{2}}\left[0\right]^{-1}$
and $\overline{g_{2}}\left[0\right]^{-1}\overline{g_{2}}\left[1\right]$,
respectively under $\mu_{2}$. These elements freely generate the
entire free group $\mathcal{Q}_{2}$. Therefore, the free rank of
$\mu_{2}\left(\mathcal{K}_{2}\right)$ is 2, and $\beta_{3,\,2}$
is faithful.

Suppose $p=3$. The group $\mathcal{K}_{3}$ is the free product $\left\langle S^{2}\right\rangle *\left\langle V_{2}\right\rangle *\left\langle SV_{2}S^{-1}\right\rangle $.
Under $\mu_{3}$, $S^{2}$ is mapped to $\overline{g_{3}}\left[0\right]^{-1}$;
$V_{2}$ to $\overline{a_{3,\,2,\,u}}\left[-1\right]$; $SV_{2}S^{-1}$
to $\overline{g_{3}}\left[0\right]^{-1}\overline{a_{3,\,1,\,l}}\left[1\right]\overline{g_{3}}\left[0\right]$.
By Theorem 3.11, the images generate a free product $\left\langle \overline{g_{3}}\left[0\right]\right\rangle *\left\langle \overline{a_{3,\,2,\,u}}\left[1\right]\right\rangle *\left\langle \overline{a_{3,\,1,\,l}}\left[1\right]\right\rangle $.
Since a finite free product of Hopfian groups is Hopfian, $\mu_{3}|_{\mathcal{K}_{3}}$
is injective, which implies $\beta_{3,\,3}$ is faithful.

Suppose $p=5$. The group $\mathcal{K}_{5}$ is freely generated by
$S^{2}$, $SV_{2}$ and $SV_{3}$, which are mapped to
\begin{align*}
\overline{g_{5}}\left[0\right]^{-1},\,\overline{g_{5}}\left[0\right]^{-1}\overline{g_{5}}\left[1\right]\overline{g_{5}}\left[3\right]^{-1}\overline{g_{5}}\left[4\right],\,\overline{g_{5}}\left[0\right]^{-1}\overline{g_{5}}\left[1\right]\overline{g_{5}}\left[3\right]^{-1}\overline{g_{5}}\left[2\right]\overline{g_{5}}\left[3\right]^{-1}\overline{g_{5}}\left[4\right],
\end{align*}
respectively under $\mu_{5}$. Therefore, the rank of $\mu_{5}\left(\mathcal{K}_{5}\right)$
is 3, equal to $\frac{5+4}{3}$, which implies $\beta_{3,\,5}$ is
faithful.

Suppose $p=7$. The group $\mathcal{K}_{7}$ is the free product $\left\langle S^{2}\right\rangle *\left\langle V_{3}\right\rangle *\left\langle SV_{3}S^{-1}\right\rangle *\left\langle V_{5}\right\rangle *\left\langle SV_{5}S^{-1}\right\rangle $.
Under $\mu_{7}$, the five elements in the brakets are respectively
mapped to
\begin{align*}
 & \overline{g_{7}}\left[0\right]^{-1},\,\overline{g_{7}}\left[6\right]^{-1}\overline{m_{7,\,4,\,l}}\left[3\right]\overline{g_{7}}\left[6\right],\\
 & \overline{g_{7}}\left[0\right]^{-1}\overline{g_{7}}\left[1\right]\overline{m_{7,\,5,\,u}}\left[2\right]\overline{g_{7}}\left[1\right]^{-1}\overline{g_{7}}\left[0\right],\\
 & \overline{g_{7}}\left[6\right]^{-1}\overline{m_{7,\,4,\,l}}\left[4\right]\overline{m_{7,\,2,\,l}}\left[5\right]\overline{m_{7,\,4,\,l}}\left[3\right]\overline{g_{7}}\left[6\right],\\
 & \overline{g_{7}}\left[0\right]^{-1}\overline{g_{7}}\left[1\right]\overline{m_{7,\,5,\,l}}\left[5\right]\overline{m_{7,\,3,\,l}}\left[4\right]\overline{m_{7,\,5,\,l}}\left[2\right]\overline{g_{7}}\left[1\right]^{-1}\overline{g_{7}}\left[0\right].
\end{align*}

Since the only multiplicative generators are involved for the type
groups $\mathcal{G}_{7,\,2},\,\mathcal{G}_{7,\,3},\,\mathcal{G}_{7,\,4},\,\mathcal{G}_{7,\,5}$,
we need to consider only the subgroups generated by the multiplicative
generators isomorphic to $C_{6}*C_{6}$ for each $\mathcal{G}_{7,\,\tau}$,
$\tau=2,\,3,\,4,\,5$. Recall that from Lemma 4.5 or (40), we have
\begin{align*}
 & \overline{m_{p,\,\tau,\,u}}\left[a\right]^{-1}=\overline{m_{p,\,\tau,\,u}}\left[-a\right],\\
 & \overline{m_{p,\,\tau,\,l}}\left[a\right]^{-1}=\overline{m_{p,\,\tau,\,l}}\left[-a\right],
\end{align*}
for any well-defined combinations of $p,\,\tau,\,a$. Therefore, the
four order 3 images are four conjugates of multiplicative generators
of different types. By rewriting the presentation of the given finitely
generated free product, we have $\mu_{7}\left(\mathcal{K}_{7}\right)\cong\mathcal{K}_{7}$,
which implies $\beta_{3,\,7}$ is faithful.

Suppose $p=11$. The group $\mathcal{K}_{11}$ is freely generated
by
\begin{align*}
S^{2},\,SV_{4},\,V_{4}S^{-1},\,V_{6},\,SV_{6}S^{-1},
\end{align*}
which are under $\mu_{11}$ respectively mapped to
\begin{align*}
 & \overline{g_{11}}\left[0\right]^{-1},\\
 & \overline{g_{11}}\left[0\right]^{-1}\overline{g_{11}}\left[1\right]\overline{g_{11}}\left[9\right]^{-1}\overline{g_{11}}\left[4\right]\overline{g_{11}}\left[7\right]^{-1}\overline{g_{11}}\left[9\right]\overline{g_{11}}\left[5\right]^{-1}\overline{g_{11}}\left[8\right]\overline{g_{11}}\left[3\right]^{-1}\overline{g_{11}}\left[8\right]\overline{g_{11}}\left[6\right]^{-1}\overline{g_{11}}\left[10\right],\\
 & \overline{g_{11}}\left[10\right]^{-1}\overline{g_{11}}\left[6\right]\overline{g_{11}}\left[8\right]^{-1}\overline{g_{11}}\left[3\right]\overline{g_{11}}\left[6\right]^{-1}\overline{g_{11}}\left[2\right]\overline{g_{11}}\left[4\right]^{-1}\overline{g_{11}}\left[7\right]\overline{g_{11}}\left[4\right]^{-1}\overline{g_{11}}\left[9\right]\overline{g_{11}}\left[1\right]^{-1}\overline{g_{11}}\left[0\right],
\end{align*}
\begin{align*}
 & \left(\overline{g_{11}}\left[10\right]^{-1}\overline{g_{11}}\left[6\right]\overline{g_{11}}\left[8\right]^{-1}\overline{g_{11}}\left[3\right]\overline{g_{11}}\left[6\right]^{-1}\overline{g_{11}}\left[2\right]\overline{g_{11}}\left[4\right]^{-1}\right.\\
 & \left.\overline{g_{11}}\left[7\right]\overline{g_{11}}\left[2\right]^{-1}\overline{g_{11}}\left[6\right]\overline{g_{11}}\left[3\right]^{-1}\overline{g_{11}}\left[8\right]\overline{g_{11}}\left[6\right]^{-1}\overline{g_{11}}\left[10\right]\right),\\
 & \left(\overline{g_{11}}\left[0\right]^{-1}\overline{g_{11}}\left[1\right]\overline{g_{11}}\left[9\right]^{-1}\overline{g_{11}}\left[4\right]\overline{g_{11}}\left[7\right]^{-1}\overline{g_{11}}\left[9\right]\overline{g_{11}}\left[5\right]^{-1}\overline{g_{11}}\left[8\right]\right.\\
 & \left.\overline{g_{11}}\left[3\right]^{-1}\overline{g_{11}}\left[5\right]\overline{g_{11}}\left[9\right]^{-1}\overline{g_{11}}\left[7\right]\overline{g_{11}}\left[4\right]^{-1}\overline{g_{11}}\left[9\right]\overline{g_{11}}\left[1\right]^{-1}\overline{g_{11}}\left[0\right]\right),
\end{align*}
which generate a rank $5=\frac{11+4}{3}$ free subgroup in $\mathcal{Q}_{5}$
by Stallings' folding process. Therefore, $\beta_{3,\,11}$ is faithful.

Suppose $p=13$. The group $\mathcal{K}_{13}$ is isomorphic to $F_{3}*\left(*_{i=4}C_{3}\right)$,
where the free factor is generated by
\begin{align*}
S^{2},\,SV_{5},\,SV_{8},
\end{align*}
and the finite cyclic factors are generated by
\begin{align*}
V_{4},\,SV_{4}S^{-1},\,V_{10},\,SV_{10}S^{-1}.
\end{align*}

Under $\mu_{13}$, the first three generators are respectively mapped
to
\begin{align*}
 & \overline{g_{13}}\left[0\right]^{-1},\\
 & \overline{g_{13}}\left[0\right]^{-1}\overline{g_{13}}\left[1\right]\overline{g_{13}}\left[11\right]^{-1}\overline{m_{13,\,9,\,l}}\left[9\right]\overline{g_{13}}\left[8\right]\overline{m_{13,\,10,\,u}}\left[3\right]\overline{g_{13}}\left[7\right]^{-1}\overline{g_{13}}\left[12\right],\\
 & \left(\overline{g_{13}}\left[0\right]^{-1}\overline{g_{13}}\left[1\right]\overline{g_{13}}\left[11\right]^{-1}\overline{m_{13,\,9,\,l}}\left[9\right]\overline{g_{13}}\left[8\right]\overline{g_{13}}\left[7\right]^{-1}\overline{g_{13}}\left[2\right]\right.\\
 & \left.\overline{g_{13}}\left[5\right]^{-1}\overline{g_{13}}\left[6\right]\overline{g_{13}}\left[11\right]^{-1}\overline{g_{13}}\left[8\right]\overline{m_{13,\,10,\,u}}\left[3\right]\overline{g_{13}}\left[7\right]^{-1}\overline{g_{13}}\left[12\right]\right),
\end{align*}
and the last four generators are respectively mapped to
\begin{align*}
 & \overline{g_{13}}\left[12\right]^{-1}\overline{g_{13}}\left[7\right]\overline{m_{13,\,10,\,u}}\left[3\right]\overline{g_{13}}\left[7\right]^{-1}\overline{g_{13}}\left[12\right],\\
 & \overline{g_{13}}\left[0\right]^{-1}\overline{g_{13}}\left[1\right]\overline{g_{13}}\left[11\right]^{-1}\overline{m_{13,\,9,\,l}}\left[4\right]\overline{g_{13}}\left[11\right]\overline{g_{13}}\left[1\right]^{-1}\overline{g_{13}}\left[0\right],\\
 & \left(\overline{g_{13}}\left[12\right]^{-1}\overline{g_{13}}\left[7\right]\overline{m_{13,\,10,\,u}}\left[10\right]\overline{g_{13}}\left[8\right]^{-1}\overline{g_{13}}\left[11\right]\overline{g_{13}}\left[6\right]^{-1}\overline{g_{13}}\left[5\right]\overline{m_{13,\,4,\,u}}\left[9\right]\right.\\
 & \left.\overline{g_{13}}\left[5\right]^{-1}\overline{g_{13}}\left[6\right]\overline{g_{13}}\left[11\right]^{-1}\overline{g_{13}}\left[8\right]\overline{m_{13,\,10,\,u}}\left[3\right]\overline{g_{13}}\left[7\right]^{-1}\overline{g_{13}}\left[12\right]\right),\\
 & \left(\overline{g_{13}}\left[0\right]^{-1}\overline{g_{13}}\left[1\right]\overline{g_{13}}\left[11\right]^{-1}\overline{m_{13,\,9,\,l}}\left[9\right]\overline{g_{13}}\left[8\right]\overline{g_{13}}\left[7\right]^{-1}\overline{g_{13}}\left[2\right]\overline{g_{13}}\left[5\right]^{-1}\overline{m_{13,\,3,\,l}}\left[10\right]\right.\\
 & \left.\overline{g_{13}}\left[5\right]\overline{g_{13}}\left[2\right]^{-1}\overline{g_{13}}\left[7\right]\overline{g_{13}}\left[8\right]^{-1}\overline{m_{13,\,9,\,l}}\left[4\right]\overline{g_{13}}\left[11\right]\overline{g_{13}}\left[1\right]^{-1}\overline{g_{13}}\left[0\right]\right).
\end{align*}

As observed in the case $p=7$, the four order 3 images are included
in the four conjugates of multiplicative generators of different types.
Through a proper rewriting process of the presentation of the given
finitely generated free product, we establish that $\mu_{13}\left(\mathcal{K}_{13}\right)\cong\mathcal{K}_{13}$.

Alternatively, employing the algorithm presented in the proof of Theorem
1.1, we compute the images under $\widehat{\Lambda_{13}}\left[0\right]$.
By Definition 5.14, the subgroup of $K_{13}\left[0\right]$ generated
by these images is isomorphic to $C_{3}$. It is because
\begin{align*}
 & h_{9}\left(9\right)=\frac{9-2}{9+2}=3,\\
 & h_{10}\left(3\right)=\frac{3-5}{3+5}=3,\\
 & h_{4}\left(9\right)=\frac{9+2}{9-2}=9=3^{2},\\
 & h_{3}\left(10\right)=\frac{10-8}{10+8}=3,
\end{align*}
where the images are in the unit group $\mathbb{F}_{13}^{\times}$.
Therefore, the kernel, which coincides with both $W$ and $K_{W}$,
forms an index 3 free subgroup. By identifying the $15$ generators,
we can apply Stallings' folding process. In any case, we conclude
that $\beta_{3,\,13}$ is faithful. $\qedhere$

\noindent \end{proof}

For any prime number $p$ greater than 13, we can determine whether
$\beta_{3,\,p}$ is faithful in finite time by Theorem 1.1. However,
as $p$ increases, the process of rewriting the images of the generators
of $\mathcal{K}_{p}$ in terms of the generators of $\mathcal{Q}_{p}$
becomes significantly longer. We conjecture that $\beta_{3,\,p}$
is faithful for every $p$, but a generic proof appears to require
a more detailed analysis.

Based on the proof of Theorem 5.18, when $p>3$, we observe that the
order 3 elements in $\mathcal{K}_{p}$ are conjugates of multiplicative
generators. This observation generally applies to any order 3 element
in $\mathcal{Q}_{p}$, by Theorem 3.11, the Kurosh subgroup theorem
for amalgamated free products, and the fact $3$ does not divide $p$.
However, there remain some other nontrivial features, leading us to
propose the following conjecture.

\noindent \begin{conjecture}

Suppose a prime $p$ satisfies $p\equiv1$ (mod 6). Then, the quaternionic
subgroup of level 0, $\mathcal{Q}_{p}\left[0\right]$, contains the
image group $\mu_{p}\left(\mathcal{K}_{p}\right)$. In other words,
there are no additive generators in the rewritten words of the images
of the generators of $\mathcal{K}_{p}$ under $\mu_{p}$. Moreover,
the image group
\begin{align*}
\widehat{\Lambda_{p}}\left[0\right]\left(\mu_{p}\left(\mathcal{K}_{p}\right)\right)
\end{align*}
is isomorphic to $C_{3}$. In particular, the free group $\ker\left(\widehat{\Lambda_{p}}\left[0\right]\right)$
contains the image of the free group $\ker\left(\epsilon_{p}\right)$
under $\mu_{p}$.

\noindent \end{conjecture}

\section{Proof of Theorem 1.2}

For any prime $p$, the symmetric group $S_{3}$ has a faithful representation
to $\mathrm{GL}\left(3,\,\mathbb{F}_{p}\right)$ by the usual permutation
representation. Abusing notation, we identify $S_{3}$ with the representation
in $\mathrm{GL}\left(3,\,\mathbb{F}_{p}\right)$ when a prime $p$
is fixed. Following Salter \citep[Definition 2.5]{MR4228497}, we
formally define the following group.

\noindent \begin{definition}

The \emph{target group} for $\mathbb{F}_{p}$, $\Gamma_{3,\,p}$,
is defined as
\begin{align*}
\Gamma_{3,\,p} & :=\left\{ A\in\mathrm{GL}\left(3,\,\mathbb{F}_{p}\left[t,t^{-1}\right]\right)\::\:vA=v,\:A\overrightarrow{1}=\overrightarrow{1},\:\overline{A}J_{3}A^{T}=J_{3},\:A|_{t=1}\in S_{3}\right\} .
\end{align*}
\end{definition}

Let $\overline{\Gamma_{3,\,p}}$ be the quotient of $\Gamma_{3,\,p}$
by its center. We will also call $\overline{\Gamma_{3,\,p}}$ the
target group, if there is no possibility for confusion. Define the
map $\overline{\beta_{3,\,p}}:B_{3}\to\overline{\Gamma_{3,\,p}}$
to be the composition of the quotient $q_{3,\,p}:\Gamma_{3,\,p}\to\overline{\Gamma_{3,\,p}}$
and the Burau representation modulo $p$ $\beta_{3,\,p}$. Then, we
pose a \emph{Salter-type problem modulo} $p$ on $B_{3}$, following
Salter's question \citep[Question 1.1]{MR4228497}.

\noindent \begin{question}

Is the map $\overline{\beta_{3,\,p}}:B_{3}\to\overline{\Gamma_{3,\,p}}$
surjective?

\noindent \end{question}

When we addressed Salter's original question with the base ring $\mathbb{Z}$
in \citep{lee2024salters}, we observed that the target group is the
same as the formal Burau group we defined, as the condition $A|_{t=1}\in S_{3}$
is redundant \citep[Lemma 3.15]{lee2024salters}. However, over finite
fields, the redundancy may not hold. For example, consider an element
$\overline{g_{11}}\left[3\right]^{2}$ in the quaternionic group $\mathcal{Q}_{11}$.
It has a representative
\begin{align*}
4g_{11}\left[3\right]^{2} & =\left(\begin{array}{cc}
4t^{2}+2t+2+8t^{-1} & t+4+t^{-1}\\
-\left(t+4+t^{-1}\right)\Phi & 8t+2+t^{-1}+4t^{2}
\end{array}\right),
\end{align*}
which has determinant 1 in $\mathbb{F}_{11}$. Since $\left(4g_{11}\left[3\right]^{2}\right)_{11}=1+t$
modulo $1+t+t^{2}$, $\phi_{11}\left(\mathcal{B}_{11}\right)$ includes
$\left(-t\right)4g_{11}\left[3\right]^{2}$ by Lemma 2.10. From (9)
and Definition 2.3, we have a matrix $A\in\mathcal{B}_{11}$ given
by
\begin{align*}
A & =\left(\begin{array}{ccc}
7t^{3}+9t^{2}+2t+1 & 10t^{3}+8t & *\\
t^{2}+8t+10 & 3t^{2}+6t+10 & *\\
* & * & *
\end{array}\right),
\end{align*}
such that $\phi_{11}\left(A\right)=\left(-t\right)4g_{11}\left[3\right]^{2}$.
However, by evaluating $A$ at $t=1$, we have
\begin{align*}
A|_{t=1} & =\left(\begin{array}{ccc}
8 & 7 & 8\\
8 & 8 & 7\\
7 & 8 & 8
\end{array}\right),
\end{align*}
which is not included in the image of the usual permutation representation
of $S_{3}$. Therefore, $\Gamma_{3,\,11}$ is a proper subgroup of
$\mathcal{B}_{11}$.

However, this discrepancy does not pose a major problem in answering
Question 6.2, since $\Gamma_{3,\,11}$ is a finite index subgroup
of $\mathcal{B}_{11}$. The obvious upper bound for the index is
\begin{align}
\left[\mathcal{B}_{p}\::\:\Gamma_{3,\,p}\right] & \le\frac{\left|\mathrm{SL}\left(3,\,p\right)\right|}{3},
\end{align}
although the actual index is likely much smaller. In this paper, we
will not address specific indices or attempt to obtain a more precise
upper bound.

\noindent \begin{lemma}

Suppose a prime $p$ satisfies $p\equiv1$ (mod 6). Then, the quaternionic
subgroup $\mathcal{Q}_{p}$ is not finitely generated.

\noindent \end{lemma}
\begin{proof}

\noindent In the same way as Definition 5.14, but without imposing
restrictions on the degree of polynomials $f\in\mathbb{F}_{p}\left[x\right]$,
define a semidirect product $K_{p}$ to be
\begin{align*}
\left(E_{p^{\infty}}\times E_{p^{\infty}}\right)\rtimes\mathbb{F}_{p}^{\times},
\end{align*}

\noindent and define a map $\widehat{\Lambda_{p}}:\mathcal{Q}_{p}\to K_{p}$.
Then, for any nonnegative integer $k$, the image of each element
$M\in\mathcal{Q}_{p}\left[k\right]$ under $\widehat{\Lambda_{p}}\left[k\right]$
coincides with the image of $M$ under $\widehat{\Lambda_{p}}$.

Suppose that $\mathcal{Q}_{p}$ is finitely generated by a finite
set of generators $g_{1},\cdots,\,g_{n}$. By expressing $g_{1},\cdots,\,g_{n}$
as words in the elementary, affine and stable generators, there must
be an integer $k$ such that $k-1$ is the maximal degree of $f\in\mathbb{F}_{p}\left[x\right]$
among the additive generators $\overline{a_{p,\,\tau,\,u}}\left[f\right]$,
$\overline{a_{p,\,\tau,\,l}}\left[f\right]$ that appear in the words.
Consequently, the group generated by $\widehat{\Lambda_{p}}\left(g_{1}\right),\cdots,\,\widehat{\Lambda_{p}}\left(g_{n}\right)$
is contained in the finite group $K_{p}\left[k\right]$. However,
the set of additive generators has an infinite image under $\widehat{\Lambda_{p}}$,
leading to a contradiction. $\qedhere$

\noindent \end{proof}

From now on, define the group $G_{p}^{1}$ as the intersection of
$\overline{\phi_{p}}\circ\beta_{3,\,p}\left(B_{3}\right)$ and $\mathcal{Q}_{p}^{1}$.

\noindent \begin{corollary}

For the primes $p$ such that $p\equiv1,\,3$ (mod 6), the index $\left[\mathcal{B}_{p}\::\:\beta_{3,\,p}\left(B_{3}\right)\right]$
is infinite.

\noindent \end{corollary}
\begin{proof}

\noindent Suppose that $\left[\mathcal{B}_{p}\::\:\beta_{3,\,p}\left(B_{3}\right)\right]$
is finite. Then, $\left[\overline{\phi_{p}}\left(\mathcal{B}_{p}\right)\::\:\overline{\phi_{p}}\circ\beta_{3,\,p}\left(B_{3}\right)\right]$
is also finite. Recall that we defined groups $I_{p}$ and $I_{p}^{1}$
in Definition 2.14. Then, $\left[I_{p}^{1}\;:\:G_{p}^{1}\right]$
is finite, because the index
\begin{align*}
\left[\overline{\phi_{p}}\left(\mathcal{B}_{p}\right)\::\:I_{p}\right]
\end{align*}

\noindent is exactly $p+1$ from the proof of Lemma 5.6. Because Theorem
2.17 ensures that $\left[\mathcal{Q}_{p}^{1}\::\:I_{p}^{1}\right]$
is finite, the index
\begin{align*}
\left[\mathcal{Q}_{p}^{1}\;:\:G_{p}^{1}\right]
\end{align*}
is also finite. However, since $G_{p}^{1}$ is finitely generated,
it implies that $\mathcal{Q}_{p}$ is finitely generated, contradicting
Lemma 6.3. Therefore, the index $\left[\mathcal{B}_{p}\::\:\beta_{3,\,p}\left(B_{3}\right)\right]$
is infinite. $\qedhere$

\noindent \end{proof}
\begin{prooftheorem12}

\noindent Suppose $p=2$. In the proof of Theorem 5.18, we observed
that $\mathcal{K}_{2}$ is freely generated by $S^{2}$ and $SV_{1}$
which are mapped to $\overline{g_{2}}\left[0\right]^{-1}$ and $\overline{g_{2}}\left[0\right]^{-1}\overline{g_{2}}\left[1\right]$,
respectively under $\mu_{2}$. These elements generate the entire
free group $\mathcal{Q}_{2}$. Therefore, we conclude that $\beta_{3,\,2}\left(B_{3}\right)=\mathcal{B}_{2}$.

Suppose a prime $p$ satisfies the condition $p\equiv1$ or $p\equiv3$
(mod 6). By Corollary 6.4, the index $\left[\mathcal{B}_{p}\::\:\beta_{3,\,p}\left(B_{3}\right)\right]$
is infinite. Therefore, $\left[\Gamma_{3,\,p}\::\:\beta_{3,\,p}\left(B_{3}\right)\right]$
must be infinite by (83), the obvious finite upper bound of $\left[\mathcal{B}_{p}\::\:\Gamma_{3,\,p}\right]$.
By Corollary 2.12, we conclude that the index
\begin{align*}
\left[\overline{\Gamma_{3,\,p}}\::\:\overline{\beta_{3,\,p}}\left(B_{3}\right)\right]
\end{align*}
must be infinite.

Finally, suppose $p\equiv5$ (mod 6). By Theorem 3.11, the quaternionic
group $\mathcal{Q}_{p}$ is free of rank $p$. By Lemma 2.18, the
subgroup $\mathcal{Q}_{p}^{1}$ has index 2 in $\mathcal{Q}_{p}$,
and has rank $2p-1$. By Theorem 2.17, since the subgroup $I_{p}^{1}$
has a finite index in $\mathcal{Q}_{p}^{1}$, the free rank of $I_{p}^{1}$
is not smaller than $2p-1$ by the Schreier index formula.

On the other hand, consider the group $G_{p}^{1}$. It is a subgroup
of $I_{p}^{1}$, and is also a free group. However, by Lemma 5.9,
since the free group $\mathcal{K}_{p}$ has rank $\frac{p+4}{3}$,
the rank of $G_{p}^{1}$ is at most $\frac{p+4}{3}$. Given this,
the index
\begin{align*}
\left[I_{p}^{1}\::\:G_{p}^{1}\right]
\end{align*}
is infinite, since in a finitely generated free group, a free subgroup
of smaller rank cannot have finite index by the Schreier index formula.

Therefore, $\left[\overline{\phi_{p}}\left(\mathcal{B}_{p}\right)\::\:\overline{\phi_{p}}\circ\beta_{3,\,p}\left(B_{3}\right)\right]$
is also infinite. Using the obvious upper bound (83), we conclude
that $\left[\overline{\Gamma_{3,\,p}}\::\:\overline{\phi_{p}}\circ\beta_{3,\,p}\left(B_{3}\right)\right]$
is infinite. $\qed$

\noindent \end{prooftheorem12}

$ $

\begin{spacing}{0.9}
\bibliographystyle{amsplain}
\phantomsection\addcontentsline{toc}{section}{\refname}\bibliography{bibgen}

\end{spacing}

$ $

{\small{}Donsung Lee; \href{mailto:disturin@snu.ac.kr}{disturin@snu.ac.kr}}{\small\par}

{\small{}Department of Mathematical Sciences and Research Institute
of Mathematics,}{\small\par}

{\small{}Seoul National University, Gwanak-ro 1, Gwankak-gu, Seoul,
South Korea 08826}{\small\par}

\clearpage{}

\pagebreak{}

\pagenumbering{arabic}

\renewcommand{\thefootnote}{A\arabic{footnote}}
\renewcommand{\thepage}{A\arabic{page}}
\renewcommand{\thetable}{A\arabic{table}}
\renewcommand{\thefigure}{A\arabic{figure}}

\setcounter{footnote}{0} 
\setcounter{section}{0}
\setcounter{table}{0}
\setcounter{figure}{0}
\end{document}